\documentclass[reqno]{amsart}

\usepackage{color}
\usepackage{mathrsfs}
\usepackage{amscd}
\usepackage{amsmath}
\usepackage{latexsym}
\usepackage{amsfonts}

\usepackage{amssymb}
\usepackage{amsthm}
\usepackage{graphicx}
\usepackage{makecell}
\usepackage{array,color}
\usepackage{booktabs}
\usepackage{multirow}
\usepackage{bm}
\usepackage{verbatim}
\usepackage{tabularx}
\usepackage{array}
\usepackage{float}
\usepackage{longtable}
\usepackage{pdflscape}
\usepackage{ragged2e}
\newcolumntype{Y}{>{\RaggedRight\arraybackslash}X}

\newcolumntype{C}[1]{%
	>{\centering\arraybackslash}m{#1}}

\usepackage{hyperref}

\newtheorem{theorem}{Theorem}[section]
\newtheorem{proposition}[theorem]{Proposition}

\newtheorem{lemma}[theorem]{Lemma}
\newtheorem{definition}[theorem]{Definition}

\newtheorem{remark}[theorem]{Remark}

\newcommand\R{\mathbb{R}}

\newcommand{\medoplus}
{\mathbin{\scalebox{1.2}{$\oplus$}}}

\def\Id{\mathop{\mathrm{Id}}\nolimits}
\def\<{{\langle}}
\def\>{{\rangle}}

\numberwithin{equation}{section}

\title[Decay estimates for Beam equations with potentials in dimension two]
{Decay Estimates for the two-dimensional Beam Equation with Potentials}

\author{Shuangshuang Chen, Han Cheng, Zijun Wan and Xiaohua Yao\textsuperscript{\dag}}

\address{Shuangshuang Chen, Department of Mathematics, Central China Normal University, Wuhan, 430079, P.R. China}
\email{chenss@mails.ccnu.edu.cn}

\address {Han Cheng, Institute of Applied Physics and Computational Mathematics, Beijing, 100088, P.R. China}
\email{chmathh@163.com}

\address{Zijun Wan, Institute of Applied Physics and Computational Mathematics, Beijing, 100088, P.R. China}
\email{zijunwan@mails.ccnu.edu.cn}

\address{Xiaohua Yao, Department of Mathematics and Key Laboratory of Nonlinear Analysis and Applications (Ministry of Education), Central China Normal University, Wuhan, 430079, P.R. China}
\email{yaoxiaohua@ccnu.edu.cn}

\date{\today}

\begin{document}

\begin{abstract}

This paper establishes the  sharp time-decay estimates for the following
two-dimensional beam  (plate)  equation with a real-valued
decaying potential $V$:
\[
\partial_t^2u+(\Delta^2+V)u=0,
\qquad u(0,x)=f(x),\quad \partial_tu(0,x)=g(x),
\]
When zero is \textit{a regular point or a first-kind resonance} of
$H=\Delta^2+V$, we  prove the sharp $L^1\to L^\infty$ estimates for the solution operators:
$$
\left\|\cos(t\sqrt H)P_{\mathrm{ac}}(H)\right\|_{L^1\to L^\infty}
+\bigg\|\frac{\sin(t\sqrt H)}{t\sqrt H}P_{\mathrm{ac}}(H)\bigg\|_{L^1\to L^\infty}
\lesssim |t|^{-1},
$$
and obtain an enhanced decay $(|t|\log|t|)^{-1}$ in logarithmically weighted spaces $L^1_\omega\to L^\infty_{-\omega}$ with $\omega(x)=\log(2+|x|)$.
This reveals that the potential $V$ can enhance dispersive decay by changing the zero-energy structure, endowing the    sine evolution with a logarithmic gain $(\log|t|)^{-1}$ unattainable in the free case.

For \textit{second-kind resonances}, 
    the unweighted estimate above deteriorates to the sharp rate
$
|t|^{-1}(\log|t|)^2
$
exactly when  one resonance
state $\phi_1$ has the trace moment
\(\langle |x|^2V,\phi_1\rangle\neq 0\), and another nonzero
state $\phi_2$ makes all the second-order moments
\(\langle x_ix_jV,\phi_2\rangle=0\).  If either
 is absent, the unweighted decay returns to 
\(|t|^{-1}\). 
In this regime, the weighted cosine reaches $(|t|\log|t|)^{-1}$ and  the weighted  sine derives the same
bound  when no nonzero bounded resonance state is present;
otherwise it retains the free $|t|^{-1}$ decay.

For \textit{third-kind resonances or zero eigenvalue}, a \(d\)-wave
 gives the strongest threshold obstruction, reducing  both propagators to the logarithmic decay
\((\log|t|)^{-1}\).  For the pure  eigenvalue case (i.e., neither $d$- nor $p$-wave resonance),  both propagators recover the optimal unweighted  rate $|t|^{-1}.$

	These results provide a nearly complete picture of the time-decay behavior for the two-dimensional plate equation with a potential, covering all possible threshold singularities at zero energy.

\end{abstract}

	\maketitle
	\enlargethispage{2pt}
	\tableofcontents
	\section{Introduction and main results}
	\subsection{Introduction}	In this paper, we investigate time decay estimates for the following
two-dimensional beam (or plate) equation with a real-valued potential
$V(x)$:
\begin{equation}\label{beam_equation}
\begin{cases}
\partial_t^2u+\Delta^2u+V(x)u=0, \qquad x\in\mathbb R^2,\\
u(0,x)=f(x), \qquad \partial_tu(0,x)=g(x).
\end{cases}
\end{equation}
The two-dimensional beam (or plate) equation is a fundamental model
originating from classical elasticity theory, dating back to
Kirchhoff and Love, and describes the transverse vibrations of thin
elastic plates, see, e.g., 
\cite{Kirchhoff1850,Love1927,TimoshenkoWoinowskyKrieger1959}. In the
unperturbed case, the dynamics are governed by the biharmonic
operator $\Delta^2$, which reflects the intrinsic bending stiffness
of the material. In realistic contexts, spatially decaying potential
perturbations $V(x)$ are typically used as a simplified model for
structural heterogeneities, localized defects, or material
impurities.

Let $V(x)$ satisfy $ |V(x)| \lesssim \langle x \rangle^{-\mu}$ for some $\mu > 0$. By the Kato-Rellich theorem, $H=\Delta^2+V$ is self-adjoint on $L^2(\mathbb{R}^2)$ with domain $H^4(\mathbb{R}^2)$. The solution to the beam equation \eqref{beam_equation} can be explicitly expressed as
\[
	u(t, x) = \cos (t \sqrt{H }) f(x) + \frac{\sin (t \sqrt{H })}{\sqrt{H}} g(x).
\]
Since the negative eigenvalues of $H$ can lead to exponentially growing modes, we will restrict our attention to the evolution projected onto the absolutely continuous spectrum. That is, we focus on the continuous part of the solution:
\[
	u_c(t, x) = \cos (t \sqrt{H })P_{\mathrm{ac}}(H) f(x) + \frac{\sin (t \sqrt{H })}{\sqrt{H}} P_{\mathrm{ac}}(H) g(x),
\]
where $P_{\mathrm{ac}}(H)$ is the orthogonal projection onto the
absolutely continuous spectral subspace of $H$. Under the assumption that $H$ has no embedded positive eigenvalues (see Subsection \ref{eigenvalue} below for a sufficient condition), there exist several interesting works about the decay estimates of the continuous part $u_{c}(t, x)$.
 In dimension $n=3$, Chen et al. \cite{CLSY_ArXiv} proved that the $L^1 \to L^\infty$ decay rates for the solution operators are ${O}(|t|^{-3/2})$ when zero is a regular point or a first-kind resonance, while the decay rate changes when other types of zero-energy resonances occur. In dimension $n=1$, Chen-Wan-Yao \cite{CWY} established an $O(|t|^{-1/2})$ decay rate, which remains unchanged  in the presence of resonances. For higher dimensions ($n\geq5$), the decay rate $O(|t|^{-n/2})$ in the regular case can be deduced using wave operator methods developed by Erdo\u{g}an and Green \cite{EG21,Erdogan-Green23}.

In even dimensions, and particularly in dimension two, logarithmic terms in the resolvent expansion complicate the analysis at zero energy. For the Schr\"odinger operator $-\Delta+V$ in two dimensions, Schlag \cite{Schlag-2005} first established $L^1\to L^\infty$ estimates when zero is regular. Later, Erdo\u{g}an and Green \cite{Erdogan-Green} systematically treated zero-energy resonances and obtained improved decay estimates in log-weighted spaces for the regular case \cite{EG-2013}.
For the beam equation, the relevant operator is $H=\Delta^2+V$. The bi-Laplacian $\Delta^2$ exhibits high-order degeneracy at zero energy. This degeneracy, combined with the inherent logarithmic singularities in two dimensions, makes the dispersive analysis of $\Delta^2+V$ a challenging task.

In the free case (i.e., $V=0$), this  simplifies to $\sqrt{H} = -\Delta$, and
it is known that the following optimal decay estimates  in $\mathbb{R}^2$ hold (see Proposition \ref{free_case}):
   \[
\|\cos(t\Delta)\|_{L^1\to L^\infty} + \Big\|\frac{\sin(t\Delta)}{t\Delta}\Big\|_{L^1\to L^\infty} \sim {|t|}^{-1}.  
    \]
More interestingly, in the weighted space setting $L^1_\omega (\R^2)\to L^{\infty}_{-\omega}(\R^2)$ with $\omega(x) = \log(2+|x|)$, the decay rate of cosine evolution operator can be improved: 
 \begin{align} \label{free_weighted2} \|\cos(t\Delta)\|_{L^1_\omega\to L^\infty_{-\omega}} \sim ({|t|\log |t|})^{-1},  \quad |t|\gg1,    \end{align} 
which is strictly unattainable for both $e^{it\Delta}$ and $\sin(t\Delta)/(t\Delta)$ in the  weighted space. This enhanced phenomenon \eqref{free_weighted2} exactly happens in these dimensions $n \equiv 2 \pmod{4}$ (see Remark \ref{free_cos_higher_dim}). 

For a decaying potential $V$, the long-time bounds depend on the
low-energy resolvent of $H$. 
The potential may alter the zero-energy structure
and the low-energy behavior of the absolutely continuous spectral
measure.
As described in Definition \ref{definition1}, zero can be a regular point, one of three types of resonances, or an eigenvalue. These different threshold configurations lead to distinct dispersive behaviors.

When zero is regular or a first-kind resonance of $H$, we obtain the following $L^1\to L^\infty$ bounds, which match the free decay rate (see Theorem \ref{main_theorem-1}):
 \begin{align*}
    \left\|\cos (t \sqrt{H}) P_{\mathrm{ac}}(H)\right\|_{L^1 \to L^{\infty}} +
    \bigg\|\frac{\sin (t \sqrt{H})}{t\sqrt{H}} P_{\mathrm{ac}}(H)\bigg\|_{L^1 \to L^{\infty}} \lesssim {|t|}^{-1},
    \end{align*}
In the   weighted spaces $L^1_\omega\to L^\infty_{-\omega}$, we obtain
\begin{equation}\label{the first and regular weigthed estimate}
\left\|\cos(t\sqrt H)P_{\mathrm{ac}}(H)\right\|_{L^1_\omega\to L^\infty_{-\omega}}
+
\left\|\frac{\sin(t\sqrt H)}{t\sqrt H}P_{\mathrm{ac}}(H)\right\|_{L^1_\omega\to L^\infty_{-\omega}}
\lesssim \frac{1}{|t|\log|t|},
\qquad |t|\gg1.
\end{equation}
Compared with the free case, the weighted sine bound gains a factor
$(\log|t|)^{-1}$. This gain reflects the effect of $V$ on the
zero-energy resonance structure.

The second-kind resonance case exhibits a more subtle threshold
phenomenon.  Although the free bi-Laplacian $\Delta^2$ itself has a
second-kind resonance at zero, the perturbed operator may either
retain the free decay rate or develop a sharp logarithmic loss. To
describe the criterion, define
\[
\mathfrak M_2(\phi)
=
\bigl(
\langle x_ix_jV,\phi\rangle
\bigr)_{i,j=1}^2,
\qquad
\tau(\phi)
=
\operatorname{tr}\mathfrak M_2(\phi)
=
\langle |x|^2V,\phi\rangle,
\]
where \(\phi\) belongs to the second-kind resonance space of \(H\).
Theorem \ref{main_theorem-2} gives an unweighted
\(L^1\to L^\infty\) bound of order $|t|^{-1}(\log|t|)^2$
when the
resonance space contains a state \(\phi_1\) with
$
\tau(\phi_1)\neq0,
$
and also a nonzero state \(\phi_2\) with
$
\mathfrak M_2(\phi_2)=0.
$ Thus the logarithmic loss requires both moment
conditions. If either condition fails, both unweighted bounds recover 
to $|t|^{-1}$.
The weighted cosine reaches $(|t|\log|t|)^{-1}$ and  the weighted  sine derives the same
bound  when no nonzero bounded resonance state is present (i.e., $\ker \tau=\{0\}$);
otherwise it retains the free $|t|^{-1}$ decay.

Finally, we consider third-kind resonances and zero eigenvalues.
Theorem \ref{main_theorem-31} shows that the presence of a
\(d\)-wave resonance gives the strongest threshold obstruction:
\begin{equation*}  
    \left\|\cos (t \sqrt{H}) P_{\mathrm{ac}}(H)\right\|_{L^1 \to L^{\infty}} +    
    \bigg\|\frac{\sin (t \sqrt{H})}{t\sqrt{H}} P_{\mathrm{ac}}(H)\bigg\|_{L^1 \to L^{\infty}} \sim ({\log |t|})^{-1},\quad |t|\gg1.
    \end{equation*}
Moreover, we  obtain several improved  estimates in the absence of a $d$-wave. In particular, in the pure $L^2$ eigenvalue case (i.e., neither $d$-wave nor $p$-wave), we not only recover the unweighted $|t|^{-1}$ bound,  but also
obtain the enhanced  bound $(|t|\log|t|)^{-1}$ as in \eqref{the first and regular weigthed estimate} if the fourth-order
tensor
$
T_{ijkl}(\phi)
=
\langle x_ix_jx_kx_lV,\phi\rangle
$
is isotropic,  namely,  there exists a scalar \(\gamma_\phi\) such that
$
T_{ijkl}(\phi)
=
\gamma_\phi
\bigl(
\delta_{ij}\delta_{kl}
+
\delta_{ik}\delta_{jl}
+
\delta_{il}\delta_{jk}
\bigr).
$

These results cover all zero-energy types for the two-dimensional plate equation with a decaying potential. They identify the moment conditions governing logarithmic losses and gains in the long-time decay estimates.

\subsection{Some notations}\label{Notations}
In this subsection, we collect some notations used throughout the paper.
	\begin{itemize}

\item For $a,b\in\mathbb{R}^+$, $a\lesssim b$ (resp. $a\gtrsim b$) means $a\le cb$ (resp. $a\ge cb$) for some $c>0$. In particular, if $b\lesssim a\lesssim b$, we write $a\sim b$. Moreover, for $a\in \mathbb{R}$,   $a\pm$ denote $a\pm \epsilon$   for any arbitrarily small  $\epsilon >0$.
        \vskip0.2cm
		\item For a positive function $g$, we write $A(\lambda)=O_k(g(\lambda))$ if
    \[
        \|\partial_\lambda^\ell A(\lambda)\|
        \le C_\ell\lambda^{-\ell}g(\lambda),\qquad 0\le\ell\le k.
    \]
    Here the norm is absolute value for scalar functions and the stated
    operator norm for operator-valued functions.
   \vskip 0.2cm
   \item The notation $\mathcal{O}(f(t))$ refers to an operator satisfying $\|\mathcal{O}(f(t))\|_{L^1_\omega \rightarrow L^{\infty}_{-\omega}} \lesssim |f(t)|.$
   \vskip 0.2cm
   \item  Let \(\mathbb{B}(L^2)\) denote the space of bounded linear operators on \(L^2\). For \(A\in\mathbb{B}(L^2)\), we write \(A^*\) for its adjoint.
        \vskip0.2cm
		
		\item Let $\chi_1 \in C_c^{\infty}(\mathbb{R})$ be even, with $0\le\chi_1\le1$, such that $\chi_1(\lambda)=1$ for $|\lambda|<\lambda_0\ll1$ and $\chi_1(\lambda)=0$ if $|\lambda|>2\lambda_0$, where $0<2\lambda_0<1$ is chosen sufficiently small for the low-energy expansion of $\left( M^{\pm}(\lambda)\right)^{-1}$ in Theorem \ref{thm:M_inverse}. Define $\chi_2(\lambda):=1-\chi_1(\lambda)$ and $\widetilde{\chi}_j(\lambda):=\chi_j(\lambda^4)$  for  $j=1,2.$
\vskip0.2cm
\item
    $ \langle f, g \rangle := \int_{\mathbb{R}^2} f(x) \overline{g(x)} \, dx$ denotes the inner product  or the dual pair.
	For $x\in\mathbb{R}^n$, write $\langle x \rangle := 1 + | x |$.
		For $s \in \mathbb{R}$, we  define the weighted $L^2(\mathbb{R}^2) $ spaces:
		$$L^{2,s}(\mathbb{R}^2) := \{ f \in L_{\text {loc}}^2(\mathbb{R}^2) \mid  \langle \cdot \rangle^s f \in L^2(\mathbb{R}^2) \} .$$
\end{itemize}
	\subsection{Main results}\label{main result}
  In the following, we first give the precise definition of zero resonance types of $H=\Delta^2+V(x)$ on $\mathbb{R}^2$.
For $\sigma \in \mathbb{R}$, let $W_{-\sigma}\left(\mathbb{R}^2\right)$ denote the intersection space
$$
W_{-\sigma}\left(\mathbb{R}^2\right)=\bigcap_{s>\sigma} L^{2,-s}\left(\mathbb{R}^2\right) ,
$$
which is increasing in $\sigma\in\R$ and satisfies $L^{2,-\sigma}(\mathbb{R}^2) \subset W_{-\sigma}(\mathbb{R}^2)$. In particular, $W_0\left(\mathbb{R}^2\right) \supset L^2\left(\mathbb{R}^2\right)$.
\begin{definition}\label{definition1}
{\rm Let $H = \Delta^2 + V(x)$ with $|V(x)| \lesssim \langle x \rangle^{-\mu}$ for some $\mu > 0$. We say that
\begin{itemize}
\item[(i)] Zero is a {\it first kind resonance}  of $H$ if the equation $H\phi = 0$  has only zero solution  in $W_{-1}(\mathbb{R}^2)$, but has a nonzero solution  $\phi \in W_{-2}(\mathbb{R}^2)$ and there exists a polynomial $\sum_{|\alpha|=1} C_\alpha x^\alpha$ such that $\phi-\sum_{|\alpha|=1} C_\alpha x^\alpha\in W_{-1}(\mathbb{R}^2)$.
\vskip0.1cm

\item[(ii)] Zero is a {\it second kind resonance} of $H$ if
the equation $H\phi = 0$ has a nonzero solution  in $W_{-1}(\mathbb{R}^2)$, but only has a zero solution  $\phi$ in $W_{0}(\mathbb{R}^2)$.
\vskip0.1cm

\item[(iii)] Zero is a {\it third kind resonance} of $H$ if the equation $H\phi = 0$ has a nonzero solution  in $W_0(\mathbb{R}^2)$, but there exists only a zero solution  $\phi$ in  $L^2(\mathbb{R}^2)$.

\vskip0.1cm
\item[(iv)] Zero is an \emph{eigenvalue}  of $H$ if there exists a nonzero $\phi \in L^2(\mathbb{R}^2)$ satisfying $H\phi = 0$.
\vskip0.1cm

\item[(v)] Zero is a {\it regular point} of $H$ if $H$ has neither zero eigenvalues nor zero resonances.
		\end{itemize}}
	\end{definition}
For convenience, we use $\mathbf{k} \in \{0,1,2,3,4\}$ to denote zero energy types: $\mathbf{k}=0$ (regular), $\mathbf{k}=1,2,3$ (resonances), $\mathbf{k}=4$ (eigenvalue).
Moreover, two equivalent characterizations of zero resonances for $H$ will be provided in Proposition~\ref{characterizations} and  Proposition~\ref{characterizations Q} in Section~\ref{subsec:Q}.

 We now state our main results, which are organized into the following four theorems  based on the nature of the zero-energy resonances. To begin, we define the logarithmic weight function $\omega(x) = \log(2+|x|)$ and the associated weighted spaces
\[
L^p_{\pm \omega} := \bigl\{ f \in L_{\mathrm{loc}}^p(\mathbb{R}^2) \bigm| \omega(x)^{\pm 1} f \in L^p(\mathbb{R}^2) \bigr\}, \qquad p = 1, \infty.
\]

\begin{theorem}\label{main_theorem-1}
Let $H = \Delta^2 + V$ and $|V(x)| \lesssim \langle x \rangle^{-11-}$.
Assume that $H$ has no positive embedded eigenvalues and that zero is \textbf{a regular point or a first-kind resonance} of $H$.
Then
\begin{equation}\label{regular_first}
    \left\|\cos (t \sqrt{H}) P_{\mathrm{ac}}(H)\right\|_{L^1 \to L^{\infty}} +
    \bigg\|\frac{\sin (t \sqrt{H})}{t\sqrt{H}} P_{\mathrm{ac}}(H)\bigg\|_{L^1 \to L^{\infty}} \lesssim |t|^{-1},
\end{equation}
where $P_{\mathrm{ac}}(H)$ denotes the projection onto the absolutely continuous spectrum space of $H$. Moreover, in the weighted space setting $L^1_\omega \to L^{\infty}_{-\omega}$, we can obtain the improved decay estimate:
\begin{equation}\label{regular_first_weight}
    \left\|\cos (t \sqrt{H}) P_{\mathrm{ac}}(H)\right\|_{L^1_\omega \to L^{\infty}_{-\omega}} +
    \bigg\|\frac{\sin (t \sqrt{H})}{t\sqrt{H}} P_{\mathrm{ac}}(H)\bigg\|_{L^1_\omega \to L^{\infty}_{-\omega}} \lesssim \frac{1}{|t| \log|t|}, \quad |t|\gg1.
\end{equation}
\end{theorem}

\begin{remark}
{\rm
We provide several comments regarding the decay estimates \eqref{regular_first} and \eqref{regular_first_weight}:
\begin{itemize}
    \item
    For the unweighted estimate \eqref{regular_first}, the $|t|^{-1}$ decay rate exactly matches the intrinsic dispersive rate of the free operator in $\mathbb{R}^2$ (see Proposition \ref{free_case}):
    \[
\|\cos(t\Delta)\|_{L^1\to L^\infty} + \Big\|\frac{\sin(t\Delta)}{t\Delta}\Big\|_{L^1\to L^\infty} \sim |t|^{-1}.
    \]
    It suggests that the estimate \eqref{regular_first} is  optimal in the unweighted setting $L^1\to L^\infty$ (see Remark \ref{rmk:unweighted-sharpness}).
    \item
    For the weighted estimate \eqref{regular_first_weight}, it is noteworthy that the perturbed sine evolution $\sin (t \sqrt{H})/(t\sqrt{H})$ exhibits an extra logarithmic decay factor of $(\log|t|)^{-1}$ compared to the following free weighted estimates (see Proposition \ref{free_case}):
       \[
       \Big\|\frac{\sin(t\Delta)}{t\Delta}\Big\|_{L^1_\omega\to L^\infty_{-\omega}} \sim |t|^{-1}.
       \]
\item  Moreover, we point out that the decay rate $\mu$ of the potential $V$   in Theorem~\ref{main_theorem-1}  is determined by the types of zero energy. The decay requirement guarantees the
	asymptotic expansion  in Theorem~\ref{thm:M_inverse}.	In contrast, the high-energy analysis requires only the weaker decay
	condition $\mu>4$ (see Theorem~\ref{main_theorem_high}). Hence the decay assumption on $V$ may not be optimal,
even when zero is regular.
    \end{itemize}}
\end{remark}

For the \textbf{second-kind resonance case}, define the zero-energy
solution space in $W_{-1}(\mathbb R^2)$ by
\[
\mathcal{E} = \big\{ \phi \in W_{-1}(\mathbb{R}^2) \ \big| \ H\phi = 0 \big\}.
\]
When zero is a second-kind resonance of $H$, Definition \ref{definition1} implies that $\mathcal{E} \neq \{0\}$ and  $\mathcal{E} \cap W_0(\mathbb{R}^2) = \{0\}$.
To classify the dispersive behavior based on the second-order potential moments on $\mathcal{E},$
we define 
\begin{equation}\label{eq:second-moment-map}
	\mathfrak{M}_2:
	\mathcal{E}
	\longrightarrow
	\operatorname{Sym}_2(\mathbb{C}),
	\qquad
	\mathfrak{M}_2(\phi)
	=
	\left(
	\langle x_ix_jV,\phi\rangle
	\right)_{i,j=1}^2,
\end{equation}
and its trace
\begin{equation}\label{eq:radial-second-moment}
	\tau(\phi)
	:=
	\operatorname{tr}\mathfrak{M}_2(\phi)
	=
	\langle |x|^2V,\phi\rangle.
\end{equation}
The trace moment characterizes bounded resonance states.
Every $\phi\in\mathcal E$
is locally bounded and satisfies 
\[
\phi(x)
=
-\frac{\overline{\tau(\phi)}}{8\pi}\log|x|+O(1),
\qquad |x|\to\infty,
\] 
see \eqref{eq:second-far-field}. Hence
\[
\ker\tau=\mathcal E\cap L^\infty(\mathbb R^2).
\]
Thus $\ker\tau=\{0\}$ means that $\mathcal E$ contains no
nonzero bounded resonance state.

The condition $\tau\not\equiv0$ on $\mathcal E$ means that
the resonance space contains a state with nonzero trace second
moment, whereas $\ker\mathfrak M_2\ne\{0\}$ means that it contains
a nonzero state whose second-order moments all vanish.
Theorem~\ref{main_theorem-2} shows that a logarithmic loss occurs
only when these two types of states coexist.

\begin{theorem}\label{main_theorem-2}
Let $H=\Delta^2+V$ with $|V(x)|\lesssim\langle x\rangle^{-14-}$.
Assume that $H$ has no positive embedded eigenvalues and that zero is a
second-kind resonance. Let $\mathfrak M_2$ and $\tau$ be defined by
\eqref{eq:second-moment-map} and \eqref{eq:radial-second-moment}.
\begin{itemize}
\item[(i)] Assume that $\tau\not\equiv0$ on $\mathcal E$ and
$\ker\mathfrak M_2\ne\{0\}$. Then
\begin{align}\label{log-cos-sin}
&\left\|\cos(t\sqrt H)P_{\mathrm{ac}}(H)\right\|_{L^1\to L^\infty}
+\bigg\|\frac{\sin(t\sqrt H)}{t\sqrt H}P_{\mathrm{ac}}(H)\bigg\|_{L^1\to L^\infty}
\lesssim\frac{(\log(2+|t|))^2}{|t|},
\end{align}
 and in the weighted setting $L^1_\omega\to L^\infty_{-\omega},$ for $|t|\gg 1,$
            \begin{align*}
		\left\|
			\cos(t\sqrt H)P_{\mathrm{ac}}(H)
			\right\|_{L^1_\omega\to L^\infty_{-\omega}}\lesssim \frac{\log|t|}{|t|}\ \ \text{and} \ \ \bigg\|
		\frac{\sin(t\sqrt H)}{t\sqrt H}
			P_{\mathrm{ac}}(H)
			\bigg\|_{L^1_\omega\to L^\infty_{-\omega}}
			\sim
			\frac{(\log|t|)^2}{|t|}.
		\end{align*}
\item[(ii)] Assume that  $\tau\equiv0$ on $\mathcal E$ or
$\ker\mathfrak M_2=\{0\}$. Then, 
\begin{align}\label{2-best}
&\left\|\cos(t\sqrt H)P_{\mathrm{ac}}(H)\right\|_{L^1\to L^\infty}
+\bigg\|\frac{\sin(t\sqrt H)}{t\sqrt H}P_{\mathrm{ac}}(H)\bigg\|_{L^1\to L^\infty}
\lesssim |t|^{-1}.
\end{align}
In the weighted setting $L^1_\omega\to L^\infty_{-\omega}$, for $|t|\gg1$,
\begin{align}\label{2nd-wei}
\left\|\cos(t\sqrt H)P_{\mathrm{ac}}(H)\right\|_{L^1_\omega\to L^\infty_{-\omega}}
\lesssim\frac{1}{|t|\log|t|}.
\end{align}
If $\ker\tau=\{0\}$, estimate \eqref{2nd-wei} also holds with
$\cos(t\sqrt H)$ replaced by $\sin(t\sqrt H)/(t\sqrt H)$.
If $\ker\tau\ne\{0\}$, 
\[
\bigg\|\frac{\sin(t\sqrt H)}{t\sqrt H}P_{\mathrm{ac}}(H)
\bigg\|_{L^1_\omega\to L^\infty_{-\omega}}
\sim |t|^{-1}.
\]
\end{itemize}
\end{theorem}

\begin{remark}
{\rm
The logarithmic loss in Theorem~\ref{main_theorem-2} (i) occurs
only when the resonance space $\mathcal E$ contains both a state $\phi_1$
with nonzero trace moment and a nonzero state $\phi_2$
with all quadratic moments zero.
	In particular, this  happens only when
	$
	\dim\mathcal{E}\geq2.
	$

In Theorem~\ref{main_theorem-2} (ii), the weighted sine bound improves from $|t|^{-1}$
to $(|t|\log|t|)^{-1}$ when $\ker\tau=\{0\}$ (i.e., $\mathcal E$ contains no
nonzero bounded resonance state).
This does not contradict the free case.
For $V=0$, $\mathcal E=\operatorname{span}\{1\}$ and $\tau\equiv0$,
so $\ker\tau=\mathcal E\ne\{0\}$.
The free $|t|^{-1}$ decay (Proposition~\ref{free_case})
thus agrees with the $\ker\tau\ne\{0\}$ case of Theorem~\ref{main_theorem-2} (ii).
}
\end{remark}


We next consider a \textbf{third-kind resonance or a zero eigenvalue}.
By Definition \ref{definition1}, this is equivalent to
$$ \mathcal{E}_1 := \big\{ \phi \in W_{0}(\mathbb{R}^2) \ \big| \ H\phi = 0 \big\} \neq \{0\},$$
which means there exists at least one nonzero solution $\phi\in W_{0}(\mathbb{R}^2)$ to $H\phi=0.$
Below we split our results into Theorems \ref{main_theorem-31} and \ref{main_theorem-32} by the spatial decay of the zero-energy states of $\mathcal{E}_1$:
\begin{itemize}
    \item \textbf{Presence of a $d$-wave resonance:}  $\mathcal{E}_1$ is not fully contained in $L^2$, meaning there exists a nonzero solution $\phi \in W_{0}(\mathbb{R}^2) \setminus L^2(\mathbb{R}^2)$ to $H\phi = 0$.
    \vskip0.1cm
    \item \textbf{Absence of a $d$-wave resonance:} Zero is an eigenvalue and every $\phi\in\mathcal E_1$ belongs to
$L^2(\mathbb R^2)$.
\end{itemize}

\begin{theorem}\label{main_theorem-31}
Let $H = \Delta^2 + V$ with $|V(x)| \lesssim \langle x \rangle^{-18-}$.
Assume that $H$ has no positive embedded eigenvalues and that zero is either a third-kind resonance or an eigenvalue accompanied by a d-wave resonance. Then
    \[
        \left\|\cos (t \sqrt{H}) P_{\mathrm{ac}}(H)\right\|_{L^1 \to L^{\infty}} +
        \bigg\|\frac{\sin (t \sqrt{H})}{t\sqrt{H}} P_{\mathrm{ac}}(H)\bigg\|_{L^1 \to L^{\infty}} \sim \frac{1}{\log |t|}, \quad |t| \gg 1.
    \]
\end{theorem}

We next assume that zero is an eigenvalue but that the $d$-wave
resonance is absent. In this case
\begin{align}\label{E--2}
\mathcal{E}_1
=
\mathcal{E}_2,
\qquad
\mathcal{E}_2
:=
\left\{
\phi\in L^2(\mathbb{R}^2)
\mid
H\phi=0
\right\},
\end{align}
namely, all threshold states in $W_0(\mathbb{R}^2)$ are
$L^2$-eigenfunctions.

In the absence of a $d$-wave resonance, a possible $p$-wave  is the remaining source of logarithmic loss in the unweighted estimates. If no $p$-wave  occurs, the decay  is determined by the fourth-order potential moments of the zero eigenfunctions.
To describe this case, define
\begin{equation}\label{eq:E3}
	\mathcal{E}_3
	:=
	\left\{
	\phi\in\mathcal{E}_2
	\ \middle|\
	\begin{aligned}
		&
		\langle x_1^3x_2V,\phi\rangle
		=
		\langle x_1x_2^3V,\phi\rangle
		=
		0,
		\\
		&
		\langle x_1^4V,\phi\rangle
		=
		\langle x_2^4V,\phi\rangle
		=
		3\langle x_1^2x_2^2V,\phi\rangle
	\end{aligned}
	\right\}.
\end{equation}
Equivalently, $\mathcal{E}_3$ consists of those zero eigenfunctions
whose fourth-order potential moment tensor
$
T_{ijkl}(\phi)
:=
\langle x_ix_jx_kx_lV,\phi\rangle
$
is isotropic, i.e. for $\phi\in\mathcal{E}_3$ there
exists a scalar $\gamma_\phi\in\mathbb C$ such that
\[
T_{ijkl}(\phi)
=
\gamma_\phi
\left(
\delta_{ij}\delta_{kl}
+
\delta_{ik}\delta_{jl}
+
\delta_{il}\delta_{jk}
\right),
\qquad
 i,j,k,l=1,2.
\]
Indeed,  this tensorial condition is equivalent to
$
T_{1112}=T_{1222}=0,$ and
$
T_{1111}=T_{2222}=3T_{1122},
$
which is exactly the set of moment conditions in
\eqref{eq:E3}. Thus
$
\mathcal{E}_2=\mathcal{E}_3
$
means that every zero eigenfunction has isotropic fourth-order moments tensor form.

\begin{theorem}\label{main_theorem-32}

Let $H = \Delta^2 + V$ with $|V(x)| \lesssim \langle x \rangle^{-22-}$.
	Assume that $H$ has no positive embedded eigenvalues and that zero is an eigenvalue of $H$ but has no $d$-wave resonance $(\text{i.e.,}\  \mathcal{E}_1 = \mathcal{E}_2 )$.
	\begin{enumerate}
		\item If $H$ admits a p-wave resonance $($i.e., there exists a nonzero $\phi \in W_{-1}(\mathbb{R}^2) \setminus W_0(\mathbb{R}^2)$ such that $H\phi = 0 )$, then the $L^1 \to L^{\infty}$ estimate \eqref{log-cos-sin}  holds.
		\vskip0.2cm
		\item If $H$ lacks a p-wave resonance, then  the $L^1 \to L^{\infty}$ estimate \eqref{regular_first} holds.
		Moreover, we can  further distinguish the following two subcases:
		\begin{itemize}
			\item If $\mathcal{E}_2\neq \mathcal{E}_3$, then we obtain  the same weighted estimate  as in the free case, namely,
\[
		\left \|\cos (t \sqrt{H}) P_{\mathrm{ac}}(H)\right\|_{L^1_\omega\to L^\infty_{-\omega}} \lesssim\frac{1}{|t|\log |t|}\ \ \text{and}\ \ \bigg\|\frac{\sin (t \sqrt{H})}{t\sqrt{H}} P_{\mathrm{ac}}(H)\bigg\|_{L^1_\omega\to L^\infty_{-\omega}}\sim|t|^{-1}, \quad |t| \gg 1.
	\]

			\vskip0.1cm
			\item If $\mathcal{E}_2=\mathcal{E}_3,$ then the  enhanced weighted estimate \eqref{regular_first_weight} holds as in the regular case.
		\end{itemize}
	\end{enumerate}
\end{theorem}

For the reader's convenience, we summarize the main dispersive estimates obtained above in Table~\ref{tab:main-results}.
\vspace{-5mm}
\newcommand{\Tcell}[1]{\makecell[c]{\rule[-1.9ex]{0pt}{5.5ex}#1}}
\begin{table}[H]
\centering
\caption{Summary of the long-time decay estimates}
\label{tab:main-results}
\vspace{1mm}
\renewcommand{\arraystretch}{1.35}
\setlength{\tabcolsep}{4pt}
\footnotesize
\begin{tabular}{|
>{\centering\arraybackslash}m{0.18\textwidth}|
>{\centering\arraybackslash}m{0.23\textwidth}|
>{\centering\arraybackslash}m{0.20\textwidth}|
>{\centering\arraybackslash}m{0.15\textwidth}|
>{\centering\arraybackslash}m{0.15\textwidth}|}
\hline
\multirow{2}{*}{\makecell[c]{\textbf{Zero-energy type}}}
&
\multirow{2}{*}{\makecell[c]{\textbf{Sub-condition}}}
&
\multirow{2}{*}{\makecell[c]{\textbf{Unweighted}
\(\boldsymbol{L^1\to L^\infty}\)\\$\boldsymbol{U_{\cos}(t)}$ \ \bf{\&} \ $\boldsymbol{U_{\sin}(t)}$ }}
&
\multicolumn{2}{c|}{\textbf{Weighted }
\(\boldsymbol{L^1_\omega\to L^\infty_{-\omega}}\)}
\\
\cline{4-5}
&
&
&
\(\boldsymbol{U_{\cos}(t)}\)
&
\(\boldsymbol{U_{\sin}(t)}\)
\\
\hline

\Tcell{Regular / 1st-kind \\ \((\mathbf k=0,1)\)}
&
\Tcell{None}
&
\Tcell{\( |t|^{-1}\)}
&
\Tcell{\((|t|\log|t|)^{-1}\)}
&
\Tcell{\( (|t|\log|t|)^{-1}\)}
\\
\hline

\multirow{3}{=}{\Tcell{2nd-kind Resonance\\ \((\mathbf k=2)\)}}
&
\Tcell{
\(\tau\not\equiv0\)\ \text{and}\
\(\ker\mathfrak M_2\neq\{0\}\)}
&
\Tcell{\( |t|^{-1}(\log|t|)^2\)}
&
\Tcell{\( |t|^{-1}(\log|t|)\)}
&
\Tcell{\(\sim |t|^{-1}(\log|t|)^2\)}
\\
\cline{2-5}

&
\multirow{2}{=}[-5mm]{\centering\makecell[c]{\(\tau\equiv0\) \  \text{or}  \
\(\ker\mathfrak M_2=\{0\}\)}}
&
\multirow{2}{*}[-5mm]{\(|t|^{-1}\)}
&
\multirow{2}{*}[-5mm]{\((|t|\log|t|)^{-1}\)}
&
\Tcell{\((|t|\log|t|)^{-1}\)\\
if \(\ker\tau=\{0\}\)}
\\
\cline{5-5}

& & & &
\Tcell{\(\sim |t|^{-1}\)\\
if \(\ker\tau\neq\{0\}\)}
\\
\hline

\Tcell{3rd-kind or\\ Eigenvalue+\(d\)-wave\\
\((\mathbf k=3\) or \(4)\)}
& \Tcell{\(d\)-wave present}
& \Tcell{\(\sim(\log|t|)^{-1}\)}
&
\Tcell{--}
&
\Tcell{--}
\\
\hline

\multirow{3}{=}{\Tcell{Eigenvalue+no \(d\)-wave\\ \((\mathbf k=4)\)}}
&
\Tcell{\(p\)-wave present}
&
\Tcell{\( |t|^{-1}(\log|t|)^2\)}
&
\Tcell{--}
&
\Tcell{--}
\\[2mm]
\cline{2-5}

&
\Tcell{No \(p\)-wave and
\(\mathcal E_2\neq\mathcal E_3\)}
&
\Tcell{\( |t|^{-1}\)}
&
\Tcell{\( (|t|\log|t|)^{-1}\)}
&
\Tcell{\(\sim |t|^{-1}\)}
\\[2mm]
\cline{2-5}

&
\Tcell{No \(p\)-wave and
\(\mathcal E_2=\mathcal E_3\)}
&
\Tcell{\( |t|^{-1}\)}
&
\Tcell{\( (|t|\log|t|)^{-1}\)}
&
\Tcell{\((|t|\log|t|)^{-1}\)}
\\[2mm]
\hline
\end{tabular}
\vskip0.2cm
\parbox{0.95\textwidth}{\footnotesize
Here
\(U_{\cos}(t)=\cos(t\sqrt H)P_{\mathrm{ac}}(H)\) and
\(U_{\sin}(t)=\sin(t\sqrt H)/(t\sqrt H) P_{\mathrm{ac}}(H)\).
The notations
\(\mathfrak M_2\), \(\tau\), \(\mathcal E_2\), and \(\mathcal E_3\)
are defined in \eqref{eq:second-moment-map},
\eqref{eq:radial-second-moment}, \eqref{E--2}, and \eqref{eq:E3},
respectively.}
\end{table}

\subsection{Further background}
		
\subsubsection{Embedded positive eigenvalue}\label{eigenvalue}
According to Kato's theorem \cite{Kat59}, the Schr\"odinger operator $-\Delta+V$ has no positive eigenvalues if $V$ is bounded and $V=o(|x|^{-1})$ as $|x|\to\infty$ (see also \cite{FHHH82, IJ03, KT06, Simon69}). This criterion, however, fails for the fourth-order operator $H=\Delta^2+V$. For any $n\ge1$, one can construct $V\in C_0^\infty(\R^n)$ such that $H$ possesses positive eigenvalues (see Feng et al. \cite[Section 7.1]{FSWY20}). These results show that the absence of positive eigenvalues for fourth-order Schr\"odinger operators is more subtle and less stable than in the second-order case under potential perturbations.
Nevertheless, if $V\in C^1(\R^2)\cap L^\infty(\R^2)$ is repulsive, i.e., $(x\cdot\nabla)V\le0$, then $H$ has no eigenvalues (see \cite[Theorem 1.11]{FSWY20}). This  criterion also applies to higher-order elliptic operators
$P(D)+V$ in any dimension $n\ge1$, provided that $P$ is
nonnegative and homogeneous.

\subsubsection{Wave equations }

For the classical wave equation with a real-valued decaying potential,
\[
\left\{\begin{array}{l}
	\partial_t^2 u + (-\Delta + V) u = 0, \quad (t,x) \in \mathbb{R} \times \mathbb{R}^n, \\
	u(0,x) = f(x), \quad \partial_t u(0,x) = g(x),
\end{array}\right.
\]
decay estimates have been extensively studied. When $V=0$, the solution operators $\cos(t\sqrt{-\Delta})$ and $\sin(t\sqrt{-\Delta})/\sqrt{-\Delta}$ satisfy  $W^{k+1,1}\to L^\infty$ and $W^{k,1}\to L^\infty$ estimates  with $O(|t|^{-1/2})$  in dimension two for $k>\frac12$. The endpoint
\(k=\frac{1}{2}\) is more delicate and typically requires a refined
formulation in Hardy-, Besov-, or BMO-type spaces; see \cite{MSW}.   By contrast, in odd spatial dimensions \(n\ge3\), the strong Huygens principle
provides a more localized representation of the free wave
propagators; see \cite{Str05}.

For \(V\neq0\) and \(n\ge3\), early systematic \(L^p\) and
\(L^\infty\)-decay estimates were obtained by Beals--Strauss
\cite{B-S} and Beals \cite{Beals}; see also \cite{GV03,DP}.
In dimension two, comparatively fewer results are available for
\(W^{k,1}\to L^\infty\) or regularized \(L^1\to L^\infty\)
estimates. Moulin \cite{Mou} proved high-frequency dispersive
estimates for the wave and Schr\"odinger groups. Assuming that zero
is regular, Kopylova \cite{Kop} obtained the decay rate
\(t^{-1}(\log t)^{-2}\) in polynomially weighted energy spaces.
Green \cite{Green14} established regularized \(L^1\to L^\infty\)
bounds with decay \(|t|^{-1/2}\), proved the weighted improvement
\(|t|^{-1}(\log|t|)^{-2}\) in the regular case, and analyzed
zero-energy obstructions. Beceanu \cite{Bec16} subsequently established pointwise, weighted,
and Strichartz estimates for almost scaling-critical potentials,
assuming that zero is regular.

In higher dimensions, Cardoso--Vodev \cite{CV} obtained optimal
dispersive estimates for \(4\le n\le7\), while
Erdo\u{g}an--Goldberg--Green \cite{EGG14} studied the influence of
zero-energy obstructions on the four-dimensional Schr\"odinger and
wave evolutions. Related Strichartz, smoothing, and maximal estimates
for rough, inverse-square, and critically decaying potentials can be
found in \cite{Bec-Gold,GV95,BPSZ03,BPSZ04}.

\subsubsection{Higher-order Schr\"odinger operators.}
For the free operator $\Delta^2$,
Ben-Artzi--Koch--Saut \cite{BKS00} proved pointwise estimates
for the kernel of $e^{-it\Delta^2}$ and its spatial  derivatives,
yielding $L^1\to L^\infty$ decay of order $|t|^{-n/4}$.
For $H=\Delta^2+V$, Feng--Soffer--Yao \cite{FSY}
established Jensen--Kato  estimates  $O(|t|^{-n/4})$ for \(n\ge5\) and  the  long-time \(L^1\to L^\infty\)
estimate  $O(|t|^{-1/2})$  for  $n=3$ when zero is a regular point. Subsequently, Erdo\u{g}an--Green--Toprak \cite{EGT19} and Green--Toprak \cite{GT19} showed the $L^1\to L^\infty$ estimates  $O(|t|^{-n/4})$  for $n=3$ and $n=4$, when zero is regular or a first-kind resonance, and observed rate changes for other zero-energy resonances.

For $n=1$, Soffer--Wu--Yao \cite{SWY22} proved $O(|t|^{-1/4})$ estimates regardless of zero resonance types. For $n=2$, Li--Soffer--Yao \cite{LSY21} obtained optimal $O(|t|^{-1/2})$ decay in the regular and first-kind resonance cases. More generally, for \(H=(-\Delta)^m+V\), \(m>1\), and
\(2m<n<4m\), Erdo\u{g}an--Goldberg--Green
\cite{EGG_ArXiv} obtained the optimal
\(O(|t|^{-n/(2m)})\) \(L^1\to L^\infty\) bound for
scaling-critical potentials under suitable spectral assumptions.

Substantial advances have also been made on $L^p$-boundedness of wave operators for $H=(-\Delta)^m+V$ with $m>1$: see \cite{GG21b, MWY22, MWY23, MWY_ArXiv23_2} for $n<2m$ with $m=2$; \cite{Yajima_2024JST} for $n=2m$ with $m=2$; \cite{EG21, Erdogan-Green23, EGG23, EGL} for $n>2m$ with $m\in\mathbb{N}^+$. Recently, Cheng et al. \cite{CSWY25} established sharp boundedness of wave operators for all zero-resonance singularities in odd dimensions $1\le n\le 4m-1$ with $m\ge2$.
\subsection{The ideas of the proof}
In this subsection, we outline the main ideas underlying the proofs of our main theorems.
The starting point is the Stone formulas:
\begin{align*}
	\cos (t \sqrt{H}) P_{\mathrm{ac}}(H) & =\frac{2}{\pi i} \int_0^{\infty} \lambda^3 \cos (t \lambda^2)\left[R_V^{+}(\lambda^4)-R_V^{-}(\lambda^4)\right]  d \lambda, \\
	\frac{\sin (t \sqrt{H})}{t\sqrt{H}} P_{\mathrm{ac}}(H) & =\frac{2}{\pi i t} \int_0^{\infty} \lambda \sin (t \lambda^2)\left[R_V^{+}(\lambda^4)-R_V^{-}(\lambda^4)\right] d \lambda.
\end{align*}
To make these identities precise, we recall that $R_0(z)=(\Delta^2-z)^{-1}$ for $z\in \mathbb{C}\setminus[0,\infty)$ and $R_V(z)=(H-z)^{-1}$ for $z\in \mathbb{C}\setminus\sigma(H)$ denote the resolvents of the free operator $\Delta^2$ and the perturbed operator $H=\Delta^2+V$ respectively.
Their boundary values (limiting resolvents) on $(0,\infty)$ are defined as
\begin{align*}
	R^\pm_0(\lambda)=\lim_{\epsilon \searrow 0}R_0(\lambda\pm i\epsilon ),\qquad  R_V^\pm(\lambda)=\lim_{\epsilon  \searrow 0}R_V(\lambda\pm i\epsilon ),\quad \lambda>0.
\end{align*}
The existence of $R^\pm_0(\lambda)$ as bounded operators from $L^{2,s}(\mathbb{R}^2)$ to $L^{2,-s}(\mathbb{R}^2)$ for any $s>1/2$ follows from the limiting absorption principle for the resolvent $(-\Delta-z)^{-1}$ (see, e.g., Agmon \cite{Agmon}) and the splitting identity
$$
R_0(z)=\frac{1}{2\sqrt z}\left[(-\Delta-\sqrt z)^{-1}-(-\Delta+\sqrt z)^{-1}\right],\quad z\in \mathbb{C}\setminus[0,\infty), \ \Im\sqrt z>0.
$$
The existence of $R^\pm_V(\lambda)$ under suitable potential decay conditions is established in the existing literature (see, e.g., \cite{FSY}).

We split the Stone integrals into low and high energies. The
high-energy estimates use resolvent bounds and integration by parts.
At low energy, the singular terms of $R_V^\pm(\lambda^4)$ depend on
the resonance type and require separate expansions.

Set $v=\sqrt{|V|}$ and let $U=\operatorname{sgn}V$ on $\{V\ne0\}$,
with $U=1$ on $\{V=0\}$. Then $U^2=\Id$ and $V=Uv^2$. We define the Birman-Schwinger operator $M^\pm(\lambda) = U + vR_0^\pm(\lambda^4)v$. The symmetric resolvent identity,
\begin{equation}\label{resolvent identity}
	R_V^\pm(\lambda^4) = R_0^\pm(\lambda^4) - R_0^\pm(\lambda^4) v \big(M^\pm(\lambda)\big)^{-1} v R_0^\pm(\lambda^4),
\end{equation}
reduces the problem to analyzing the asymptotic expansion of $(M^{\pm}(\lambda))^{-1}$ as $\lambda \to 0^+$.

For zero-energy type $\mathbf k\in\{0,1,2,3,4\}$, Theorem~\ref{thm:M_inverse} shows that $(M^\pm(\lambda))^{-1}$ admits the following expansion:
\begin{equation}\label{asymptotic expansion}
	\bigl(M^{\pm}(\lambda)\bigr)^{-1}= \sum_{0\leq\alpha,\beta\leq\mathbf{k}+2}
	\lambda^{2-k_\alpha-k_\beta} Q_\alpha \mathcal{M}_{\alpha,\beta}^{\pm}(\lambda) Q_\beta,
\end{equation}
where the singularity exponents $k_\alpha $ are given by \eqref{tab:k_alpha}. The operators $\mathcal{M}_{\alpha,\beta}^{\pm}(\lambda)\in\mathbb{B}(L^2)$ exhibit additional decay as $\lambda \to 0^+$ for certain indices $\alpha, \beta$. The moment cancellations of $Q_\alpha$ in Remark~\ref{cancellationQ}
give positive powers of $\lambda$ in the projected free resolvent
$Q_\alpha vR_0^\pm(\lambda^4)$; see Lemma~\ref{lemma_projection}.
These powers offset the singular factors in the inverse expansion.

Substituting the identities \eqref{resolvent identity} and \eqref{asymptotic expansion} into the low-energy part of the Stone formulas, we reduce the analysis to bounding a finite number of kernels
\begin{equation*}
	\big(\mathcal{K}_{B}^+ - \mathcal{K}_{B}^-\big)(t,x,y) \quad \text{and} \quad \big(\mathcal{N}_{B}^+ - \mathcal{N}_{B}^-\big)(t,x,y) ,
\end{equation*}
where the component kernels are given by
\begin{align*}
	\mathcal{K}_{B}^\pm(t,x,y) &= \frac{2}{\pi i}\int_0^{\infty} \cos(t \lambda^2)\lambda \widetilde{\chi}_1(\lambda) \left[\lambda^{4-k_\alpha-k_\beta} R_0^\pm(\lambda^{4})v Q_\alpha B^\pm(\lambda)Q_\beta vR_0^\pm(\lambda^{4})\right](x,y) \, d \lambda, \\
	\mathcal{N}_{B}^\pm(t,x,y) &= \frac{2}{\pi it}\int_0^{\infty} \sin(t\lambda^2)\lambda^{-1}\widetilde{\chi}_1(\lambda)
	\left[\lambda^{4-k_\alpha-k_\beta} R_0^\pm(\lambda^{4})v Q_\alpha B^\pm(\lambda) Q_\beta vR_0^\pm(\lambda^{4})\right](x,y) \, d \lambda,
\end{align*}
with $B^\pm(\lambda) \in \{ \mathcal{M}_{\alpha, \beta}^{\pm}(\lambda) \mid 0 \le \alpha,\beta \le \mathbf{k}+2 \}$. Lemma~\ref{lemma_projection} reduces these kernels to the oscillatory
integrals estimated in Lemma~\ref{oscillatory}.

To illustrate our general strategy for controlling these oscillatory integrals, let us examine a representative term where $B^\pm(\lambda)=\mathcal{M}_{1,1}^{\pm}(\lambda)$.
Using the Taylor expansion for $F_\pm$ and the cancellation property $Q_1 v = 0$, we extract an additional positive power of $\lambda$ (see Lemma~\ref{lemma_projection}):
$$
\bigl(Q_1 v R_0^\pm(\lambda^4)\bigr)(x,y)
= -\lambda^{-1} Q_1\Bigl(|\cdot| v \int_0^1 F_\pm'(\lambda|\theta\cdot-y|)\cos\xi\,d\theta\Bigr)(x)
=: \lambda^{-1} \Omega_{1,\pm}(\lambda,x,y),
$$
where $R^\pm_0(\lambda^4)(x,y)=\lambda^{-2}F_\pm(\lambda|x-y|)$ (see \eqref{def:F_pm}). Hence, the term $Q_1 v R_0^\pm(\lambda^4)$ is less singular than the free resolvent. Furthermore, we derive the  bound
\[
\bigl\|\partial_\lambda^\ell\bigl(e^{\mp i\lambda|y|} \Omega_{1,\pm}(\lambda,\cdot,y)\bigr)\bigr\|_{L^2}
\lesssim \langle \lambda y\rangle^{-1/2} \lambda^{-\ell}, \qquad \ell=0,1,2.
\]
With these properties, the kernels can be recast into the canonical forms appearing in Lemma~\ref{oscillatory}:
\begin{align*}
	\mathcal{K}^\pm(t,x,y) &= \frac{2}{\pi i}\int_0^{\infty} \cos (t \lambda^2)e^{\pm i\lambda r} \lambda \mathcal{E}^\pm(\lambda, x,y) \, d\lambda, \\
	\mathcal{N}^\pm(t,x,y) &= \frac{2}{\pi i t}\int_0^{\infty} \sin(t\lambda^2) e^{\pm i\lambda r} \lambda^{-1} \mathcal{E}^\pm(\lambda, x,y) \, d\lambda,
\end{align*}
with  $r(x,y) = |x| + |y|$ and the effective amplitude factor $$\mathcal{E}^\pm(\lambda,x,y) = e^{\mp i\lambda r} \widetilde{\chi}_1(\lambda) \langle \mathcal{M}_{1,1}^{\pm}(\lambda) \Omega_{1,\pm}(\cdot, y), \Omega_{1,\mp}(\cdot, x) \rangle.$$ Applying Leibniz's rule and H\"older's inequality along with the bound $\|\partial_\lambda^{\ell} \mathcal{M}_{1,1}^{\pm}(\lambda)\|_{\mathbb{B}(L^2)}\lesssim\lambda^{-\ell}$ yields
\[
\bigl|\partial_\lambda^{\ell} \mathcal{E}^\pm(\lambda,x,y) \bigr| \lesssim \bigl\langle \lambda (|x|+|y|)\bigr\rangle^{-1/2} \lambda^{-\ell} \widetilde{\chi}_1(\lambda/2),  \qquad  \ell=0,1,2.
\]
Then $\mathcal{E}^\pm$ satisfies condition \eqref{condit 2} of Lemma~\ref{oscillatory} (ii) with parameters $(\sigma, \nu)=(0,0)$ and $ h(x,y)=1$, yielding  the bound $|\mathcal{K}^\pm|+|\mathcal{N}^\pm |\lesssim \langle t \rangle^{-1}$ uniformly in $x, y$.

To establish the spatially weighted estimates for $|t| \ge 2$, we use the following expansion:
\begin{align*}
	R^\pm_0(\lambda^4)(x,y) &= \frac{a_0^\pm}{\lambda^2} + \bigl(a_1^\pm+b_0\log\lambda\bigr)|x-y|^2 + b_0|x-y|^2\log|x-y| + {O}_2\big(\lambda^2|x-y|^4\big) \nonumber \\
	&=: \frac{a_0^\pm}{\lambda^2} + E_1^\pm(\lambda, x, y).
\end{align*}
The cancellation property $Q_1 v = 0$ annihilates the highly singular leading term $\lambda^{-2}$, meaning $\bigl( Q_1 v R_0^{\pm}(\lambda^4) \bigr) = \bigl( Q_1 v E_1^{\pm} \bigr)$. This cancellation mitigates the low-energy singularity by a factor of $\lambda^{2-}$ relative to $R_0^\pm$ but introduces  spatial growth.  Provided that the potential $V$ decays sufficiently fast, we deduce the  bound:
\[
\|\partial_\lambda^\ell  \bigl(Q_1 v R_0^{\pm}(\lambda^4)\bigr)(\cdot,y)\|_{L^2}
\lesssim \|v \partial_\lambda^\ell E_1^\pm(\lambda,\cdot,y)\|_{L^2}
\lesssim \langle y\rangle^4\lambda^{-\ell-}, \quad \ell = 0, 1, 2.
\]
In this weighted scenario, we align the phase to $r(x,y)=0$ and define the amplitude factor as
$$
\mathcal{E}^\pm(\lambda,x,y)
= \widetilde{\chi}_1(\lambda) \lambda^{2}
\bigl\langle \mathcal{M}_{1,1}^\pm(\lambda) \bigl( Q_1 v E_1^{\pm} \bigr)(\cdot,y),
\bigl( Q_1 v E_1^{\mp} \bigr)(\cdot,x) \bigr\rangle.
$$
Consequently, we derive the weighted bound
\[
\bigl| \partial_\lambda^\ell \mathcal{E}^\pm(\lambda,x,y) \bigr|
\lesssim \lambda^{1/2-\ell}\widetilde{\chi}_1(\lambda/2) \langle x\rangle^{4} \langle y\rangle^{4}, \quad   \ell=0,1,2.
\]
An application of Lemma~\ref{oscillatory}(ii) with parameters $(\sigma, \nu)=(-1/2,0)$ and  $h(x,y)=\langle x\rangle^{4} \langle y\rangle^{4}$ yields the improved temporal decay $|\mathcal{K}^\pm|+|\mathcal{N}^\pm |\lesssim |t|^{-5/4}\langle x \rangle^4 \langle y \rangle^4$. Combine this with the uniform $O(|t|^{-1})$ bound and the inequality $\min(1,a/b) \lesssim \log a/\log b$ for $a,b\ge 2$, then
	  for $|t| \ge 2$, we  obtain the  weighted bound:
$$
\min\left(\frac{1}{|t|}, \frac{\langle x\rangle^4\langle y\rangle^4}{|t|^{5/4}}\right) = \frac{1}{|t|} \min\left(1, \frac{\langle x\rangle^4\langle y\rangle^4}{|t|^{1/4}}\right) \lesssim \frac{\omega(x)\omega(y)}{|t|\log|t|}.
$$

For the remaining operators $B^\pm(\lambda) \in \{\mathcal{M}_{\alpha,\beta}^{\pm}(\lambda) \mid 0 \le \alpha,\beta \le \mathbf{k}+2\}$, the essential strategy remains unchanged. By exploiting the structural bounds of $Q_\alpha v R_0^{\pm}(\lambda^4)$ in conjunction with the operator norm bounds for $\mathcal{M}_{\alpha, \beta}^{\pm}(\lambda)$, the analysis is reduced to estimating canonical oscillatory integrals in Lemma~\ref{oscillatory}.

For $\mathbf k\ge2$, the lower bounds also require the exact leading
coefficients of $\mathcal M_{\alpha,\beta}^\pm$ and a proof that the
resulting finite-rank operator is nonzero. These steps are carried out
in Propositions~\ref{sine-sharp}, \ref{sine-sharp-2},
\ref{2-sine-sharp-2}, and~\ref{sharp-3-4}, and
Lemma~\ref{lemma:non-zero_1}.

\section{Resolvent expansions}\label{sec:reso-expan}
In this section, we record the free resolvent expansions, define the
threshold projections, and state the inverse and kernel estimates used
in the low-energy analysis. The projection characterizations and the
main technical estimates are proved in
Sections~\ref{sec:proof_characterization}--\ref{section 8}.

\subsection{Free resolvent  }\label{subsec:free_resolvent}
We use the following free resolvent expansions.
Recall the expression of the free resolvent for the Laplacian (see, e.g., \cite{Erdogan-Green}):
\begin{equation}\label{resolvent-R}
	R^\pm (-\Delta;\lambda^2)(x,y) = \pm\frac{i}{4}H_0^\pm(\lambda|x-y|),
\end{equation}
where $H_0^\pm (z)= J_0(z)\pm i Y_0(z) $ are Hankel functions of order zero. From the series expansions for Bessel functions
(see, e.g., \cite{AS64}), we have
\begin{gather*}
    J_0(z)=1-\frac14z^2+\frac1{64}z^4-\frac1{2304}z^6+ \frac{1}{147456}z^8+O(z^{10}),\\
    Y_0(z)=\frac{2}{\pi}\bigl(\log\frac{z}{2}+\gamma\bigr)J_0(z)
          +\frac{2}{\pi}\Bigl(\frac14z^2-\frac3{128}z^4+\frac{11}{13824}z^6-\frac{25}{1769472}z^8+O(z^{10})\Bigr),
\end{gather*}
where $\gamma $ is Euler's constant. In addition, for $|z| \gtrsim 1$,
\begin{equation}\label{id-H0big}
H^\pm_0(z)=e^{\pm iz}w_\pm(z),\qquad |w_\pm^{(\ell)}(z)|\lesssim(1+|z|)^{-1/2-\ell}, \
\ell= 0,1,2,\cdots.
\end{equation}

Note that
\begin{equation*}
    R^\pm_0(\lambda^4) = \frac{1}{2\lambda^2} \left( R^\pm(-\Delta; \lambda^2) - R(-\Delta; -\lambda^2) \right), \quad \lambda > 0.
\end{equation*}
Then by \eqref{resolvent-R} and \eqref{id-H0big}, it follows that
\begin{equation}\label{reso-big}
    R_0^\pm(\lambda^4)(x,y) = \frac{i}{8\lambda^2} \left[ \pm e^{\pm i\lambda |x - y|} w_\pm(\lambda|x - y|) - e^{-\lambda |x - y|} w_+(i\lambda|x - y|) \right],\ \ \  \lambda|x - y| \gtrsim 1.
\end{equation}
 For $\lambda|x-y| \ll 1 $,  the series expansions of Bessel functions give
   \begin{equation}\label{reso_small}
   \begin{split}
    R^\pm_0(\lambda^4)(x,y)=&\frac{a_0^\pm}{\lambda^2}G_0(x,y)+g_0^\pm(\lambda)G_2(x,y)+b_0G_2^0(x,y)+a_2^\pm\lambda^2G_4(x,y)\\
&+g_1^\pm(\lambda)\lambda^4G_6(x,y)+b_1\lambda^4G_6^0(x,y)+a_4^\pm\lambda^6G_8(x,y)
    +O_2\bigl(\lambda^{8-\varepsilon}|x-y|^{10-\varepsilon}\bigr),
    \end{split}
    \end{equation}
where
$ b_0 = \frac{1}{8\pi}, \ b_1 = \frac{1}{4608\pi},\
a_0^\pm = \pm \frac{i}{8}, \  
a_2^\pm = \pm \frac{i}{512}, \      
a_4^\pm=\pm \frac{i}{1179648},$ and
$$
g_0^\pm(\lambda)=b_0\log\lambda+a_1^\pm,\qquad
g_1^\pm(\lambda)=b_1\log\lambda+a_3^\pm.
$$
The coefficients satisfy $a_j^-=\overline{a_j^+}$ for $0\le j\le4$.
In particular,
\[
\operatorname{Im}a_1^\pm=\mp\frac1{32},\qquad
\operatorname{Im}a_3^\pm=\mp\frac1{18432}.
\]
\begin{equation}\label{def:G_k}
   G_{k}(x,y)= |x-y|^k, \quad
   G_k^0(x,y)= |x-y|^k\log|x-y|.
\end{equation}
Combining \eqref{reso-big} and \eqref{reso_small}, we obtain the following expansion of $R^\pm_0(\lambda^4)(x,y)$ for any $\lambda > 0$:
\begin{lemma}[cf.~{\cite[Lemma 2.2]{LSY21}}]\label{lem-reso}
For any $\lambda> 0 $, we have the following expansions:
\begin{equation}\label{free-R0pmlambda4}
   \begin{split}
	  R^\pm_0(\lambda^4)(x,y)=&\frac{a_0^\pm}{\lambda^2}G_0(x,y)+g_0^\pm(\lambda)G_2(x,y)+b_0G_2^0(x,y)+a_2^\pm\lambda^2G_4(x,y)\\
&+g_1^\pm(\lambda)\lambda^4G_6(x,y)+b_1\lambda^4G_6^0(x,y)+O_2\bigl(\lambda^{6}|x-y|^{8}\bigr).
		\end{split}
	\end{equation}
\end{lemma}
For convenience in the subsequent analysis,
We define
\begin{equation}\label{def:F_pm}
R^\pm_0(\lambda^4)(x,y):=\lambda^{-2}F_\pm(\lambda|x-y|),\quad
\widetilde{R}^\pm(\lambda|x-y|):=\lambda^{2}e^{\mp i\lambda|x-y|}R^\pm_0(\lambda^4)(x,y).
\end{equation}
Then $F_\pm(p)=e^{\pm ip}\widetilde{R}^\pm(p)$. From \eqref{reso-big} and \eqref{reso_small},    we have
\begin{equation}\label{free-widetilde_R}
\begin{aligned}
\widetilde{R}^\pm(p)=&\,e^{\mp ip}\Bigl[a_0^\pm+b_0p^2\log p+a_1^\pm p^2
                     +a_2^\pm p^4+b_1p^6\log p+a_3^\pm p^6+O\bigl(p^{8}\bigr)\Bigr]\chi_1(p)\\
                     &+\frac{i}{8}\Bigl(\pm w_\pm(p)-e^{\mp ip}e^{-p}w_+(ip)\Bigr)\chi_2(p).
\end{aligned}
\end{equation}

In the sequel,  we denote by $G_k^0$ (resp. $G_k$) the integral operator with kernel $G_k^0(x,y)$ (resp. $G_k(x,y)$) defined in \eqref{def:G_k}.
In particular, $b_0 G_2^0$ is a fundamental solution of $\Delta^2$, i.e., $\Delta^2 (b_0 G_2^0) = \delta_0$ in the sense of distributions. 

\subsection{\texorpdfstring{Asymptotic expansion of $R_V^\pm(\lambda^4)$ near zero}{Perturbed resolvent near zero}}\label{subsec:Q}  
For the projection arguments, assume $V\not\equiv0$; Section~\ref{sec:free_case} treats $V=0$. We first define the spaces used to invert the symmetric resolvent operator.  

\begin{definition}\label{defS}
	Let $x = (x_1, x_2) \in \mathbb{R}^2$, $v=\sqrt{|V|}$.  Define $P = \|V\|_{L^1(\mathbb{R}^2)}^{-1} \langle \cdot, v\rangle v$ and $T_0 : = U + b_0 v G_2^0 v $. Set  $S_{-1} = \Id$, $S_6 = 0$, and for $j=0,\dots,5$, let $S_j$ be the orthogonal projection onto the subspace $S_jL^2$ defined as follows:
	\begin{itemize}
			\item $S_0L^2 = \{ f \in L^2 \mid \langle  v, f \rangle = 0\}$.
		\item $S_1L^2 = \{ f \in S_0L^2 \mid \langle x_j v, f \rangle = 0,\ j = 1,2 \}$.
		\item $S_2L^2 = \{ f \in S_1L^2 \mid S_1T_0f = 0 \}$.
		\item $S_3L^2 = \{ f \in S_2L^2 \mid S_0T_0f = 0 \}$.
		\item $S_4^0L^2 = \{ f \in S_3L^2 \mid \langle |x|^2 v, f \rangle = 0 \}$.
        \item $S_4^1L^2 = \{ f \in S_3L^2 \mid \langle x_ix_j v, f \rangle = 0,\ i,j = 1,2\}$.
		\item $S_4L^2 = \{ f \in S_3L^2 \mid \langle x_ix_j v, f \rangle = 0,\ i,j = 1,2,\ \text{and}\ PT_0f = 0 \}$.
		\item $S_5L^2 = \{ f \in S_4L^2 \mid \langle x_ix_jx_k v, f \rangle = 0,\ i,j,k = 1,2 \}$.
	\end{itemize}
\end{definition}

The decay assumption \eqref{condition} makes these projections
well defined and $vG_2^0v$ Hilbert--Schmidt. Since $I-S_1$ has finite
rank and $U^2=I$, the compression $S_1US_1$ is Fredholm on $S_1L^2$.
Its compact perturbation $S_1T_0S_1$ is therefore Fredholm, so its
kernel projection $S_2$ has finite rank. Also, $P+S_0=I$ and the
definitions give $T_0S_j=S_jT_0=0$ for $j=4,5$.

From the definitions above, we have the chain of inclusions
\[
S_0 L^2 \supseteq S_1 L^2 \supseteq S_2 L^2 \supseteq S_3 L^2 \supseteq S_4^0 L^2  \supseteq S_4^1 L^2 \supseteq S_4 L^2 \supseteq S_5 L^2,
\]
and the projections $S_j$ (for $j=2,\cdots,5$) and $S_4^0, S_4^1$ are finite-rank operators.

We now introduce the following complementary family of orthogonal projections.
\begin{definition}\label{defQ}
Define
$$
	Q_j := S_{j-1} - S_j \quad \text{for } j = 0,1,\dots,6, \quad \text{and} \quad Q_4^0 = S_3 - S_4^0, \quad Q_4^1 = S_4^0 - S_4.
$$
and further decompose $Q_4^1$ into the sum of the following projectors:
$$
Q_{41}^1 := S_4^0 - S_4^1, \qquad Q_{42}^1 := S_4^1 - S_4. 
$$
\end{definition}
In particular, $Q_0=P$.
According to Definitions \ref{defS} and \ref{defQ}, we obtain the following cancellation properties:
\begin{remark}\label{cancellationQ}
	The orthogonal projections $Q_4^0$, $Q_4^1$, and $Q_\alpha$ (for $\alpha \in \mathbb{N}_0$ with $1 \le \alpha \le 6$) satisfy:
	\begin{align*}
		& Q_{1}v = 0 \quad (\text{i.e., } \langle v, Q_{1}f \rangle = 0,\ \forall f \in L^2), \\
		& Q_{\alpha}v = Q_{\alpha}(x_jv) =Q_{4}^{0}v = Q_{4}^{0}(x_jv) =  0 \quad (\alpha = 2,3,4), \\
		& Q_{4}^{1}v = Q_{4}^{1}(x_jv )= Q_{4}^{1}(|x|^2v )= Q_{41}^{1}v = Q_{41}^{1}(x_jv )= Q_{41}^{1}(|x|^2v )= 0,   \\
        & Q_{42}^{1}v = Q_{42}^{1}(x_jv )= Q_{42}^{1}(x_ix_jv )=Q_{5}v = Q_{5}(x_jv) = Q_{5}(x_ix_jv) = 0,\quad Q_5T_0=T_0Q_5=0,\\
		& Q_{6}v = Q_{6}(x_jv) = Q_{6}(x_ix_jv) = Q_{6}(x_ix_jx_kv) = 0,\ Q_6T_0=T_0Q_6=0.
	\end{align*}
\end{remark}

 Besides Definition \ref{definition1}, zero-energy resonances of $H$ can also be characterized by the orthogonal projections $S_j$; the proof is given in Section \ref{sec:proof_characterization}. Recall that $\mathbf{k} \in \{0,1,2,3,4\}$  denotes  the zero-energy spectral types: $\mathbf{k}=0$ for a regular point, $\mathbf{k} \in \{1,2,3\}$ for zero resonances, and $\mathbf{k}=4$ for a zero eigenvalue. 
Throughout this paper, the decay index $\mu$ of the potential $V$ is required to satisfy 
\begin{equation}\label{condition}
	\mu> \begin{cases}11, & \mathbf{k}=0,1, \\
14, &  \mathbf{k}=2,\\
18, & \mathbf{k}=3,4. \end{cases}
		\end{equation}
\begin{proposition}\label{characterizations}
Assume that $\left|V(x)\right| \lesssim \left \langle  x\right \rangle^{-\mu}$ for some $\mu>0$
satisfying the condition \eqref{condition}.
Then 
\begin{itemize}
\item[(i)] Zero  is a regular point of $H$  $( \text{i.e.,}\  \mathbf{k}=0 )$  if and only if $S_{2}=0$;
\item[(ii)] Zero has type $\mathbf{k}\in\{1,2,3,4\}$ if and only if  $S_{\mathbf{k}+1} \neq 0$ and $S_{\mathbf{k}+2}=0$.
\end{itemize}
\end{proposition}

\begin{proposition}\label{characterizations Q}
Under the assumptions of Proposition~\ref{characterizations}, zero is
regular if and only if $\sum_{j=0}^2Q_j=\Id$. For $1\le\mathbf k\le4$,
zero has type $\mathbf k$ if and only if
\[
\sum_{j=0}^{\mathbf k+1}Q_j\ne\Id,\qquad
\sum_{j=0}^{\mathbf k+2}Q_j=\Id.
\]
\end{proposition}
This follows from $\sum_{j=0}^mQ_j=\Id-S_m$.
We use the convention $U=1$ on $\{V=0\}$, so that $U^2=\Id$ and
$V=Uv^2$. The symmetric resolvent identity is
 \begin{equation}\label{id-RV}
R^\pm_V(\lambda^4) = R_0^\pm(\lambda^4) -R^\pm_0(\lambda^4)v(M^\pm(\lambda))^{-1}vR^\pm_0(\lambda^4),
\end{equation}
where $M^{\pm}(\lambda)=U+ vR^\pm_0(\lambda^4)v$. Hence, we need to establish the expansions of
$(M^\pm(\lambda))^{-1}$.

Recall that  $T_0=U+b_0 v G_2^0 v$.
Using the expansions of $R^\pm_0(\lambda^4)(x,y)$ from Lemma \ref{lem-reso}, we obtain the following asymptotic expansions of $M^\pm(\lambda)$:
\begin{itemize}
\item If $\mu > 11$, then
\begin{equation}\label{Mpm-1}
\begin{split}
M^\pm(\lambda) = \frac{a_0^\pm}{\lambda^2}v G_{0} v + g_0^\pm(\lambda)v G_{2} v + T_0 + \Gamma_2^\pm (\lambda);
\end{split}
\end{equation}

\item If $\mu > 14$, then
\begin{equation}\label{Mpm-2}
\begin{split}
M^\pm(\lambda) = \frac{a_0^\pm}{\lambda^2}v G_{0} v + g_0^\pm(\lambda)v G_{2} v + T_0 + a_2^\pm\lambda^2 vG_4 v + \Gamma_{4-\epsilon}^\pm (\lambda);
\end{split}
\end{equation}

\item
If $\mu > 18$, then
\begin{equation}\label{Mpm-3}
\begin{split}
M^\pm(\lambda) = \frac{a_0^\pm}{\lambda^2}v G_{0} v + g_0^\pm(\lambda)v G_{2} v + T_0 + a_2^\pm\lambda^2 vG_4 v + g_1^\pm(\lambda)\lambda^4vG_6v + b_1\lambda^4 vG_6^0v  + \Gamma_6^\pm (\lambda).
\end{split}
\end{equation}
\end{itemize}
Here, $\Gamma_s^\pm (\lambda)$ denote $\lambda$-dependent Hilbert-Schmidt operators satisfying
\[
\big\| \Gamma_s^\pm (\lambda) \big\|_{\mathbb{B}(L^2)} +\lambda \big\| \partial_\lambda \Gamma_s^\pm (\lambda) \big\|_{\mathbb{B}(L^2)}+\lambda^2 \big\| \partial_\lambda^2 \Gamma_s^\pm (\lambda) \big\|_{\mathbb{B}(L^2)}
\lesssim \lambda^{s}.
\]

We use the $O_k$ notation from Subsection~\ref{Notations} with the
$\mathbb B(L^2)$ norm. Thus each $O_2$ remainder below includes bounds
for two energy derivatives.

We prove the following inverse expansion in Section~\ref{subsec:proof_inverse}. 

\begin{theorem}\label{thm:M_inverse}
	Assume that zero has type $\mathbf k\in\{0,1,2,3,4\}$ and $|V(x)| \lesssim \left \langle  x\right \rangle^{-\mu}$ for some $\mu>0$ satisfying the condition \eqref{condition}.
	Then there exists $0 < \lambda_0 \ll 1$ such that for all $0 < \lambda \le (2\lambda_0)^{1/4}$, the operator $\left(M^{\pm}(\lambda)\right)^{-1}$ on $L^2(\mathbb{R}^2)$ admits the asymptotic expansion:
	\begin{equation}\label{eq:M_inverse}  
		\left(M^{\pm}(\lambda)\right)^{-1} = \sum_{0 \leq \alpha, \beta \leq \mathbf{k}+2} \lambda^{2-k_\alpha-k_\beta} Q_\alpha \mathcal{M}_{\alpha, \beta}^{\pm}(\lambda) Q_\beta ,  
	\end{equation}  
	where the values of the parameter $k_\alpha$ are assigned as follows:  
	\begin{equation}\label{tab:k_alpha}  
		\begin{array}{c|ccccccc}  
			\alpha & 0 & 1 & 2 & 3 & 4 & 5 & 6 \\  
			\hline  
			k_\alpha & 0 & 1 & 1 & 1 & 2 & 3 & 3  
		\end{array}  
	\end{equation}  
	Moreover, the blocks $\mathcal{M}_{\alpha, \beta}^{\pm}(\lambda)$ satisfy:  
	
	\textbf{{\rm (I)}\  If zero is a regular point or a first-kind resonance of $H$} $(\text{i.e.,}\  \mathbf{k}=0,1 )$, then 
	\begin{align}\label{M0-1}
	\Big\| \partial_\lambda^\ell \mathcal{M}_{\alpha,\beta}^\pm(\lambda) \Big\|_{\mathbb{B}(L^2)} \lesssim
	\begin{cases}
		\lambda^{-\ell}, & \text{for } ( \alpha,\beta) \neq (3,3), \\
		|\log \lambda| \lambda^{-\ell}, &  \text{for } (\alpha,\beta) = (3,3),
	\end{cases}
	\end{align}
for  all $0 \le \alpha, \beta \le 3$ and $\ell=0,1,2$.	In particular, for  $(\alpha, \beta) = (0,0)$,
	\[
		\mathcal{M}_{0,0}^\pm(\lambda) = (a_0^\pm)^{-1} \|V\|_{L^1}^{-1}Q_0  + \lambda (-\log \lambda)^{3/2}\Lambda^\pm(\lambda).
	\]	  
Here and hereafter,  $\Lambda^\pm(\lambda)$ denotes a generic operator in $\mathbb{B}(L^2)$ (possibly varying at each occurrence) satisfying
$$
\left\|\partial_\lambda^\ell \Lambda^\pm(\lambda)\right\|_{\mathbb{B}(L^2)} \lesssim \lambda^{-\ell},\quad \ell = 0, 1, 2.$$
	\textbf{{\rm (II)}\ If  zero is a  second-kind resonance of $H$} $(\text{i.e.,}\  \mathbf{k}=2 )$,  then
	for all $0 \le \alpha, \beta \le 3$ and $\ell=0,1,2$, the estimates \eqref{M0-1} remain valid. Moreover,
for $\alpha \in \{1,2,3\}$,   
	\[  
	\big\|\partial_\lambda^\ell \mathcal{M}_{4,\alpha}^\pm(\lambda)\big\|_{\mathbb{B}(L^2)} + \big\|\partial_\lambda^\ell \mathcal{M}_{\alpha,4}^\pm(\lambda)\big\|_{\mathbb{B}(L^2)} \lesssim \lambda^{1-\ell}|\log \lambda|^{4}, \quad \ell = 0,1,2.  
	\]  
	It remains to specify the case
	$
	(\alpha,\beta)
	\in
	\{(0,4),(4,0),(4,4)\}
,$
	which are given as follows:
    
	\noindent\textbf{Case 1: $Q_4^0 = 0$.}   For all $(\alpha,\beta)
	\in
	\{(0,4),(4,0),(4,4)\}$, we have
	\[
	\mathcal{M}_{\alpha,\beta}^\pm(\lambda) = Q_\alpha \Lambda^\pm(\lambda)Q_\beta.
	\]
	\textbf{Case 2: $Q_4^0\ne0$.} We distinguish two subcases.
	\begin{itemize}  
		\item \textit{ When $Q_{42}^1 = 0$:}  
		$$
		\begin{aligned}
			&\mathcal{M}_{4,0}^\pm(\lambda)=(\log \lambda)^{-1}  Q_4 \Lambda^\pm(\lambda) Q_0,  \quad
			\mathcal{M}_{0,4}^\pm(\lambda)=(\log \lambda)^{-1}  Q_0 \Lambda^\pm(\lambda) Q_4,  \\
			&\mathcal{M}_{4,4}^\pm(\lambda) = \sum_{h,l\in\{0,1\}}(\log \lambda)^{-(2-h-l)} Q_4^h \Lambda^\pm(\lambda) Q_4^l;
		\end{aligned}
		$$	
		\item \ \textit{When $Q_{42}^1 \neq 0$:}  
		\begin{equation*}  
			\begin{aligned}  
				\mathcal{M}_{0,4}^{\pm}(\lambda) &= (\log \lambda)^{-1} Q_0 \Lambda^\pm(\lambda) Q_{4}^0 + Q_0 \Lambda^\pm(\lambda) Q_{4}^1, \\  
				\mathcal{M}_{4,0}^{\pm}(\lambda) &= (\log \lambda)^{-1} Q_{4}^0 \Lambda^\pm(\lambda) Q_0 + Q_{4}^1 \Lambda^\pm(\lambda) Q_0, \\  
				\mathcal{M}_{4,4}^{\pm}(\lambda) &= Q_{4} \Lambda^\pm(\lambda) Q_{4} + (\log\lambda)\bigl(Q_{4} \Lambda^\pm(\lambda) Q_{4}^1+Q_{4}^1 \Lambda^\pm(\lambda) Q_{4}\bigr) + (\log\lambda)^2 Q_{4}^1 \Lambda^\pm(\lambda) Q_{4}^1.
			\end{aligned}  
		\end{equation*}  
	\end{itemize}  

	\textbf{{\rm (III)} If  zero  is a  third-kind resonance or an eigenvalue of $H$} $(\text{i.e.,}\  \mathbf{k}=3,4 )$,  then
	for all $0 \le \alpha, \beta \le 4$ and $\ell=0,1,2$, all the conclusions established in the second resonance case   remain valid. Moreover,
	for $\alpha \in \{5,6\}$ and $0 \le \beta \le 4$,
	\[  
	\big\|\partial_\lambda^{\ell} \mathcal{M}_{\alpha,\beta}^\pm(\lambda)\big\|_{\mathbb{B}\left(L^2\right)}  
	+ \big\|\partial_\lambda^{\ell} \mathcal{M}_{\beta,\alpha}^\pm(\lambda)\big\|_{\mathbb{B}\left(L^2\right)}  
	\lesssim\lambda^{-\ell},\quad \ell=0,1,2.
	\]  
	For the remaining principal pair $(\alpha,\beta) \in \{(5,5), (5,6), (6,5), (6,6)\},$  
	\[  
	\mathcal{M}_{\alpha,\beta}^\pm(\lambda) = 
	 \mathcal{A}_{\alpha,\beta}(\lambda) + (\log \lambda)^{-2}\Gamma_{\alpha,\beta}^{\pm}(\lambda).  
	\]  
	Here, the principal terms $\mathcal{A}_{\alpha,\beta}(\lambda)$ are independent of the sign $\pm$ and 
	all $\mathcal{A}_{\alpha,\beta}(\lambda), \Gamma_{\alpha,\beta}^{\pm}(\lambda)\in \mathbb{B}(L^2)$  satisfy
	\[
	\big\|\partial_\lambda^\ell \mathcal{A}_{\alpha,\beta}(\lambda)\big\|_{\mathbb{B}(L^2)} + \big\|\partial_\lambda^\ell \Gamma_{\alpha,\beta}^{\pm}(\lambda)\big\| _{\mathbb{B}(L^2)} \lesssim \lambda^{-\ell}, \quad \ell = 0,1,2.
	\]
\end{theorem}
To clarify the singular hierarchy in Theorem~\ref{thm:M_inverse},
Table~\ref{tab:Q-leading-order} records the zero-threshold types,
their \(S\)-projection characterizations, and the corresponding
worst leading orders of \((M^\pm(\lambda))^{-1}\) as
\(\lambda\to0^+\).

\vskip-0.3cm
\begin{table}[H]
	\centering
	\caption{Zero-threshold types and leading orders of
		\((M^\pm(\lambda))^{-1}\).}
	\label{tab:Q-leading-order}
	\vskip0.1cm
	\renewcommand{\arraystretch}{1.45}
	\setlength{\tabcolsep}{4pt}
	\footnotesize
	
	\newcommand{\Qcell}[1]{\makecell[c]{\rule[-1.8ex]{0pt}{6.0ex}#1}}
	
	\begin{tabular}{|
			>{\centering\arraybackslash}m{0.20\textwidth}|
			>{\centering\arraybackslash}m{0.34\textwidth}|
			>{\centering\arraybackslash}m{0.36\textwidth}|}
		\hline
		\textbf{Zero-threshold type}
		&
		\textbf{\(S\)-characterization}
		&
		\textbf{The leading order of \((M^\pm(\lambda))^{-1}\)}
		\\[1mm]
		\hline
		
		\Qcell{Regular point\\ \((\mathbf k=0)\)}
		&
		\Qcell{\(S_2=0\)}
		&
		\Qcell{\(O(1)\)}
		\\
		\hline
		
		\Qcell{1st-kind resonance\\ \((\mathbf k=1)\)}
		&
		\Qcell{
			\(\displaystyle S_2\neq 0\),
			\ \ \(\displaystyle S_3=0\)}
		&
		\Qcell{\(O(|\log\lambda|)\)}
		\\
		\hline
		
		\multirow{2}{=}{\centering 2nd-kind\\ \((\mathbf k=2)\)}
		&
		\multirow{2}{=}{\centering
				\(S_3\neq0\),\ \(S_4=0\)}
		&
		\Qcell{
			\(O\!\left(\lambda^{-2}|\log\lambda|^2\right)\)\ \
			if\ \  \(Q_4^0\neq0\) and \(Q_{42}^1\neq0\)}
		\\
		\cline{3-3}
		
		&
		&
		\Qcell{
		\ \   \ \   	\(O(\lambda^{-2})\)
		  \  \  \ \ \ \   	if \ \ \(Q_4^0=0\) or \(Q_{42}^1=0\)}
		\\
		\hline
		
		\Qcell{3rd-kind \\ \((\mathbf k=3)\)}
		&
		\Qcell{
			\(\displaystyle S_4\neq 0\),
			\ \ \(\displaystyle S_5=0\)}
		&
		\Qcell{\(O\!\left(\lambda^{-4}|\log\lambda|^{-1}\right)\)}
		\\
		\hline
		
		\Qcell{Zero eigenvalue\\ \((\mathbf k=4)\)}
		&
		\Qcell{
			\(\displaystyle S_5\neq0\)}
		&
		\Qcell{\(O(\lambda^{-4})\)}
		\\
		\hline
	\end{tabular}
	
	\vskip0.2cm
	\parbox{0.95\textwidth}{\footnotesize
The projections \(S_j\) are defined in Definition \ref{defS}. The
\(Q\)-projections are the successive orthogonal differences of 
\(S\):
\[
Q_j=S_{j-1}-S_j,\qquad 0\le j\le6.
\]
 In the second-kind case, 
\(Q_4\) is further decomposed into $Q_4^0+Q_4^1$ where
\[
Q_4^0=S_3-S_4^0,\quad
Q_4^1=S_4^0-S_4,\quad
Q_4^1=Q_{41}^1+Q_{42}^1,\quad
Q_{41}^1=S_4^0-S_4^1,\quad
Q_{42}^1=S_4^1-S_4.
\]
}
\end{table}

The moment cancellations in Remark~\ref{cancellationQ} improve the
projected free resolvent and offset the singular blocks in
\eqref{eq:M_inverse}. We record the kernel estimates and oscillatory
integral bounds used below.

\subsection{Two fundamental lemmas}\label{subsec:osci_integral}
We first estimate the kernel of $Q_\alpha vR_0^\pm(\lambda^4)$,
defined by
\[
\bigl(Q_\alpha v R_0^\pm(\lambda^4) f\bigr)(x) = Q_\alpha\!\Bigl( v(\cdot) \int_{\mathbb{R}^2} R_0^\pm(\lambda^4)(\cdot, y) f(y) \, dy \Bigr)(x).
\]
Its integral kernel is given explicitly by
\(
\bigl(Q_\alpha v R_0^\pm(\lambda^4)\bigr)(x, y) = Q_\alpha\!\left( v\, R_0^\pm(\lambda^4)(\cdot, y) \right)(x).
\)
The next lemma gives its decompositions and $L^2$ bounds.

\begin{lemma}\label{lemma_projection}
	Let $Q_\alpha$ $(\alpha\in\mathbb{N}_0, 0\le\alpha\le6)$ and $Q_4^0, Q_4^1, Q_{41}^1, Q_{42}^1$ be defined by Definition \ref{defQ}. Assume $\left|V(x)\right| \lesssim \left \langle  x\right \rangle^{-\mu}$ for some $\mu>0$ satisfying the condition \eqref{condition}. Then, for $0<\lambda\ll1$, the integral kernels admit the following decompositions:
	\begin{enumerate}
		\item[(i)] For $0\le\alpha\le6$,
		\[
		\bigl(Q_\alpha v R_0^\pm(\lambda^4)\bigr)(x,y)=\lambda^{-1-\delta_{\alpha0}}\,\Omega_{\alpha,\pm}(\lambda,x,y),
		\]
		where $\delta_{\alpha0}=1$ if $\alpha=0$ and $0$ otherwise.
		
		\item[(ii)] For $2\le\alpha\le6$, 
		\begin{align*}
		\bigl(Q_\alpha v R_0^\pm(\lambda^4)\bigr)(x,y)&=\mathcal{J}_{\alpha, \pm} (\lambda,x,y).
		\end{align*}
		\item[(iii)] For $2\le\alpha\le6$,
		setting $m_\alpha=0$ for $\alpha=2,3,4$, $m_5=1$, and $m_6=2$:
		$$
			\bigl(Q_\alpha v R_0^\pm(\lambda^4)\bigr)(x,y) =(b_0Q_\alpha v G_2^0)(x,y)+b_0(\log\lambda)\,Q_\alpha(v|\cdot|^2)(x)+ \lambda^{m_\alpha} (\mathcal{T}_{\alpha, \pm} + \mathcal{T}_\alpha)(\lambda, x, y),
        $$
		where the term $b_0 (\log\lambda) \, Q_\alpha(v|\cdot|^2)(x)$ vanishes
        for $\alpha=5,6$.
		In particular,
		\begin{align*}
\bigl(Q_4^0 v R_0^\pm(\lambda^4)\bigr)(x,y) &=(b_0Q_4^0 v G_2^0)(x,y)+b_0(\log\lambda)\,Q_4^0(v|\cdot|^2)(x)+(\mathcal{T}_{4,\pm}^0+\mathcal{T}_4^0)(\lambda,x,y),\\
\bigl(Q_4^1 v R_0^\pm(\lambda^4)\bigr)(x,y) &=(b_0Q_4^1 v G_2^0)(x,y)+(\mathcal{T}_{4,\pm}^1+\mathcal{T}_4^1)(\lambda,x,y),\\
\bigl(Q_{41}^1  v R_0^\pm(\lambda^4)\bigr)(x,y) &=(b_0Q_{41}^1  v G_2^0)(x,y)+(\mathcal{T}_{4,\pm}^{1,1}+\mathcal{T}_4^{1,1})(\lambda,x,y),\\
\bigl(Q_{42}^1  v R_0^\pm(\lambda^4)\bigr)(x,y) &=(b_0Q_{42}^1  v G_2^0)(x,y)+\lambda\bigl(\mathcal{T}_{4,\pm}^{1,2}+\mathcal{T}_4^{1,2})(\lambda,x,y).
		\end{align*}
	\end{enumerate}
	Furthermore, for $\ell=0,1,2$, these integral kernels satisfy the following bounds:
	\begin{align}\label{J}
		&\|\partial_\lambda^\ell\big(e^{\mp i\lambda|y|}\mathcal{J}_{\pm}(\lambda,\cdot,y)\big)\|_{L^2} \lesssim\lambda^{-\ell}|\log\lambda|\langle\lambda y\rangle^{-1/2},\  \text{for } \mathcal{J}_{\pm} \in \{\mathcal{J}_{\alpha,\pm}|_{2\le\alpha\le6}\},\\
&\bigl\|(b_0Q v G_2^0)(\cdot,y)\bigr\|_{L^2} \lesssim
		\begin{cases}
		\log(2+|y|), & \text{for } Q \in \{Q_\alpha|_{2 \le \alpha \le 4}, Q_4^0\}, \\[2pt]
		1, & \text{for } Q \in \{Q_4^1, Q_{41}^1, Q_{42}^1, Q_5, Q_6\},
		\end{cases} \label{lemma_projection_G} \\
&\|\partial_\lambda^\ell\mathcal{T}(\lambda,\cdot,y)\|_{L^2} \lesssim
		\begin{cases}
		\lambda^{-\ell}\log(2+|y|), & \text{for } \mathcal{T} \in \{\mathcal{T}_\alpha|_{2\le\alpha\le4}, \, \mathcal{T}_4^0\}, \\[2pt]
		\lambda^{-\ell}, & \text{for } \mathcal{T} \in \{\mathcal{T}_4^{1}, \mathcal{T}_4^{1,1}, \mathcal{T}_4^{1,2}, \mathcal{T}_5, \mathcal{T}_6\},
		\end{cases} \label{lemma_projection_T} \\   &\|\partial_\lambda^\ell\big(e^{\mp i\lambda|y|}K_{\pm}(\lambda,\cdot,y)\big)\|_{L^2} \lesssim\lambda^{-\ell}\langle\lambda y\rangle^{-1/2}, \label{lemma_projection_1}
	\end{align}
	where $K_{\pm}$ generically denotes any  term among $\Omega_{\alpha,\pm} \, (0\le\alpha\le6)$, $\mathcal{T}_{\alpha,\pm} \, (2\le\alpha\le6)$, $\mathcal{T}_{4,\pm}^j \, (j=0,1)$, or $\mathcal{T}_{4,\pm}^{1,j} \, (j=1,2)$.
\end{lemma}

The proofs of Lemmas~\ref{lemma_projection} and~\ref{oscillatory} are given in Section~\ref{section 8}.

\begin{lemma}\label{oscillatory}    Let \(h(x,y)\) and \(r:=r(x,y)\) be real-valued functions on \(\mathbb{R}^2 \times \mathbb{R}^2\) with \(h(x,y)>0\). Define the oscillatory integrals    
	\[
	\begin{aligned}
		\mathcal{K}^\pm(t,x,y) &= \int_0^{\infty} \cos (t \lambda^2)e^{\pm i\lambda r} \lambda \mathcal{E}^\pm(\lambda, x,y) \, d\lambda, \\
		\mathcal{N}^\pm(t,x,y) &= t^{-1}\int_0^{\infty} \sin(t\lambda^2) e^{\pm i\lambda r} \lambda^{-1} \mathcal{E}^\pm(\lambda, x,y) \, d\lambda.
	\end{aligned}
    \]  
	\begin{itemize}    
		\item[(i)] Suppose that for $\ell = 0, 1, 2$, the amplitude  factor $\mathcal{E}^\pm$ satisfies    
		\begin{align}\label{condit 1}    
		\left| \partial_\lambda^{\ell} \mathcal{E}^\pm(\lambda, x,y) \right| \lesssim h(x, y) \langle \lambda r \rangle^{-\frac{1}{2}} \lambda^{-\ell}.    
		\end{align}    
		Then    
		\[    
		\left| \mathcal{K}^\pm(t,x,y) \right| + \left| \mathcal{N}^\pm(t,x,y) \right| \lesssim h(x, y) |t|^{-1}.    
		\]    
		
		\item[(ii)] Suppose that for $\ell = 0, 1, 2$,  the amplitude factor $\mathcal{E}^\pm$ is supported on $ |\lambda|\leq2(2\lambda_0)^{1/4}$ and satisfies    
		\begin{align}\label{condit 2}    
		\left| \partial_\lambda^{\ell} \mathcal{E}^\pm(\lambda, x,y) \right| \lesssim h(x, y) \langle \lambda r \rangle^{-\frac{1}{2}} |\log \lambda|^{-\nu} \lambda^{-\sigma - \ell} \widetilde{\chi}_1 (\lambda/2).    
		\end{align}    
		Then the following estimates hold:    
		
		$\bullet$ For $0 < \sigma < 2$ with $\nu \in \mathbb{R}$, or $\sigma = 0$ with $\nu \le 0$:    
		\[    
		\left| \mathcal{K}^\pm(t,x,y) \right| + \left| \mathcal{N}^\pm(t,x,y) \right| \lesssim \frac{h(x, y)}{\langle t \rangle^{1-\frac{\sigma}{2}} (\log(2+|t|))^{\nu}}.    
		\]    
Moreover, if $r=0,$  this estimate  remains
   valid for the larger parameter range
   $
   -2<\sigma<2$ and 
   $\nu\in\mathbb R.
   $		
		
        $\bullet$ For $\sigma = 2$ and $\nu > 1$:    
		\[    
		\left| \mathcal{K}^\pm(t,x,y) \right| + \left| \mathcal{N}^\pm(t,x,y) \right| \lesssim  \frac{h(x, y)}{\left(\log(2+|t|)\right)^{\nu-1}}.    
		\]    
	\end{itemize} 
\end{lemma}    

\section{The decay estimates for the free case}\label{sec:free_case}
In this section, we prove the free unweighted and weighted decay
estimates and identify the leading low-energy sine coefficient.
For $V=0$, we have $\sqrt H=-\Delta$, and Stone's formula gives
 \begin{align}\label{free-stone_cos}
    \cos (t\Delta)(x,y) & =\frac{2}{\pi i} \int_0^{\infty} \lambda^3 \cos (t \lambda^2)\left[R_0^{+}(\lambda^4)-R_0^{-}(\lambda^4)\right](x,y)  d \lambda,\\
    \label{free-stone_sin}
     \frac{\sin (t \Delta)}{t\Delta}(x,y) & =\frac{2}{\pi i t} \int_0^{\infty} \lambda \sin (t \lambda^2)\left[R_0^{+}(\lambda^4)-R_0^{-}(\lambda^4)\right](x,y) d \lambda.
     \end{align}
The kernels of $\cos(t\Delta)$ and $\sin(t\Delta)/(t\Delta)$ can be written explicitly as
\begin{align}\label{eq:free_kernels}
K_t^{\cos,0}(x,y) = \frac{1}{4\pi t} \sin \frac{|x-y|^2}{4t},\quad
K_t^{\sin,0}(x,y) =\frac{1}{4\pi |t|}\int_{\frac{|x-y|^2}{4|t|}}^{\infty}\frac{\sin u}{u}\,du .
\end{align}
Indeed, the free Schr\"{o}dinger kernel
$$e^{it\Delta}(x,y) = \frac{1}{4\pi i t} \exp\bigl(i\frac{|x-y|^2}{4t}\bigr)$$
and the identity $\cos(t\Delta) = (e^{it\Delta}+e^{-it\Delta})/2$ yield $K_t^{\cos,0}(x,y)$. The identity $\sin(t\Delta)/(t\Delta)=t^{-1}\int_0^t\cos(s\Delta)\,ds$ and the substitution $u=|x-y|^2/(4s)$ give the sine kernel for $x\ne y$; its continuous extension gives the value on the diagonal.

These formulas give the following sharp bounds.

\begin{proposition}\label{free_case}
The free  operators satisfy the following sharp bounds:
\begin{enumerate}
\item[(i)] In the unweighted space $L^1 \to L^{\infty}$:
\[
\|\cos (t \Delta) \|_{L^1 \to L^{\infty}} \sim |t|^{-1}, \quad
\Bigl\|\frac{\sin (t \Delta)}{t\Delta}\Bigr\|_{L^1 \to L^{\infty}} \sim |t|^{-1}.
\]
\item[(ii)] In the weighted space $L^1_\omega \to L^\infty_{-\omega}$:
\[
\|\cos (t \Delta) \|_{L^1_\omega \to L^{\infty}_{-\omega}} \sim |t|^{-1}(\log (2+|t|))^{-1}, \quad
\Bigl\|\frac{\sin (t \Delta)}{t\Delta} \Bigr\|_{L^1_\omega \to L^{\infty}_{-\omega}} \sim |t|^{-1}.
\]
\end{enumerate}
\end{proposition}
\begin{proof}
The unweighted bounds follow  trivially by taking the supremum over the explicit kernels in \eqref{eq:free_kernels}. 

For the
weighted cosine bound, set $\rho=|x-y|$ and use
$\omega(x)\omega(y)\ge\omega(0)\omega(\rho/2)$ to obtain
\begin{equation}\label{free-weight}
	\|\cos(t\Delta)\|_{L^1_\omega\to L^\infty_{-\omega}}
	\lesssim \sup_{\rho \ge 0} \frac{|\sin(\rho^2/(4t))|}{|t|\,\omega(\rho/2)}
	= \sup_{u\ge 0} \frac{|\sin u|}{|t|\log\bigl(2+\sqrt{|t| u}\bigr)}.
	\end{equation}
	For $u>1$, we use $|\sin u|\le1$ and $\log(2+\sqrt{|t|u}) \ge \log(2 + \sqrt{|t|}) \sim \log (2 + |t|)$. For $0 \le u \le 1$, using $|\sin u|\le u$ and the substitution $v=\sqrt{|t|u}$ yields
	\[
	\sup_{0\le u\le 1} \frac{|\sin u|}{\log(2+\sqrt{|t|u})}
	\le \sup_{0\le v\le \sqrt{|t|}} \frac{v^2/|t|}{\log(2+v)} = \frac{1}{\log(2+\sqrt{|t|})} \sim \frac{1}{\log (2 + |t|)}.
	\]
	Inserting these into \eqref{free-weight} establishes the upper bound $\|\cos(t\Delta)\|_{L^1_\omega\to L^\infty_{-\omega}} \lesssim |t|^{-1}\bigl(\log (2 + |t|)\bigr)^{-1}$.
The lower bound follows by taking $y=0$ and
$|x|=\sqrt{2\pi|t|}$.
For the weighted sine bound, the unweighted upper
bound and $K_t^{\sin,0}(0,0)=1/(8|t|)$ give the sharp weighted bound $|t|^{-1}$.
\end{proof}

\begin{remark}\label{free_cos_higher_dim}{\rm 
For a general dimension $n$, the kernel of $\cos(t\Delta)$ admits the explicit form
\[
\cos(t\Delta)(x,y) = \frac{1}{(4\pi |t|)^{n/2}} \cos\!\left(\frac{|x-y|^2}{4|t|} - \frac{n\pi}{4}\right),
\]
which trivially yields the sharp decay estimate $\|\cos(t\Delta)\|_{L^1\to L^\infty} \lesssim |t|^{-n/2}$. 
In the weighted setting, the enhanced decay bound
\[
\|\cos(t\Delta)\|_{L^1_\omega \to L^\infty_{-\omega}} \sim |t|^{-n/2} \bigl(\log(2+|t|)\bigr)^{-1}
\]
holds if and only if $n \equiv 2 \pmod{4}$ (e.g., $n = 2, 6, 10, \dots$). For all other dimensions, the logarithmic gain is absent, yielding purely polynomial decay. 
}\end{remark}

\begin{remark}\label{rmk:free_case}
{\rm
For $t\ne0$ and $j=1,2$, the cutoff propagators satisfy
\[
\begin{aligned}
&\|\cos(t\Delta)\chi_j(\Delta^2)\|_{L^1\to L^\infty}
+\bigg\|\frac{\sin(t\Delta)}{t\Delta}\chi_j(\Delta^2)\bigg\|_{L^1\to L^\infty}
\lesssim |t|^{-1},\\
&\|\cos(t\Delta)\chi_j(\Delta^2)\|_{L^1_\omega\to L^\infty_{-\omega}}
\lesssim \frac{1}{|t|\log(2+|t|)}.
\end{aligned}
\]
Indeed,
$\chi_1(\Delta^2)$ has a Schwartz convolution kernel.
The inequality
$\omega(y+z)\lesssim\omega(y)\langle z\rangle$
then shows that $\chi_1(\Delta^2)$ is bounded on both $L^1$
and $L^1_\omega$.
Since $\chi_2(\Delta^2)=I-\chi_1(\Delta^2)$,
both cutoff operators have these boundedness properties.
The stated estimates follow from Proposition~\ref{free_case}.

For the high-energy sine, we further have
\begin{align}\label{re-3.3-sine}
\left\|\frac{\sin(t\Delta)}{t\Delta}\chi_2(\Delta^2)\right\|_{L^1\to L^\infty}
\lesssim |t|^{-2}.
\end{align}
To prove this, expand the sine in \eqref{free-stone_sin} into
exponentials. The amplitude
$\lambda^{-2}\widetilde\chi_2(\lambda)
\widetilde R^\pm(\lambda|x-y|)$
satisfies \eqref{condit 1} with $r=|x-y|$ and $h=1$,
by  the fact that
$\widetilde\chi_2$ vanishes near zero.
Thus \eqref{rem-1} bounds each oscillatory integral by $C|t|^{-1}$.
The factor $t^{-1}$ in \eqref{free-stone_sin} gives the claimed
$|t|^{-2}$ bound.
}
\end{remark}

\begin{proposition}\label{prop_free_sin}
For $|t| \ge 2,$ the low-frequency free sine kernel admits the asymptotic expansion:
$$
\begin{aligned}
\frac{\sin(t\Delta)}{t\Delta}\chi_1(\Delta^2)(x,y)
=\frac{1}{8 |t|}+O\left(\frac{\omega(x)\omega(y)}{|t|\log|t|}\right).
\end{aligned}
$$
\end{proposition}
\begin{proof}
By \eqref{eq:free_kernels} and the Dirichlet integral,
\begin{equation}\label{free-whole}
\frac{\sin(t\Delta)}{t\Delta}(x,y)
=\frac{1}{8|t|}-\frac{1}{4\pi|t|}
\int_0^{|x-y|^2/(4|t|)}\frac{\sin u}{u}\,du.
\end{equation}
Since $|\sin u/u|\le1$ and
$|x-y|^2\lesssim\langle x\rangle^2\langle y\rangle^2$,
\[
\left|\frac{\sin(t\Delta)}{t\Delta}(x,y)-\frac{1}{8|t|}\right|
\lesssim\frac{\langle x\rangle^2\langle y\rangle^2}{|t|^2}.
\]
Using $\chi_1+\chi_2=1$ and the high-energy estimate
\eqref{re-3.3-sine}, we obtain
\[
\left|\Big[\frac{\sin(t\Delta)}{t\Delta}
\chi_1(\Delta^2)\Big](x,y)-\frac{1}{8|t|}\right|
\lesssim\frac{\langle x\rangle^2\langle y\rangle^2}{|t|^2}.
\]
The low-energy sine bound in Remark~\ref{rmk:free_case} also bounds
this same difference by $C|t|^{-1}$. Hence
\[
\left|\Big[\frac{\sin(t\Delta)}{t\Delta}
\chi_1(\Delta^2)\Big](x,y)-\frac{1}{8|t|}\right|
\lesssim\frac1{|t|}
\min\left(1,\frac{\langle x\rangle^2\langle y\rangle^2}{|t|}\right).
\]
For fixed $N,\delta>0$ and $|t|\ge2$, we use the elementary inequality
\begin{equation}\label{eq:weight-interpolation}
\min\left(1,\frac{\langle x\rangle^N\langle y\rangle^N}{|t|^\delta}\right)
\lesssim\frac{\omega(x)\omega(y)}{\log|t|}.
\end{equation}
Taking $N=2$ and $\delta=1$ proves the proposition.
\end{proof}
\section{High energy decay estimates}\label{sec:high-energy}
In this section, we prove the high-energy decay estimates (i.e., Theorem \ref{main_theorem_high}). A finite
resolvent expansion reduces the proof to two kernels, estimated using
the free resolvent bounds and a limiting absorption estimate.
\begin{theorem}\label{main_theorem_high}
Let $H=\Delta^2+V$ with $|V(x)|\lesssim\langle x\rangle^{-4-}$.
Assume that $H$ has no positive embedded eigenvalues. Then
$$
\begin{aligned}
\Bigl\|\cos (t \sqrt{H}) P_{\mathrm{ac}}(H)\chi_2(H)\Bigr\|_{L^1\to L^\infty}
+ \Bigl\|\frac{\sin(t\sqrt{H})}{t\sqrt{H}}P_{\mathrm{ac}}(H)\chi_2(H)\Bigr\|_{L^1\to L^\infty}
\lesssim |t|^{-1}.
\end{aligned}
$$
Moreover, for $|t|\ge 2$,
$$
\begin{aligned}
 \Bigl\|\cos (t \sqrt{H}) P_{\mathrm{ac}}(H)\chi_2(H)\Bigr\|_{L^1_\omega\to L^\infty_{-\omega}}
\lesssim ({|t|\log |t|})^{-1}, \quad
\Bigl\|\frac{\sin(t\sqrt{H})}{t\sqrt{H}}P_{\mathrm{ac}}(H)\chi_2(H)\Bigr\|_{L^1\to L^\infty}
\lesssim |t|^{-2}.
\end{aligned}
$$
\end{theorem}

By Euler's formula, it suffices to estimate the following exponential propagators
with $\gamma=0$ for the cosine and $\gamma=-1$ for the sine.
Stone's formula gives
\begin{equation}\label{stone_high_1}
H^{\gamma/2}e^{-it\sqrt H}P_{\mathrm{ac}}(H)\chi_2(H)
=\frac{2}{\pi i}\int_0^\infty e^{-it\lambda^2}\lambda^{3+2\gamma}
\widetilde\chi_2(\lambda)
[R_V^+(\lambda^4)-R_V^-(\lambda^4)]\,d\lambda,
\end{equation}
where $\widetilde\chi_2(\lambda)=\chi_2(\lambda^4)$ is supported on 
$|\lambda|\geq\lambda_0^{1/4}$.
Iterating the resolvent identity gives
\[
\begin{aligned}
R_V^\pm(\lambda^4)
={}&R_0^\pm(\lambda^4)
-R_0^\pm(\lambda^4)VR_0^\pm(\lambda^4)+R_0^\pm(\lambda^4)VR_V^\pm(\lambda^4)VR_0^\pm(\lambda^4).
\end{aligned}
\]
Remark~\ref{rmk:free_case} gives the required bounds for the free term.
It remains to estimate the kernels
\[
\begin{aligned}
L_1^\pm(\gamma,t,x,y)
&:=\int_0^\infty e^{-it\lambda^2}\lambda^{3+2\gamma}
\widetilde\chi_2(\lambda)
[R_0^\pm(\lambda^4)VR_0^\pm(\lambda^4)](x,y)\,d\lambda,\\
L_2^\pm(\gamma,t,x,y)
&:=\int_0^\infty e^{-it\lambda^2}\lambda^{3+2\gamma}
\widetilde\chi_2(\lambda)
[R_0^\pm(\lambda^4)VR_V^\pm(\lambda^4)VR_0^\pm(\lambda^4)](x,y)\,d\lambda.
\end{aligned}
\]

\begin{proposition}\label{high_1}
Assume that $|V(x)| \lesssim \left \langle  x\right \rangle^{-4-}$. Then for $\gamma\in \{-1,0\},$
$$
\sup _{x, y \in \mathbb{R}^2}\left|L_{1}^{\pm}(\gamma,t,x,y)\right| \lesssim |t|^{-1}, \ \ \ \
\left|L_{1}^{\pm}(\gamma,t,x,y)\right| \lesssim |t|^{-2} \langle x \rangle^2 \langle y \rangle^2.
$$
Moreover, $|L_1^\pm(-1,t,x,y)|\lesssim 1$  uniformly in $t,x,y.$
\end{proposition}
\begin{proof}
By \eqref{free-widetilde_R},
$R_0^\pm(\lambda^4)(x,y)=\lambda^{-2}e^{\pm i\lambda|x-y|}
\widetilde R^\pm(\lambda|x-y|)$. Hence
\[
L_1^\pm(\gamma,t,x,y)=\int_{\mathbb R^2}K_1^\pm(t,x,y,y_1)V(y_1)\,dy_1,
\]
where, with $r=|x-y_1|+|y-y_1|$,
\begin{equation}\label{eq:high_1}
\begin{aligned}
K_1^\pm=\int_0^\infty e^{-it\lambda^2}\lambda e^{\pm i\lambda r}
\mathcal E^\pm(\lambda,x,y,y_1)\,d\lambda,
\qquad
\mathcal E^\pm=\lambda^{-2+2\gamma}\widetilde\chi_2(\lambda)
\widetilde R^\pm(\lambda|x-y_1|)\widetilde R^\pm(\lambda|y-y_1|).
\end{aligned}
\end{equation}
For $0\le\ell\le2$, \eqref{free-widetilde_R} gives
\[
\bigl|\partial_\lambda^\ell\widetilde R^\pm(\lambda\rho)\bigr|
\lesssim\lambda^{-\ell}\langle\lambda\rho\rangle^{-1/2},
\qquad \rho\ge0.
\]
Leibniz's rule and
$\langle\lambda|x-y_1|\rangle\langle\lambda|y-y_1|\rangle
\gtrsim\langle\lambda r\rangle$ therefore imply
\[
|\partial_\lambda^\ell\mathcal E^\pm|
\lesssim\lambda^{-2+2\gamma-\ell}\langle\lambda r\rangle^{-1/2}
\lesssim\lambda^{-\ell}\langle\lambda r\rangle^{-1/2},
\qquad \lambda\ge\lambda_0^{1/4}.
\]
Estimate~\eqref{rem-1} now gives $|K_1^\pm|\lesssim|t|^{-1}$.
Since $V\in L^1$, integration in $y_1$ proves the uniform decay
bound. For $\gamma=-1$, the absolute bound
$|\lambda\mathcal E^\pm|\lesssim\lambda^{-3}$ also gives
$|L_1^\pm(-1,t,x,y)|\lesssim1$.

For the weighted bound, include $e^{\pm i\lambda r}$ in the amplitude $\mathcal E^\pm$.
Since $r\lesssim\langle x\rangle\langle y\rangle\langle y_1\rangle$,
Leibniz's rule yields
\[
\left|\partial_\lambda^\ell
\bigl(e^{\pm i\lambda r}\mathcal E^\pm\bigr)\right|
\lesssim\lambda^{-2+2\gamma}
\langle x\rangle^\ell\langle y\rangle^\ell\langle y_1\rangle^\ell,
\qquad 0\le\ell\le2.
\]
The amplitude vanishes near zero because of $\widetilde\chi_2$.
At infinity, both $e^{\pm i\lambda r}\mathcal E^\pm$ and
$\lambda^{-1}\partial_\lambda(e^{\pm i\lambda r}\mathcal E^\pm)$
tend to zero. Thus two integrations by parts give
\begin{equation}\label{eq:high-ibp}
K_1^\pm
=\frac{1}{(2it)^2}\int_0^\infty e^{-it\lambda^2}
\partial_\lambda\left[
\lambda^{-1}\partial_\lambda
\bigl(e^{\pm i\lambda r}\mathcal E^\pm\bigr)\right]d\lambda.
\end{equation}
The derivative bounds above imply
\[
|K_1^\pm|\lesssim |t|^{-2}
\langle x\rangle^2\langle y\rangle^2\langle y_1\rangle^2
\int_{\lambda_0^{1/4}}^\infty\lambda^{-3+2\gamma}\,d\lambda
\lesssim |t|^{-2}
\langle x\rangle^2\langle y\rangle^2\langle y_1\rangle^2.
\]
Finally, $|V(y_1)|\lesssim\langle y_1\rangle^{-4-}$ gives
$\int\langle y_1\rangle^2|V(y_1)|\,dy_1<\infty$, which proves
the weighted estimate for $L_1^\pm$.
\end{proof}

For $L_2^\pm$, we also use the limiting absorption estimate below.
\begin{lemma}[{\cite[Theorem 2.23]{FSY}}]\label{RRR}
Assume that $|V(x)|\lesssim\langle x\rangle^{-\mu}$ with
$\mu>k+1$, $k\in\mathbb N_0$, and that $H$ has no positive eigenvalues.
For every $\sigma>k+\tfrac12$, the map
$\lambda\mapsto R_V^\pm(\lambda)$ is $C^k$ on $(0,\infty)$ with
values in $\mathbb B(L^{2,\sigma},L^{2,-\sigma})$. Moreover,
$$
\left\|\partial_\lambda^{k} R_V^{ \pm}(\lambda)\right\|_{L^{2,\sigma}(\mathbb{R}^2) \rightarrow L^{2,-\sigma}(\mathbb{R}^2)}=O\left(|\lambda|^{-\frac{3(k+1)}{4}}\right),~\lambda \rightarrow+\infty.
$$
\end{lemma}

\begin{proposition}\label{high_2}
Assume that $|V(x)|\lesssim\langle x\rangle^{-4-}$ and that $H$ has
no positive embedded eigenvalues. Then, for $\gamma=-1,0$,
$$
 \sup _{x, y \in \mathbb{R}^2}\left|L_{2}^{\pm}(\gamma,t,x,y)\right| \lesssim \langle t\rangle^{-1}, \ \ \ \
\left|L_{2}^{\pm}(\gamma,t,x,y)\right| \lesssim |t|^{-2} \langle x \rangle^2 \langle y \rangle^2.
$$
\end{proposition}
\begin{proof}
Factoring out the spatial phases from the outer free resolvent
kernels gives
\[
[R_0^\pm(\lambda^4)VR_V^\pm(\lambda^4)VR_0^\pm(\lambda^4)](x,y)
=\lambda^{-4}e^{\pm i\lambda(|x|+|y|)}
E_1^\pm(\lambda,x,y),
\]
where
\[
\begin{aligned}
E_1^\pm(\lambda,x,y)=\Big\langle
& R_V^\pm(\lambda^4)\bigl[
V e^{\pm i\lambda(|\cdot-y|-|y|)}
\widetilde R^\pm(\lambda|\cdot-y|)\bigr], V e^{\mp i\lambda(|\cdot-x|-|x|)}
\widetilde R^\mp(\lambda|\cdot-x|)\Big\rangle.
\end{aligned}
\]
Consequently, with $r=|x|+|y|$,
\begin{equation}\label{eq:high_2}
L_2^\pm=\int_0^\infty e^{-it\lambda^2}\lambda e^{\pm i\lambda r}
\mathcal E^\pm(\lambda,x,y)\,d\lambda,
\qquad
\mathcal E^\pm=\lambda^{-2+2\gamma}\widetilde\chi_2(\lambda)E_1^\pm.
\end{equation}

We estimate derivatives of the three factors defining $E_1^\pm$.
For $\lambda\ge\lambda_0^{1/4}$, the outer factors satisfy
\begin{equation}\label{eq:inequality_1}
\left|\partial_\lambda^a\left[
e^{\pm i\lambda(|z-x|-|x|)}\widetilde R^\pm(\lambda|z-x|)
\right]\right|
\lesssim\lambda^{1/2}\langle\lambda x\rangle^{-1/2}
\langle z\rangle^{a+1/2},\qquad 0\le a\le2.
\end{equation}
Indeed, $\bigl||z-x|-|x|\bigr|\le|z|$, while
\eqref{free-widetilde_R} bounds the differentiated profile by
$C\lambda^{-a}\langle\lambda(z-x)\rangle^{-1/2}$. We also use
\[
\langle\lambda(z-x)\rangle^{-1/2}
\lesssim\lambda^{1/2}\langle\lambda x\rangle^{-1/2}\langle z\rangle^{1/2}.
\]
The chain rule and Lemma~\ref{RRR} give
\begin{equation}\label{eq:inequality_2}
\|\partial_\lambda^jR_V^\pm(\lambda^4)\|_{L^{2,\sigma}\to L^{2,-\sigma}}
\lesssim\lambda^{-3},\qquad \sigma>j+\tfrac12,\quad 0\le j\le2.
\end{equation}

Write $|V(z)|\lesssim\langle z\rangle^{-\mu}$ with $\mu>4$ and
choose $0<\varepsilon<\mu-4$. In each Leibniz term, let $j$
derivatives fall on $R_V^\pm$ and $a,b$ on the outer factors, so
$a+b+j=\ell\le2$. With $\sigma=j+\tfrac12+\varepsilon$,
the weighted $L^2$ pairing and
\eqref{eq:inequality_1}--\eqref{eq:inequality_2} bound each such term by
\[
C\lambda^{-2}\langle\lambda x\rangle^{-1/2}
\langle\lambda y\rangle^{-1/2}
\|\langle z\rangle^{a+\sigma+1/2}V\|_2
\|\langle z\rangle^{b+\sigma+1/2}V\|_2.
\]
Both norms are finite: their exponents are at most
$3+\varepsilon<\mu-1$, the $L^2$ integrability threshold in
dimension two. Summing the finitely many Leibniz terms gives
\[
|\partial_\lambda^\ell E_1^\pm|
\lesssim\lambda^{-2}\langle\lambda x\rangle^{-1/2}
\langle\lambda y\rangle^{-1/2},\qquad 0\le\ell\le2.
\]
It follows that
\[
|\partial_\lambda^\ell\mathcal E^\pm|
\lesssim\lambda^{-4+2\gamma}\langle\lambda r\rangle^{-1/2}
\lesssim\lambda^{-\ell}\langle\lambda r\rangle^{-1/2}.
\]
Estimate~\eqref{rem-1} gives $|L_2^\pm|\lesssim|t|^{-1}$.
Since $|\lambda\mathcal E^\pm|\lesssim\lambda^{-3+2\gamma}$ is
integrable on the high-energy support, we also have
$|L_2^\pm|\lesssim1$. This proves the uniform $\langle t\rangle^{-1}$ bound.

For the weighted estimate, the derivative bounds for $\mathcal E^\pm$
and $r\lesssim\langle x\rangle\langle y\rangle$ give
\[
\left|\partial_\lambda^\ell
\bigl(e^{\pm i\lambda r}\mathcal E^\pm\bigr)\right|
\lesssim
\lambda^{-4+2\gamma}\langle x\rangle^\ell\langle y\rangle^\ell,
\qquad 0\le\ell\le2.
\]
We then integrate twice by parts in \eqref{eq:high_2},
as in \eqref{eq:high-ibp}.
The amplitude $e^{\pm i\lambda r}\mathcal E^\pm$ vanishes near zero.
Both this amplitude and
$\lambda^{-1}\partial_\lambda(e^{\pm i\lambda r}\mathcal E^\pm)$
tend to zero at infinity, so the boundary terms vanish. Hence
\[
|L_2^\pm|\lesssim |t|^{-2}\langle x\rangle^2\langle y\rangle^2
\int_{\lambda_0^{1/4}}^\infty\lambda^{-5+2\gamma}\,d\lambda
\lesssim |t|^{-2}\langle x\rangle^2\langle y\rangle^2,
\]
which completes the proof.
\end{proof}
\begin{proof}[Proof of Theorem~\ref{main_theorem_high}]
Propositions~\ref{high_1} and~\ref{high_2} give, for $j=1,2$
and $t\ne0$,
\begin{align}\label{L-high}
&|L_j^\pm(0,t,x,y)|
\lesssim
\min\bigl(|t|^{-1},|t|^{-2}\langle x\rangle^2\langle y\rangle^2\bigr),\quad
|t^{-1}L_j^\pm(-1,t,x,y)|
\lesssim |t|^{-1}\langle t\rangle^{-1}.
\end{align}
Applying \eqref{eq:weight-interpolation} with $N=2$ and $\delta=1$
to the first bound yields
\[
|L_j^\pm(0,t,x,y)|
\lesssim \frac{\omega(x)\omega(y)}{|t|\log|t|},
\qquad |t|\ge2.
\]
Euler's formula gives the unweighted and weighted cosine bounds
for both correction terms. The  estimate of $t^{-1}L_j^\pm(-1,t,x,y)$ in \eqref{L-high} gives the sine estimates
of order $|t|^{-1}$ for $t\ne0$ and $|t|^{-2}$ for $|t|\ge2$.
Adding the free estimates from Remark~\ref{rmk:free_case}
completes the proof.
\end{proof}

\section{Low energy estimates in the regular and first-kind resonance cases}\label{sec:regular_first}
With  the high-energy estimates in Theorem~\ref{main_theorem_high} established, the proofs of our main results (Theorems~\ref{main_theorem-1}, \ref{main_theorem-2}, \ref{main_theorem-31}, and \ref{main_theorem-32}) are reduced to the low-energy part. 
We begin with a
general strategy, valid for all threshold resonances, and then
specialize in this section to the cases where zero is either a
regular point or a first-kind resonance.

Recall Stone's formulas for the low-energy spectral multipliers:
\begin{align}
	\cos(t\sqrt{H})P_{\mathrm{ac}}(H)\chi_1(H)
	&= \frac{2}{\pi i} \int_0^{\infty} \cos(t\lambda^2)\,\lambda^{3}\,\widetilde{\chi}_1(\lambda)
	\bigl[R_V^{+}(\lambda^4)-R_V^{-}(\lambda^4)\bigr] \, d\lambda,   \label{stone_low_1} \\[4pt]
	\frac{\sin(t\sqrt{H})}{t\sqrt{H}}P_{\mathrm{ac}}(H)\chi_1(H)
	&= \frac{2}{\pi i t} \int_0^{\infty} \sin(t\lambda^2)\,\lambda\,\widetilde{\chi}_1(\lambda)
	\bigl[R_V^{+}(\lambda^4)-R_V^{-}(\lambda^4)\bigr] \, d\lambda, \label{stone_low_2}
\end{align}
where \(\widetilde{\chi}_1(\lambda) := \chi_1(\lambda^4)\) is supported on  \(|\lambda|\le (2\lambda_0)^{1/4}\).
Substituting \eqref{eq:M_inverse} into \eqref{id-RV} gives
\begin{equation}\label{RV--}
	R_V^{\pm}(\lambda^4) = R_0^\pm(\lambda^{4}) - R_0^\pm(\lambda^{4})v \bigg( \sum_{0 \leq \alpha, \beta \leq \mathbf{k}+2} \lambda^{2-k_\alpha-k_\beta} Q_\alpha \mathcal{M}_{\alpha, \beta}^{\pm}(\lambda) Q_\beta \bigg) vR_0^\pm(\lambda^{4}).
\end{equation}
Equations~\eqref{stone_low_1}--\eqref{stone_low_2} then give
\begin{align*}
	\cos(t\sqrt{H})P_{\mathrm{ac}}(H)\chi_1(H)
	&= \cos(t\Delta)\chi_1(\Delta^2)-
	\sum_{B\in \mathfrak{P}_{\mathbf{k}}}T_{\mathcal{K}_{B}^+-\mathcal{K}_{B}^-},\\
	\frac{\sin(t\sqrt{H})}{t\sqrt{H}}P_{\mathrm{ac}}(H)\chi_1(H)
	&= \frac{\sin(t\Delta)}{t\Delta}\chi_1(\Delta^2)-\sum_{B\in \mathfrak{P}_{\mathbf{k}}}T_{\mathcal{N}_{B}^+-\mathcal{N}_{B}^-},
\end{align*}
where $T$ denotes the integral operator with the indicated kernel.
When the individual integrals converge, the kernels are
\begin{align}
	\label{K_1}
	& \mathcal{K}_{B}^\pm(t,x,y)=\frac{2}{\pi i}\int_0^{\infty} \cos(t \lambda^2)\lambda \widetilde{\chi}_1(\lambda) \lambda^{4-k_\alpha-k_\beta}\bigl\langle  B^\pm(\lambda)Q_\beta vR_0^\pm(\lambda^{4})(\cdot,y), Q_\alpha vR_0^\mp(\lambda^{4})(\cdot,x)\bigr\rangle \, d \lambda, \\
	\label{Lambda_1}
	& \mathcal{N}_{B}^\pm(t,x,y)=\frac{2}{\pi it}\int_0^{\infty} \sin(t\lambda^2)\lambda^{-1}\widetilde{\chi}_1(\lambda)\lambda^{4-k_\alpha-k_\beta}
	\bigl\langle  B^\pm(\lambda)Q_\beta vR_0^\pm(\lambda^{4})(\cdot,y), Q_\alpha vR_0^\mp(\lambda^{4})(\cdot,x)\bigr\rangle \, d \lambda,
\end{align} 
and  $B^\pm(\lambda) \in \mathfrak{P}_{\mathbf{k}}:= \{ \mathcal{M}_{\alpha, \beta}^{\pm}(\lambda) \mid \alpha,\beta\in\mathbb{Z},\ 0 \le \alpha,\beta \le \mathbf{k}+2 \}.$
In all cases, the kernel differences are defined by subtracting
the displayed integrands before integration.
Combining these with the free low-energy estimates from Remark~\ref{rmk:free_case}, 
we reduce the low-energy analysis to  estimating  the kernel differences
\begin{equation*}
	\big(\mathcal{K}_{B}^+ - \mathcal{K}_{B}^-\big)(t,x,y) \quad \text{and} \quad \big(\mathcal{N}_{B}^+ - \mathcal{N}_{B}^-\big)(t,x,y),\quad  B^\pm(\lambda) \in \mathfrak{P}_{\mathbf{k}}.
\end{equation*}

We now specialize to \textbf{the regular and first-kind resonance case},
i.e. $\mathbf k\in\{0,1\}$. The goal of this section is to prove the following  Theorem~\ref{main_theorem_low_1}.

\begin{theorem}\label{main_theorem_low_1}
	Let $H = \Delta^2 + V$ and $|V(x)| \lesssim \langle x \rangle^{-11-}$.
	Assume that $H$ has no positive embedded eigenvalues and that zero is a regular point or a first-kind resonance of $H$. Then
	\begin{equation}\label{eq:thm_unweighted}
		\left\|\cos (t \sqrt{H}) P_{\mathrm{ac}}(H)\chi_1(H)\right\|_{L^1 \rightarrow L^{\infty}}+
		\bigg \|\frac{\sin (t \sqrt{H})}{t\sqrt{H}} P_{\mathrm{ac}}(H)\chi_1(H)\bigg \|_{L^1 \rightarrow L^{\infty}} \lesssim \langle t \rangle^{-1}.
	\end{equation}
	For $|t|\ge2$, the weighted estimate is
	\begin{equation}\label{eq:thm_weighted}
		\left\|\cos (t \sqrt{H}) P_{\mathrm{ac}}(H)\chi_1(H)\right\|_{L^1_\omega \rightarrow L^{\infty}_{-\omega}}
		+\bigg\|\frac{\sin (t \sqrt{H})}{t\sqrt{H}} P_{\mathrm{ac}}(H)\chi_1(H)\bigg\|_{L^1_\omega \rightarrow L^{\infty}_{-\omega}} \lesssim \frac{1}{ |t| \log |t|}.
	\end{equation}
\end{theorem}

We first bound the blocks $0\le\alpha,\beta\le3$ uniformly in space.
For the weighted estimate, we then separate the $(0,0)$ block, whose
leading sine term cancels the free contribution.

\subsection{\texorpdfstring{The $L^1 \rightarrow L^{\infty}$ estimates}{Unweighted estimates}}\label{subsec:regular_first_1}
We use the inverse bounds \eqref{M0-1} and the projected-resolvent
estimates of Lemma~\ref{lemma_projection}.
\begin{proposition}\label{prop_K}
Let $B^\pm(\lambda)=\mathcal M_{\alpha,\beta}^\pm(\lambda)$
with $0\le\alpha,\beta\le3$. Then
\[
\sup_{x,y\in\mathbb R^2}|\mathcal K_B^\pm(t,x,y)|
+\sup_{x,y\in\mathbb R^2}|\mathcal N_B^\pm(t,x,y)|
\lesssim\langle t\rangle^{-1}.
\]
\end{proposition}

\begin{proof}
Put $r=|x|+|y|$. For $0\le\alpha\le3$,
Lemma~\ref{lemma_projection}(i) gives
\[
(Q_\alpha vR_0^\pm(\lambda^4))(\cdot,y)
=\lambda^{k_\alpha-2}\Omega_{\alpha,\pm}(\lambda,\cdot,y),
\]
since $1+\delta_{\alpha0}=2-k_\alpha$.
Thus the powers of $\lambda$ from the two projected resolvents
cancel the factor $\lambda^{4-k_\alpha-k_\beta}$ in
\eqref{K_1} and \eqref{Lambda_1}.
For the $(3,3)$ block, we instead use
$Q_3vR_0^\pm=\mathcal J_{3,\pm}$ from
Lemma~\ref{lemma_projection}(ii), leaving the factor $\lambda^2$.
Define
\begin{equation}\label{rf}
\mathcal E^\pm(\lambda,x,y)
=e^{\mp i\lambda r}\widetilde\chi_1(\lambda)
\begin{cases}
\bigl\langle B^\pm\Omega_{\beta,\pm}(\lambda,\cdot,y),
\Omega_{\alpha,\mp}(\lambda,\cdot,x)\bigr\rangle,
&(\alpha,\beta)\ne(3,3),\\[3pt]
\lambda^2\bigl\langle B^\pm\mathcal J_{3,\pm}(\lambda,\cdot,y),
\mathcal J_{3,\mp}(\lambda,\cdot,x)\bigr\rangle,
&(\alpha,\beta)=(3,3).
\end{cases}
\end{equation}
Then \eqref{K_1} and \eqref{Lambda_1} become
\[
\begin{aligned}
\mathcal K_B^\pm(t,x,y)
&=\frac{2}{\pi i}\int_0^\infty
\cos(t\lambda^2)e^{\pm i\lambda r}
\lambda\mathcal E^\pm(\lambda,x,y)\,d\lambda,\\
\mathcal N_B^\pm(t,x,y)
&=\frac{2}{\pi it}\int_0^\infty
\sin(t\lambda^2)e^{\pm i\lambda r}
\lambda^{-1}\mathcal E^\pm(\lambda,x,y)\,d\lambda.
\end{aligned}
\]
We therefore estimate $\mathcal E^\pm$ to apply
Lemma~\ref{oscillatory}(ii).

In the pairings in \eqref{rf}, we place
$e^{\mp i\lambda|y|}$ in the first argument and
$e^{\pm i\lambda|x|}$ in the second.
These factors produce $e^{\mp i\lambda r}$ under our
inner-product convention and allow us to use the derivative
bounds in \eqref{J} and \eqref{lemma_projection_1}.
Combining these bounds with \eqref{M0-1}, Leibniz's rule and H\"older's inequality give,
for $0\le\ell\le2$,
\[
|\partial_\lambda^\ell\mathcal E^\pm|
\lesssim
\lambda^{-\ell}
\langle\lambda x\rangle^{-1/2}
\langle\lambda y\rangle^{-1/2}
\widetilde\chi_1(\lambda/2)
\begin{cases}
1,&(\alpha,\beta)\ne(3,3),\\
\lambda^2|\log\lambda|^3,&(\alpha,\beta)=(3,3).
\end{cases}
\]
For the $(3,3)$ block, one logarithmic factor comes from
$\mathcal{M}_{3,3}^\pm(\lambda)$ and the other two from the $\mathcal J_{3,\pm}$ factors.
Since $\lambda^2|\log\lambda|^3\lesssim1$ for $0<\lambda\ll1$
and
\[
\langle\lambda x\rangle^{-1/2}
\langle\lambda y\rangle^{-1/2}
\lesssim\langle\lambda r\rangle^{-1/2},
\]
we obtain
\[
|\partial_\lambda^\ell\mathcal E^\pm|
\lesssim
\lambda^{-\ell}\langle\lambda r\rangle^{-1/2}
\widetilde\chi_1(\lambda/2),
\qquad 0\le\ell\le2.
\]
Lemma~\ref{oscillatory}(ii), with $h=1$ and
$(\sigma,\nu)=(0,0)$, now gives
\[
|\mathcal K_B^\pm(t,x,y)|
+|\mathcal N_B^\pm(t,x,y)|
\lesssim\langle t\rangle^{-1}
\]
uniformly in $x,y$, which proves the proposition.
\end{proof}

\subsection{\texorpdfstring{The $L^1_\omega \to L^\infty_{-\omega}$ estimates}{Weighted estimates}} \label{subsec:regular_first_2}
The block $(\alpha,\beta)=(0,0)$ retains the leading term
$a_0^\pm\lambda^{-2}$ in the free resolvent expansion.
For every other block, $Q_\alpha v=0$ or $Q_\beta v=0$,
so this term vanishes in at least one projected resolvent, giving a positive
power of $\lambda$.
We treat these two cases separately.

By  Theorem~\ref{thm:M_inverse} (I), we have for $(\alpha,\beta)\neq(0,0)$:
\begin{align}\label{Mneq00}
\bigl\|\partial_\lambda^\ell\mathcal{M}_{\alpha,\beta}^\pm(\lambda)\bigr\|_{\mathbb{B}(L^2)}
	\lesssim   \lambda^{-\ell}|\log \lambda|,  \quad  \ell = 0,1,2,
\end{align}
and
\begin{equation}\label{eq:M_00}
	\mathcal{M}_{0,0}^\pm(\lambda) =(a_0^\pm)^{-1} \|V\|_{L^1(\mathbb{R}^2)}^{-1}Q_0 + \lambda (-\log \lambda)^{3/2}\Lambda^{\pm}(\lambda).
\end{equation}

For the weighted bounds, \eqref{reso_small} and \eqref{reso-big} give
\begin{equation}\label{eq:expansion}
	R_0^{\pm}(\lambda^4)(x,y) = a_0^{\pm} \lambda^{-2}G_0(x,y) + E_1^{\pm}(\lambda,x,y)=a_0^{\pm} \lambda^{-2}+ E_1^{\pm}(\lambda,x,y),
\end{equation}
where  $
E_1^{\pm}(\lambda,x,y) = a_1^\pm |x-y|^2 + b_0 |x-y|^2 \log(\lambda |x-y|)
+ O_2\bigl(\lambda^{2}|x-y|^{4}\bigr).
$ The bounds $|z|^2|\log|z||\lesssim1+|z|^4$ and
$|\log\lambda|\lesssim\lambda^{-\varepsilon}$ imply, for
$0<\lambda\ll1$ and $\ell=0,1,2$,
\begin{equation}\label{derivative_1}
	\bigl| \partial_\lambda^\ell E_1^\pm(\lambda,x,y)\bigr| \lesssim \lambda^{-\ell-}\langle x \rangle^{4} \langle y \rangle^{4}.
\end{equation}

In this subsection, we set
\begin{equation}\label{eq:def_mathcal_E}
	\mathcal{E}^\pm(\lambda,x,y)
	= \widetilde{\chi}_1(\lambda) \lambda^{4-k_\alpha-k_\beta}
	\bigl\langle B^\pm(\lambda) \bigl( Q_\beta v R_0^{\pm}(\lambda^4) \bigr)(\cdot,y),
	\bigl( Q_\alpha v R_0^{\mp} (\lambda^4)\bigr)(\cdot,x) \bigr\rangle,
\end{equation}
then  the kernels  $\mathcal{K}_{B}^\pm$ and $\mathcal{N}_{B}^\pm$ defined in \eqref{K_1} and \eqref{Lambda_1} can be expressed as
\begin{align}
	\label{K-N-weig}
	& \mathcal{K}_{B}^\pm=\frac{2}{\pi i}\int_0^{\infty} \cos(t \lambda^2)\lambda  \mathcal{E}^\pm(\lambda,x,y)\, d \lambda, \  \ \mathcal{N}_{B}^\pm=\frac{2}{\pi i t}\int_0^{\infty} \sin(t\lambda^2)\lambda^{-1} \mathcal{E}^\pm(\lambda,x,y) \, d \lambda. \end{align}
\begin{proposition}\label{prop_K_weight_1}
	Let $B^\pm(\lambda)=\mathcal{M}_{\alpha,\beta}^\pm(\lambda)$ with $0\le\alpha,\beta\le3$ and $(\alpha,\beta)\ne(0,0)$. Then for $|t|\ge2$,
	\[
	\bigl|\mathcal{K}_{B}^\pm(t,x,y)\bigr| +
	\bigl|\mathcal{N}_{B}^\pm(t,x,y)\bigr| \lesssim \omega(x)\omega(y)(|t|\log|t|)^{-1}.
	\]
\end{proposition}
\begin{proof}
By \eqref{eq:expansion}, \eqref{derivative_1}, and $Q_jv=0$ for $j\ge1$,
\begin{align}\label{R-5.3}
\|\partial_\lambda^\ell(Q_jvR_0^\pm)(\cdot,y)\|_2
\lesssim\langle y\rangle^4
\begin{cases}\lambda^{-2-\ell},&j=0,\\
\lambda^{-\ell-},&j\ge1.
\end{cases}
\end{align}
The decay assumption on $V$ makes $\langle\cdot\rangle^4v\in L^2$.
The prefactor $\lambda^{4-k_\alpha-k_\beta}$ in \eqref{eq:def_mathcal_E}  is  $\lambda^3$ when
exactly one index is zero (i.e., $\alpha=0$ or $\beta=0$) and $\lambda^2$ when both indices are
nonzero. Together with \eqref{Mneq00}, the displayed resolvent
bounds \eqref{R-5.3} therefore give amplitude orders $\lambda^{1-}$ and
$\lambda^{2-}$, respectively. In either case,
\begin{equation}\label{eq:estimate_E}
|\partial_\lambda^\ell\mathcal E^\pm|
\lesssim\lambda^{1/2-\ell}\widetilde\chi_1(\lambda/2)
\langle x\rangle^4\langle y\rangle^4,\qquad 0\le\ell\le2.
\end{equation}
Lemma~\ref{oscillatory}(ii), with $r=0$ and $(\sigma,\nu)=(-1/2,0)$,
therefore gives
\[
|\mathcal K_B^\pm|+|\mathcal N_B^\pm|
\lesssim |t|^{-5/4}\langle x\rangle^4\langle y\rangle^4.
\]
Combining this with Proposition~\ref{prop_K} and applying
\eqref{eq:weight-interpolation} with $N=4$, $\delta=1/4$, yields
\[
|\mathcal K_B^\pm|+|\mathcal N_B^\pm|
\lesssim\omega(x)\omega(y)(|t|\log|t|)^{-1}.
\]
\end{proof}

\begin{proposition}\label{prop_K_weight_2}
	Let $B^\pm(\lambda)=\mathcal{M}_{0,0}^\pm(\lambda)$. Then for $|t| \ge 2$,
	\begin{align}\label{K-N-cri}
		\bigl|\mathcal{K}_{B}^\pm(t,x,y)\bigr|  \lesssim \frac{\omega(x)\omega(y)}{|t|\log|t|},  \quad
		(\mathcal{N}_{B}^+ - \mathcal{N}_{B}^-)(t,x,y)= \frac{1}{8|t|} +O\left(\frac{\omega(x)\omega(y)}{|t|\log|t|}\right).
	\end{align}
\end{proposition}

\begin{proof}
Inserting \eqref{eq:M_00} and \eqref{eq:expansion} into
\eqref{eq:def_mathcal_E}, and using $Q_0v=v$, gives
\[
\mathcal E^\pm(\lambda,x,y)
=a_0^\pm\widetilde\chi_1(\lambda)+\mathcal E_{\rm rem}^\pm(\lambda,x,y).
\]
Indeed, the product of the three principal factors has coefficient
$(a_0^\pm)^2(a_0^\pm)^{-1}\|V\|_1^{-1}\|v\|_2^2=a_0^\pm$.
Every other product contains either the remainder of
$\mathcal M_{0,0}^\pm$ or a factor $E_1^\pm$. Thus, for $\ell\le2$,
\[
|\partial_\lambda^\ell\mathcal E_{\rm rem}^\pm|
\lesssim\lambda^{1/2-\ell}\widetilde\chi_1(\lambda/2)
\langle x\rangle^4\langle y\rangle^4.
\]
Lemma~\ref{oscillatory}(ii), with $r=0$ and
$(\sigma,\nu)=(-1/2,0)$, bounds the cosine and sine integrals of
this amplitude by $C|t|^{-5/4}\langle x\rangle^4\langle y\rangle^4$.

For the principal cosine integral,  by integrating by parts, we have
	\[
	\int_0^\infty \cos(t\lambda^2)\lambda\widetilde{\chi}_1(\lambda)\,d\lambda  = O(|t|^{-N}),  \ \  \forall N>0.
	\]
For the principal sine integral, the Dirichlet integral and one
integration by parts on the cutoff tail give
\begin{equation}\label{eq:integral_sin}
\frac1t\int_0^\infty\sin(t\lambda^2)\lambda^{-1}
\widetilde\chi_1(\lambda)\,d\lambda
=\frac{\pi}{4|t|}+O(|t|^{-2}).
\end{equation}
Since $a_0^+-a_0^-=i/4$, the factor $2/(\pi i)$ in
\eqref{K-N-weig} yields
\begin{equation}\label{asyp-N}
\mathcal N_B^+-\mathcal N_B^-
=\frac1{8|t|}+O(|t|^{-5/4}\langle x\rangle^4\langle y\rangle^4).
\end{equation}
We have therefore proved, for $|t|\ge2$,
\[
|\mathcal K_B^\pm(t,x,y)|+
\left|(\mathcal N_B^+-\mathcal N_B^-)(t,x,y)-\frac1{8|t|}\right|
\lesssim |t|^{-5/4}\langle x\rangle^4\langle y\rangle^4.
\]
Proposition~\ref{prop_K} bounds this same left-hand side by
$C|t|^{-1}$. Applying \eqref{eq:weight-interpolation} with
$N=4$ and $\delta=1/4$ proves both assertions.
\end{proof}
\begin{proof}[Proof of Theorem~\ref{main_theorem_low_1}]
Proposition~\ref{prop_K} and the $\langle t\rangle^{-1}$ bounds
for the free low-energy kernels give \eqref{eq:thm_unweighted}. 
For the weighted sine estimate, the
$(0,0)$ term in \eqref{K-N-cri} cancels the free coefficient
$1/(8|t|)$ from Proposition~\ref{prop_free_sin}.
Propositions~\ref{prop_K_weight_1} and~\ref{prop_K_weight_2}
bound the remaining terms. The same propositions and the free
cosine bound in Remark~\ref{rmk:free_case} give the weighted
cosine estimate. 
\end{proof}

\begin{remark}\label{rmk:unweighted-sharpness}
{\rm
The unweighted $O(|t|^{-1})$ bounds for both propagators are sharp.
Fix $t\ne0$ and let $m_t(s)$ be either $\cos(t\sqrt s)$ or
$\sin(t\sqrt s)/(t\sqrt s)$ for $s>0$, and zero for $s\le0$.
Under our spectral assumptions,
$P_{\mathrm{ac}}(H)=\mathbf1_{(0,\infty)}(H)$, so $m_t(H)$
is the corresponding propagator with $P_{\mathrm{ac}}(H)$.

Set $H_a=\Delta^2+V(\cdot+a)$. As $|a|\to\infty$, multiplication
by $V(\cdot+a)$ converges strongly to zero on $L^2$.
The resolvent identity then gives $H_a\to\Delta^2$ in the strong
resolvent sense. Since $m_t$ is bounded and continuous away from
zero, and $\Delta^2$ has no eigenvalue at zero, we obtain
\[
m_t(H_a)\longrightarrow m_t(\Delta^2)
\quad\text{strongly on }L^2.
\]
Translations preserve the $L^1\to L^\infty$ norm. Hence, for
$f,g\in C_c^\infty(\mathbb R^2)$,
\[
|\langle m_t(\Delta^2)f,g\rangle|
=\lim_{|a|\to\infty}|\langle m_t(H_a)f,g\rangle|
\le \|m_t(H)\|_{L^1\to L^\infty}\|f\|_1\|g\|_1.
\]
Taking the supremum over $\|f\|_1=\|g\|_1=1$ and using
Proposition~\ref{free_case}, we conclude that
\[
\|m_t(H)\|_{L^1\to L^\infty}
\ge \|m_t(\Delta^2)\|_{L^1\to L^\infty}
\gtrsim |t|^{-1}.
\]
}
\end{remark}

\section{The second-kind resonance case}\label{sec:second}
In this section, we prove the low-energy estimates for a second-kind
resonance. We express the moment conditions in terms of the threshold
projections, estimate the corresponding blocks, and compute the leading
sine coefficients to determine the sharp decay rates.

Assume the hypotheses of Theorem~\ref{main_theorem-2}. The case $V=0$ follows from Proposition~\ref{free_case}, so take $V\not\equiv0$. By
Proposition~\ref{characterizations-1}, $\mathcal E$ is identified with
$S_3L^2$ by $\phi\mapsto Uv\phi$, and $S_4=0$. Thus
\begin{equation}\label{eq:second-projection-criteria}
\tau\not\equiv0\iff Q_4^0\ne0,\qquad
\ker\mathfrak M_2\ne\{0\}\iff Q_{42}^1\ne0,\qquad
\ker\tau\ne\{0\}\iff Q_4^1\ne0.
\end{equation}
Indeed, if $\psi=Uv\phi$, then
$\langle x_ix_jV,\phi\rangle=\langle x_ix_jv,\psi\rangle$.
The kernels of $\tau$ and $\mathfrak M_2$ therefore correspond to
$S_4^0L^2$ and $S_4^1L^2$, respectively. Since $S_4=0$, these
are $Q_4^1L^2$ and $Q_{42}^1L^2$.
Together with Theorem~\ref{main_theorem_high}, these equivalences
reduce the proof of Theorem~\ref{main_theorem-2} to the following
Theorem \ref{main_theorem_low_2}.

\begin{theorem}\label{main_theorem_low_2}
Let $H = \Delta^2 + V$ with $V\not\equiv0$ and $|V(x)| \lesssim \langle x \rangle^{-14-}$.
Assume that $H$ has no positive embedded eigenvalues and that zero is a second-kind resonance of $H$. We distinguish two cases:
\begin{itemize}
    \item[(i)] If $Q_4^0 \neq 0$ and $Q_{42}^1 \neq 0$, then
$$
\begin{aligned}
& \left\|\cos (t \sqrt{H}) P_{\mathrm{ac}}(H)\chi_1(H)\right\|_{L^1 \rightarrow L^{\infty}}
+ \bigg \|\frac{\sin (t \sqrt{H})}{t\sqrt{H}} P_{\mathrm{ac}}(H)\chi_1(H)\bigg \|_{L^1 \to L^{\infty}}
\lesssim \langle t \rangle^{-1}(\log (2+|t|))^2,
\end{aligned}
$$
For $|t|\gg1$, the weighted estimates are
\[
\begin{aligned}
\left\|\cos(t\sqrt H)P_{\mathrm{ac}}(H)\chi_1(H)\right\|_{L^1_\omega\to L^\infty_{-\omega}}
&\lesssim\frac{\log|t|}{|t|},\quad
\bigg\|\frac{\sin(t\sqrt H)}{t\sqrt H}P_{\mathrm{ac}}(H)\chi_1(H)\bigg\|_{L^1_\omega\to L^\infty_{-\omega}}
\sim\frac{(\log|t|)^2}{|t|}.
\end{aligned}
\]

    \item[(ii)] If $Q_4^0 = 0$ or $Q_{42}^1 = 0$,  then  the unweighted  bound improves to
\begin{equation}\label{2-t}
    \left\|\cos (t \sqrt{H}) P_{\mathrm{ac}}(H)\chi_1(H)\right\|_{L^1 \to L^{\infty}} +
    \bigg\|\frac{\sin (t \sqrt{H})}{t\sqrt{H}} P_{\mathrm{ac}}(H) \chi_1(H) \bigg\|_{L^1 \to L^{\infty}} \lesssim \langle t \rangle^{-1}.
\end{equation}
In the weighted setting $L^1_\omega\to L^\infty_{-\omega}$, for $|t| \gg 1,$
\begin{equation}\label{eq:second_weight_cos}
\left\|
\cos(t\sqrt H)P_{\mathrm{ac}}(H)\chi_1(H)
\right\|_{L^1_\omega\to L^\infty_{-\omega}}
\lesssim \frac{1}{|t|\log|t|},
\end{equation}
If $Q_4^1=0$, estimate \eqref{eq:second_weight_cos} also holds
with $\cos(t\sqrt H)$ replaced by $\sin(t\sqrt H)/(t\sqrt H)$.
If $Q_4^1\ne0$, then
\begin{equation}\label{eq:second_weight_sin}
\bigg\|\frac{\sin(t\sqrt H)}{t\sqrt H}P_{\mathrm{ac}}(H)\chi_1(H)
\bigg\|_{L^1_\omega\to L^\infty_{-\omega}}
\sim |t|^{-1}.
\end{equation}
      \end{itemize}
      \end{theorem}

Subsection~\ref{subsec:second_1} treats Theorem \ref{main_theorem_low_2} (i), where the $(4,4)$
block produces the leading logarithmic term. Subsection~\ref{subsec:second_2}
treats Theorem \ref{main_theorem_low_2} (ii), where the surviving constant coefficient determines
the weighted sine bound. 

\subsection{\texorpdfstring{The estimates for the case $Q_{4}^0 \neq 0$ and $Q_{42}^1 \neq 0$}{Nonzero trace and zero quadratic moments}} \label{subsec:second_1}  
Proposition~\ref{prop_K} handles the blocks with both indices at most
$3$. It remains to consider blocks containing the index $4$.

Recall from Theorem \ref{thm:M_inverse} (II) that for $\alpha = 1,2,3$:  
\begin{align}\label{2nd-M}
	\|\partial_\lambda^\ell \mathcal{M}_{4,\alpha}^\pm(\lambda)\|_{\mathbb{B}(L^2)} + \|\partial_\lambda^\ell \mathcal{M}_{\alpha,4}^\pm(\lambda)\|_{\mathbb{B}(L^2)} \lesssim \lambda^{1-\ell}|\log \lambda|^{4}, \quad \ell = 0,1,2.  
\end{align}
Moreover, for pairs involving indices 0 and 4, we have:  
\begin{align}  
	\mathcal{M}_{0,4}^{\pm}(\lambda) &= (\log \lambda)^{-1} Q_0 \Lambda^\pm(\lambda) Q_{4}^0 + Q_0 \Lambda^\pm(\lambda) Q_{4}^1, \nonumber\\  
	\mathcal{M}_{4,0}^{\pm}(\lambda) &= (\log \lambda)^{-1} Q_{4}^0 \Lambda^\pm(\lambda) Q_0 + Q_{4}^1 \Lambda^\pm(\lambda) Q_0, \nonumber\\ 
	\label{M44-1}
    	\mathcal{M}_{4,4}^{\pm}(\lambda) &= Q_{4} \Lambda^\pm(\lambda) Q_{4} + (\log\lambda)\bigl(Q_{4} \Lambda^\pm(\lambda) Q_{4}^1+Q_{4}^1 \Lambda^\pm(\lambda) Q_{4}\bigr) + (\log\lambda)^2 Q_{4}^1 \Lambda^\pm(\lambda) Q_{4}^1.
\end{align} 

\begin{proposition}\label{prop_second_2_3}
	Let $B^\pm(\lambda) \in \{\mathcal{M}_{4,\alpha}^{\pm}(\lambda), \mathcal{M}_{\alpha,4}^{\pm}(\lambda)\}$ with $\alpha \in \{0,1,2,3,4\}$. Then
	\[  
\sup _{x, y \in \mathbb{R}^2}\left| \mathcal{K}_{B}^\pm(t,x,y)\right| + \sup _{x, y \in \mathbb{R}^2}\left|\mathcal{N}_{B}^\pm (t,x,y)\right| \lesssim
	\begin{cases}
	\langle t\rangle^{-1}, & \alpha\in\{0,1,2,3\},\\
	\langle t\rangle^{-1}\bigl(\log(2+|t|)\bigr)^2, & {\alpha} = 4.
\end{cases}  
	\]  
\end{proposition}

\begin{proof}
Lemma~\ref{lemma_projection} gives
\[
\begin{aligned}
(Q_0vR_0^\pm(\lambda^4))(\cdot,z)
&=\lambda^{-2}\Omega_{0,\pm}(\lambda,\cdot,z),\\
(Q_\alpha vR_0^\pm(\lambda^4))(\cdot,z)
&=\lambda^{-1}\Omega_{\alpha,\pm}(\lambda,\cdot,z),
\qquad 1\le\alpha\le3,\\
(Q_4vR_0^\pm(\lambda^4))(\cdot,z)
&=\mathcal J_{4,\pm}(\lambda,\cdot,z).
\end{aligned}
\]
After extracting $e^{\pm i\lambda|z|}$, the
$e^{\mp i\lambda|z|}\Omega_{\alpha,\pm}$ factors satisfy
\eqref{lemma_projection_1}, while \eqref{J} gives one factor
$|\log\lambda|$ for $e^{\mp i\lambda|z|}\mathcal J_{4,\pm}$.
For $Q_4^1vR_0^\pm$, the cancellation
$Q_4^1(v|\cdot|^2)=0$ gives the finer decomposition
\begin{equation}\label{eq:resolvent_decomp-Q4}
Q_4^1vR_0^\pm
=b_0Q_4^1vG_2^0+\mathcal T_4^1+\mathcal T_{4,\pm}^1.
\end{equation}
By \eqref{lemma_projection_G}--\eqref{lemma_projection_1},
\[
\begin{aligned}
\left\|\partial_\lambda^\ell
\bigl(b_0Q_4^1vG_2^0+\mathcal T_4^1\bigr)(\cdot,z)
\right\|_2
&\lesssim \lambda^{-\ell},\quad
\left\|\partial_\lambda^\ell
\bigl(e^{\mp i\lambda|z|}
\mathcal T_{4,\pm}^1(\cdot,z)\bigr)\right\|_2
\lesssim
\lambda^{-\ell}\langle\lambda z\rangle^{-1/2},
\qquad 0\le\ell\le2.
\end{aligned}
\]
Thus the terms in this decomposition have no logarithmic loss.

Substituting these formulas into \eqref{K_1} and
\eqref{Lambda_1}, each resulting pair of terms takes the form
\begin{equation}\label{eq:amplitude_def_prop1}
\begin{aligned}
&\frac{2}{\pi i}\int_0^\infty
\cos(t\lambda^2)\lambda
e^{\pm i\lambda r}\mathcal E^\pm\,d\lambda,\quad
&\frac{2}{\pi it}\int_0^\infty
\sin(t\lambda^2)\lambda^{-1}
e^{\pm i\lambda r}\mathcal E^\pm\,d\lambda.
\end{aligned}
\end{equation}
We extract $e^{\pm i\lambda|z|}$ from
$\Omega_{\alpha,\pm}(\cdot,z)$,
$\mathcal J_{4,\pm}(\cdot,z)$, and
$\mathcal T_{4,\pm}^1(\cdot,z)$.
We extract no phase from $b_0Q_4^1vG_2^0$ or $\mathcal T_4^1$.
Under our inner-product convention, the extracted phases combine
into $e^{\pm i\lambda r}$:
$r=|x|+|y|$ if we extract phases from both factors,
$r=|x|$ or $|y|$ if we extract a phase only from the corresponding
factor, and $r=0$ if we extract neither.  We bound $\mathcal E^\pm$ for each block below.

\emph{The $(0,4)$ and $(4,0)$ blocks.}
For the $(0,4)$ block, the prefactor $\lambda^2$ cancels
$\lambda^{-2}$ from $Q_0vR_0^\mp$.
The expansion
\[
\mathcal M_{0,4}^\pm
=(\log\lambda)^{-1}Q_0\Lambda^\pm Q_4^0
+Q_0\Lambda^\pm Q_4^1
\]
then removes the possible logarithmic loss in two ways.
For the first term, $(\log\lambda)^{-1}$ cancels the logarithmic
bound for $Q_4^0\mathcal J_{4,\pm}$, with $r=|x|+|y|$.
For the second term, use \eqref{eq:resolvent_decomp-Q4}:
its first two terms give $r=|x|$, and its last term gives
$r=|x|+|y|$, all without a logarithmic loss.
Leibniz's rule and H\"older's inequality therefore give, for every resulting amplitude,
\[
|\partial_\lambda^\ell\mathcal E^\pm|
\lesssim
\lambda^{-\ell}\langle\lambda r\rangle^{-1/2}
\widetilde\chi_1(\lambda/2),
\qquad 0\le\ell\le2.
\]
The $(4,0)$ block follows by the same argument with
$x$ and $y$ interchanged.

\emph{The $(\alpha,4)$ and $(4,\alpha)$ blocks,
$1\le\alpha\le3$.}
The prefactor $\lambda^{4-k_\alpha-k_4}=\lambda$ cancels $\lambda^{-1}$ from
$Q_\alpha vR_0^\mp$. For the $(\alpha,4)$ block,
$r=|x|+|y|$ and
\[
\mathcal E^\pm
=e^{\mp i\lambda r}\widetilde\chi_1(\lambda)
\left\langle
\mathcal M_{\alpha,4}^\pm\mathcal J_{4,\pm}(\cdot,y),
\Omega_{\alpha,\mp}(\cdot,x)
\right\rangle.
\]
The bound \eqref{2nd-M} for
$\mathcal M_{\alpha,4}^\pm$
and the single logarithmic factor from $\mathcal J_{4,\pm}$
give
\[
|\partial_\lambda^\ell\mathcal E^\pm|
\lesssim
\lambda^{1-\ell}|\log\lambda|^5
\langle\lambda r\rangle^{-1/2}\widetilde\chi_1(\lambda/2)
\lesssim
\lambda^{-\ell}\langle\lambda r\rangle^{-1/2}
\widetilde\chi_1(\lambda/2).
\]
The $(4,\alpha)$ blocks satisfy the same estimate.
Lemma~\ref{oscillatory}(ii), with $(\sigma,\nu)=(0,0)$,
now proves the assertion for $\alpha=0,1,2,3$.

\emph{The $(4,4)$ block.}
Here $\lambda^{4-k_4-k_4}=1$.
By \eqref{J}, each $\mathcal J_{4,\pm}$ factor contributes
one factor $|\log\lambda|$, whereas the terms in
\eqref{eq:resolvent_decomp-Q4} have no logarithmic loss.
Together with the coefficients in \eqref{M44-1}, this gives
at most $|\log\lambda|^2$ for every resulting amplitude.
Choose $r\in\{0,|x|,|y|,|x|+|y|\}$ according to the
extracted phases, as described above. Leibniz's rule and H\"older's inequality give
\[
|\partial_\lambda^\ell\mathcal E^\pm|
\lesssim
\lambda^{-\ell}|\log\lambda|^2
\langle\lambda r\rangle^{-1/2}
\widetilde\chi_1(\lambda/2),
\qquad 0\le\ell\le2.
\]
Lemma~\ref{oscillatory}(ii), with $(\sigma,\nu)=(0,-2)$,
gives the $(4,4)$ block estimate and completes the proof.
\end{proof}

For the weighted asymptotics, \eqref{D40D04D44_2} and
\eqref{Upsilon44} give the refined expansion
\begin{align}\label{huaM44}
\mathcal{M}_{4,4}^{\pm}(\lambda) &= Q_4^0\mathcal{C}^\pm Q_4^0 +(\log \lambda) \bigl(Q_4^0 \mathcal{C}_1^\pm Q_{42}^1 + Q_{42}^1 \mathcal{C}_2^\pm Q_4^0 \bigr)+ (\log \lambda)^2 (a_2^\pm)^{-1} Q_{42}^1 \mathcal{S}^{-1} Q_{42}^1 \nonumber \\
&  \quad + (\log \lambda)^{-1} Q_4^0\Lambda^\pm(\lambda) Q_4^0 + \bigl(Q_4^0 \Lambda^\pm(\lambda) Q_{4}^1 + Q_{4}^1 \Lambda^\pm(\lambda) Q_4^0\bigr)  +(\log \lambda) Q_{4}^1 \Lambda^\pm(\lambda) Q_{4}^1.  
\end{align}  
Here, with $\mathcal{W}_1 = (b_0 Q_0 v G_2 v Q_4^0)^{-1} Q_0 T_0 Q_{42}^1,$ the operators $\mathcal{C}^\pm, \mathcal{C}_1^\pm, \mathcal{C}_2^\pm$ are given by  
\begin{equation}\label{eq:mathcal_C}  
\mathcal{C}^\pm := (a_2^\pm)^{-1}\mathcal{W}_1 \mathcal{S}^{-1} \mathcal{W}_1^*, \quad  
\mathcal{C}_1^\pm := - (a_2^\pm)^{-1}\mathcal{W}_1 \mathcal{S}^{-1}, \quad  
\mathcal{C}_2^\pm := - (a_2^\pm)^{-1}\mathcal{S}^{-1}\mathcal{W}_1^*,
\end{equation}  
and $\mathcal{S}$ is the strictly positive operator on $Q_{42}^1 L^2$
defined by  $$
\mathcal{S} =\mathcal{W}_1^* \bigl[Q_4^0 v G_4 v Q_4^0 - Q_4^0 v G_4 v Q_{41}^1 (Q_{41}^1 v G_4 v Q_{41}^1)^{-1} Q_{41}^1 v G_4 v Q_4^0\bigr] \mathcal{W}_1,
$$
where  the inverse $(Q_{41}^1 v G_4 v Q_{41}^1)^{-1}$ is interpreted
as the zero operator  
if $Q_{41}^1L^2=\{0\}$.

Since \(\mathcal K_B^\pm\) and \(\mathcal N_B^\pm\) are uniformly
\(O(|t|^{-1})\) whenever
\(B^\pm(\lambda)=\mathcal M_{\alpha,\beta}^\pm(\lambda)\) with
\((\alpha,\beta)\neq(4,4)\), the desired weighted estimate reduces
to the contribution of
\(B^\pm(\lambda)=\mathcal M_{4,4}^\pm(\lambda)\), which is treated in
Proposition~\ref{prop_second_1}.

\begin{proposition}\label{prop_second_1}  
	Let $B^\pm(\lambda)= \mathcal{M}_{4,4}^{\pm}(\lambda)$. Then  
 for $|t|\geq 2$, 
\begin{align*}
\mathcal{K}_{B}^\pm(t,x,y)&=O\bigl(\omega(x)\omega(y)|t|^{-1} \log|t|\bigr),\nonumber\\
\mathcal{N}_{B}^\pm(t,x,y)&=	\mathcal{A}^\pm(x,y)\frac{(\log|t|)^2}{8i|t|}+{O}\bigl(\omega(x)\omega(y)|t|^{-1}\log|t|\bigr), 
	\end{align*} 
	where the  kernel $\mathcal{A}^\pm(x,y)$ is uniformly bounded and explicitly given by
	\begin{align}\label{eq:mathcal_A_pm}
		\mathcal{A}^\pm(x,y) ={}& b_0^2 \langle \mathcal{C}^\pm Q_4^0(v |\cdot|^2),  Q_4^0(v |\cdot|^2)\rangle 
		+ b_0^2 \langle  \mathcal{C}_1^\pm (Q_{42}^1 v G_2^0)(\cdot,y),  Q_4^0(v |\cdot|^2)\rangle  \nonumber\\
		&+ b_0^2 \langle  \mathcal{C}_2^\pm Q_4^0(v |\cdot|^2),  (Q_{42}^1 v G_2^0)(\cdot,x)\rangle + b_0^2 (a_2^\pm)^{-1} \bigl\langle  \mathcal{S}^{-1} (Q_{42}^1 v G_2^0)(\cdot,y), (Q_{42}^1 v G_2^0)(\cdot,x) \bigr\rangle.
	\end{align}
\end{proposition}  

\begin{proof}
Lemma~\ref{lemma_projection} gives
\begin{equation}\label{eq:resolvent_decomp}
\begin{aligned}
Q_4^0vR_0^\pm&=b_0(\log\lambda)Q_4^0(v|\cdot|^2)+b_0Q_4^0vG_2^0
+\mathcal T_4^0+\mathcal T_{4,\pm}^0,\\
Q_4^1vR_0^\pm&=b_0Q_4^1vG_2^0+\mathcal T_4^1+\mathcal T_{4,\pm}^1,\\
Q_{42}^1vR_0^\pm&=b_0Q_{42}^1vG_2^0+\lambda(\mathcal T_4^{1,2}+\mathcal T_{4,\pm}^{1,2}).
\end{aligned}
\end{equation}

Inserting \eqref{huaM44} and \eqref{eq:resolvent_decomp} into \eqref{K_1} and
\eqref{Lambda_1}, the
coefficient of $(\log\lambda)^2$ is exactly
$\mathcal A^\pm(x,y)$ in \eqref{eq:mathcal_A_pm}. It comes from
pairing $b_0(\log\lambda)Q_4^0(v|\cdot|^2)$ and $b_0Q_{42}^1vG_2^0$ with the four leading
operators in \eqref{huaM44}; all these contributions have $r=0$.
Every remaining term either loses one logarithm or gains a factor
$\lambda$. Using  the
same phase assignments as in Proposition~\ref{prop_second_2_3},
Leibniz's rule and H\"older's inequality bound each remaining amplitude by
\[
|\partial_\lambda^\ell\mathcal E^\pm|
\lesssim\widetilde\chi_1(\lambda/2)\lambda^{-\ell}|\log\lambda|
\langle\lambda r\rangle^{-1/2}\omega(x)\omega(y),\qquad 0\le\ell\le2.
\]
Lemma~\ref{oscillatory}(ii) with $(\sigma,\nu)=(0,-1)$ bounds these
terms by $C\omega(x)\omega(y)|t|^{-1}\log|t|$. Thus
\begin{equation}\label{K-N}
\begin{aligned}
\mathcal K_B^\pm&=\frac{2}{\pi i}\mathcal A^\pm I(t)
+O(\omega(x)\omega(y)|t|^{-1}\log|t|),\\
\mathcal N_B^\pm&=\frac{2}{\pi i}\mathcal A^\pm J(t)
+O(\omega(x)\omega(y)|t|^{-1}\log|t|).
\end{aligned}
\end{equation}
Lemma~\ref{lemma:I_J} gives
$I(t)=\frac\pi8|t|^{-1}\log|t|+O(|t|^{-1})$ and
$J(t)=\frac\pi{16}|t|^{-1}(\log|t|)^2+O(|t|^{-1}\log|t|)$,
which proves the proposition.
\end{proof}

We next prove the sine lower bound.

\begin{proposition}\label{sine-sharp}
	If $Q_4^0 \neq 0$ and $Q_{42}^1 \neq 0,$  then for $|t|\gg1,$
\begin{equation}\label{optimality_0}
	\Big\|\frac{\sin (t \sqrt{H})}{t\sqrt{H}} P_{\mathrm{ac}}(H)\chi_1(H)\Big\|_{L^1_\omega \to L^{\infty}_{-\omega}} \sim \frac{(\log |t|)^2}{|t|}.
\end{equation}  
\end{proposition}

\begin{proof}
By Propositions~\ref{prop_K} and~\ref{prop_second_2_3},
all blocks except $(4,4)$ contribute $O(|t|^{-1})$.
Together with Proposition~\ref{prop_second_1} and the free
estimate in Remark~\ref{rmk:free_case}, this gives
\begin{equation}\label{expan-sin}
\frac{\sin(t\sqrt H)}{t\sqrt H}P_{\mathrm{ac}}(H)\chi_1(H)
=-\frac{(\log|t|)^2}{8i|t|}
(\mathcal A^+-\mathcal A^-)
+\mathcal O(|t|^{-1}\log|t|)
\end{equation}
in $\mathbb B(L^1_\omega,L^\infty_{-\omega})$.
It remains to prove that the operator $\mathcal A^+-\mathcal A^-\ne0$.
Substituting \eqref{eq:mathcal_C} into the kernel \eqref{eq:mathcal_A_pm} of $\mathcal A^\pm$, we derive
\[
\begin{aligned}
\mathcal A^+-\mathcal A^-
={}&b_0^2\bigl[(a_2^+)^{-1}-(a_2^-)^{-1}\bigr]
(G_2vQ_4^0\mathcal W_1-G_2^0vQ_{42}^1)
\mathcal S^{-1}
(G_2vQ_4^0\mathcal W_1-G_2^0vQ_{42}^1)^*.
\end{aligned}
\]

We first note that
\begin{equation}\label{eq:threshold-injectivity}
\psi\in S_3L^2,\quad G_2^0v\psi\text{ constant}
\quad\Longrightarrow\quad v\psi=0\quad\Longrightarrow\quad\psi=0.
\end{equation}
The first implication follows by applying $\Delta^2$, since
$\Delta^2G_2^0=8\pi\delta_0$. For the second,
Proposition~\ref{characterizations-1} gives $\psi=Uv\phi$; hence
$\psi$ also vanishes on $\{v=0\}$. 
Now let $\psi\in Q_{42}^1L^2$ satisfy
\[
(G_2vQ_4^0\mathcal W_1-G_2^0vQ_{42}^1)\psi=0.
\]
Since $Q_4^0v=Q_4^0(x_jv)=0$ for $j=1,2$,
the function $G_2vQ_4^0\mathcal W_1\psi$ is constant.
Thus, \eqref{eq:threshold-injectivity} proves that
$G_2vQ_4^0\mathcal W_1-G_2^0vQ_{42}^1$ is injective on
$Q_{42}^1L^2$. Since  $\mathcal S^{-1}$ is
positive definite on 
$Q_{42}^1L^2$, and
$(a_2^+)^{-1}-(a_2^-)^{-1}=-1024i\ne0$, we have
$\mathcal A^+-\mathcal A^-\ne0$. 
\end{proof}

\subsection{\texorpdfstring{The estimates for the case $Q_{4}^0 = 0$ or $Q_{42}^1 = 0$}{Complementary second-kind cases}}\label{subsec:second_2}
We prove Theorem~\ref{main_theorem_low_2}(ii), beginning
with the unweighted bound.

\textbf{Part I. The $L^1\to L^\infty$ estimates.} Recall from Theorem \ref{thm:M_inverse} (II) that for $\alpha = 1,2,3$,
\begin{align}\label{2nd-2-M}
\|\partial_\lambda^\ell \mathcal{M}_{4,\alpha}^\pm(\lambda)\|_{\mathbb{B}(L^2)} + \|\partial_\lambda^\ell \mathcal{M}_{\alpha,4}^\pm(\lambda)\|_{\mathbb{B}(L^2)} \lesssim \lambda^{1-\ell}|\log \lambda|^{4}, \quad \ell = 0,1,2;
\end{align}
for the pairs $(\alpha, \beta) = (0,4), (4,0), (4,4)$,
        \begin{itemize}
        \item if $Q_4^0 = 0$, then $Q_4 = Q_4^1$ and
        \begin{align}\label{Q4=Q41}
         \mathcal{M}_{0, 4}^\pm(\lambda) =  Q_0 \Lambda^\pm(\lambda) Q_4^1, \ \  \mathcal{M}_{4, 0}^\pm(\lambda) =  Q_4^1 \Lambda^\pm(\lambda) Q_0, \ \
         \mathcal{M}_{4,4}^\pm(\lambda) = Q_4^1 \Lambda^\pm(\lambda) Q_4^1.
        \end{align}
       
    \item If $Q_4^0 \neq 0$ and $Q_{42}^1 = 0$, then $Q_4 = Q_4^0 \oplus Q_{41}^1$, and
        \begin{align}
		& \mathcal{M}_{4,0}^\pm(\lambda)=(\log \lambda)^{-1}  Q_4 \Lambda^\pm(\lambda) Q_0,  \quad
          \mathcal{M}_{0,4}^\pm(\lambda)=(\log \lambda)^{-1}  Q_0 \Lambda^\pm(\lambda) Q_4, \nonumber   \\ \label{2-M4-4}
		& \mathcal{M}_{4,4}^\pm(\lambda) = \sum_{h,l\in\{0,1\}}(\log \lambda)^{-(2-h-l)} Q_4^h \Lambda^\pm(\lambda) Q_4^l.
        \end{align}
        \end{itemize}
Proposition~\ref{prop_K} handles the blocks with both indices at
most $3$. The remaining blocks satisfy the following bound.

\begin{proposition}\label{prop_second_2_1}
	Let $B^\pm(\lambda) \in \{\mathcal{M}_{4,\alpha}^{\pm}(\lambda), \mathcal{M}_{\alpha, 4}^{\pm}(\lambda)\}$ with $\alpha \in \{0,1,2,3,4\}$. Then
	\[  
\sup _{x, y \in \mathbb{R}^2}\left| \mathcal{K}_{B}^\pm(t,x,y)\right| + \sup _{x, y \in \mathbb{R}^2}\left|\mathcal{N}_{B}^\pm (t,x,y)\right| \lesssim \langle t\rangle^{-1}.  
	\]  
\end{proposition}

\begin{proof}
For $\alpha\le3$, the proof of Proposition~\ref{prop_second_2_3}
applies to the expansions above without change. For $\mathcal M_{4,4}^\pm$,
\eqref{Q4=Q41} and \eqref{2-M4-4} attach one factor
$(\log\lambda)^{-1}$ to each $Q_4^0$ factor. This cancels the
logarithm in $Q_4^0\mathcal J_{4,\pm}$. The $Q_4^1$ factors have
the bounded decomposition \eqref{eq:resolvent_decomp-Q4}.
Consequently every amplitude 
satisfies
\[
|\partial_\lambda^\ell\mathcal E^\pm|
\lesssim\widetilde\chi_1(\lambda/2)\lambda^{-\ell}
\langle\lambda r\rangle^{-1/2},\qquad 0\le\ell\le2.
\]
Lemma~\ref{oscillatory}(ii) with $(\sigma,\nu)=(0,0)$ proves the result.
\end{proof}

\textbf{Part II. Weighted estimates.}
After Subsection~\ref{subsec:regular_first_2}, it remains to treat
$\mathcal M_{4,\alpha}^\pm$ and $\mathcal M_{\alpha,4}^\pm$ for
$0\le\alpha\le4$. We also revisit $\mathcal M_{0,0}^\pm$, whose
leading coefficient may differ from that in the regular and
first-kind cases.

For $\alpha=1,2,3$, the bound \eqref{2nd-2-M} of
$\mathcal M_{4,\alpha}^\pm,\mathcal M_{\alpha,4}^\pm$ and the projected expansion
\eqref{eq:expansion} give an amplitude of order
$O_2(\lambda^{1/2}\langle x\rangle^4\langle y\rangle^4)$.
As in Proposition~\ref{prop_K_weight_1}, combine the resulting
$O(|t|^{-5/4}\langle x\rangle^4\langle y\rangle^4)$ bound with
Proposition~\ref{prop_second_2_1} and \eqref{eq:weight-interpolation}.
This gives, for $|t|\ge2$,
\begin{equation}\label{eq:fast_decay}  
\left| \mathcal{K}_{B}^\pm(t,x,y)\right| +
\left|\mathcal{N}_{B}^\pm(t,x,y)\right| \lesssim \frac{\omega(x)\omega(y)}{|t|\log |t|}, \  B^\pm(\lambda) \in \{\mathcal{M}_{4,\alpha}^{\pm}(\lambda), \mathcal{M}_{\alpha,4}^{\pm}(\lambda)\}\ \text{with}\ \alpha \in \{1,2,3\}.  
\end{equation}  
It remains to treat $\alpha,\beta\in\{0,4\}$.

\noindent\underline{\textbf{Case 1: $Q_4^0=0$.}}
Here $Q_4=Q_4^1\ne0$. Combining
\eqref{D00}, \eqref{Upsilon44-2} with $
\mathcal M_{\alpha,\beta}^\pm(\lambda)=\mathcal D_{\alpha,\beta}^\pm(\lambda)+\Upsilon_{\alpha,\beta}^\pm(\lambda)
$ for $\alpha,\beta \in \{0,4\}$, we obtain 
the following refined expansions:
\[
	\mathcal{M}_{\alpha,\beta}^{\pm}(\lambda) = \mathcal{D}_{\alpha,\beta}^{\pm} + \lambda (-\log \lambda)^{3/2} \Lambda^\pm(\lambda), \qquad \alpha,\beta \in \{0,4\}, 
\] where
the operators $\mathcal{D}_{\alpha,\beta}^{\pm}$ are explicitly given by  
\begin{align}\label{Doo}
	\mathcal{D}_{0,0}^{\pm} &= (a_0^\pm)^{-1} \|V\|_{L^1}^{-1}Q_0 + (a_0^\pm)^{-2}\mathcal{W} (\boldsymbol{d}^\pm)^{-1} \mathcal{W}^*, \\
	\label{eq:expre_D_00} 
	\mathcal{D}_{4,4}^{\pm} &= (\boldsymbol{d}^\pm)^{-1}, \quad  
	\mathcal{D}_{0,4}^{\pm} = - (a_0^\pm)^{-1}  \mathcal{W} (\boldsymbol{d}^\pm)^{-1}, \quad  
	\mathcal{D}_{4,0}^{\pm} = - (a_0^\pm)^{-1} (\boldsymbol{d}^\pm)^{-1} \mathcal{W}^*,
\end{align}  
with $\mathcal{W} = \|V\|_{L^1}^{-1}Q_0 T_0 Q_4^1$ and  
$
(\boldsymbol{d}^\pm)^{-1} = [a_2^\pm Q_4^1 v G_4 v Q_4^1 - (a_0^\pm)^{-1} \|V\|_{L^1}^{-1} Q_4^1 T_0  Q_0 T_0 Q_4^1]^{-1}.  
$ The second term in $\mathcal D_{0,0}^\pm$ also contributes to
the leading sine coefficient.

\begin{proposition}\label{prop_second_critical_block}  
	Let $B^\pm(\lambda) = \mathcal{M}_{\alpha,\beta}^{\pm}(\lambda)$ for $(\alpha,\beta) \in \{0,4\} \times \{0,4\}$. If $Q_4^0 = 0$ $(\text{i.e.,}\  Q_4=Q_4^1 )$, then  
	\[  
\bigl|\mathcal{K}_{B}^\pm(t,x,y)\bigr| \lesssim \frac{\omega(x)\omega(y)}{|t|\log|t|}, \quad  
\mathcal{N}_{B}^\pm(t,x,y) = \frac{1}{2 i |t|} C^\pm_{\alpha,\beta}(x,y) + {O}\!\left(\frac{\omega(x)\omega(y)}{|t|\log|t|}\right),  \ \ |t| \ge 2,
	\]  
    where the integral kernels $C^\pm_{\alpha,\beta}(x,y)$ 
	are uniformly bounded and explicitly given by
	\begin{equation}\label{2nd-2-C}  
		\begin{split}
		C^\pm_{0,0}(x,y) &= (a_0^\pm)^2\bigl\langle  \mathcal{D}_{0,0}^{\pm}(Q_0v G_0)(\cdot), (Q_0v G_0)(\cdot) \bigr\rangle,\\
        C^\pm_{0,4}(x,y) &= b_0a_0^\pm \bigl\langle  \mathcal{D}_{0,4}^{\pm} (Q_{4}^1 v G_2^0)(\cdot,y), (Q_0v G_0)(\cdot) \bigr\rangle, \\  
		C^\pm_{4,0}(x,y) &= b_0a_0^\pm \bigl\langle  \mathcal{D}_{4,0}^{\pm}(Q_0v G_0)(\cdot), (Q_{4}^1 v G_2^0)(\cdot,x) \bigr\rangle,\\
		C^\pm_{4,4}(x,y)&= b_0^2 \bigl\langle \mathcal{D}_{4,4}^{\pm} (Q_{4}^1 v G_2^0)(\cdot,y), (Q_{4}^1 v G_2^0)(\cdot,x) \bigr\rangle.  
		\end{split}
	\end{equation}  
\end{proposition} 

\begin{proof}
The cancellations of $Q_4^1$ and the free expansion give
\[
Q_0vR_0^\pm=a_0^\pm\lambda^{-2}Q_0vG_0+Q_0vE_1^\pm,
\qquad Q_4^1vR_0^\pm=b_0Q_4^1vG_2^0+Q_4^1v\widetilde E_1^\pm,
\]
where, for $0\le\ell\le2$,
\[
\|\partial_\lambda^\ell(Q_0vE_1^\pm)(\cdot,y)\|_2
\lesssim\lambda^{-\ell-}\langle y\rangle^4,\qquad
\|\partial_\lambda^\ell(Q_4^1v\widetilde E_1^\pm)(\cdot,y)\|_2
\lesssim\lambda^{2-\ell}\langle y\rangle^4.
\]
Substituting these expansions and
$\mathcal M_{\alpha,\beta}^\pm=\mathcal D_{\alpha,\beta}^\pm
+\lambda (-\log \lambda)^{3/2} \Lambda^\pm(\lambda)$ into \eqref{eq:def_mathcal_E} gives
\[
\mathcal E^\pm(\lambda,x,y)
=\widetilde\chi_1(\lambda)C_{\alpha,\beta}^\pm(x,y)
+O_2(\lambda^{1/2}\langle x\rangle^4\langle y\rangle^4).
\]
The principal products give the stated $C_{\alpha,\beta}^\pm$,
which are bounded by \eqref{lemma_projection_G}. The error satisfies
\eqref{eq:estimate_E}. Applying the integral estimates from the proof
of Proposition~\ref{prop_K_weight_2} gives
\[
|\mathcal K_B^\pm(t,x,y)|+
\Big|\mathcal N_B^\pm(t,x,y)
-\frac{1}{2i|t|}C_{\alpha,\beta}^\pm(x,y)\Big|
\lesssim |t|^{-5/4}\langle x\rangle^4\langle y\rangle^4.
\]
Propositions~\ref{prop_K} and~\ref{prop_second_2_1} also bound this
left-hand side by $C|t|^{-1}$.
Applying
\eqref{eq:weight-interpolation} with $N=4$ and $\delta=1/4$
proves the proposition.

\end{proof}  

\begin{proposition}\label{sine-sharp-2}
	If $Q_4^0= 0,$  then for $|t|\gg1,$
\begin{equation}
\label{si-optimality-3}
	\Big\|\frac{\sin (t \sqrt{H})}{t\sqrt{H}} P_{\mathrm{ac}}(H)\chi_1(H)\Big\|_{L^1_\omega \to L^{\infty}_{-\omega}} \sim |t|^{-1}.
\end{equation}  
\end{proposition}
\begin{proof}
The term $(a_0^\pm)^{-1}\|V\|_1^{-1}Q_0$ in
\eqref{Doo} cancels the free sine coefficient $G_0/(8|t|)$.  Combining this with Proposition \ref{prop_second_critical_block}, 
we have
\begin{equation}\label{expan-si-2}
\frac{\sin(t\sqrt H)}{t\sqrt H}P_{\rm ac}(H)\chi_1(H)
=-\frac{1}{2i|t|}Y\bigl[(\boldsymbol d^+)^{-1}
-(\boldsymbol d^-)^{-1}\bigr]Y^*
+\mathcal O\bigl((|t|\log|t|)^{-1}\bigr),
\end{equation}
in $\mathbb B(L^1_\omega,L^\infty_{-\omega})$, where $Y=G_0vQ_0\mathcal W-b_0G_2^0vQ_4^1$.
 By \eqref{d-M},
$i[(\boldsymbol d^+)^{-1}-(\boldsymbol d^-)^{-1}]$ is positive definite on $Q_{4}^1L^2.$
The first term of $Y\psi$ is constant, so
\eqref{eq:threshold-injectivity} makes $Y$ injective on $Q_4^1L^2$.
Since $Q_4^1=Q_4\ne0$,  the operator $Y\bigl[(\boldsymbol d^+)^{-1}
-(\boldsymbol d^-)^{-1}\bigr]Y^*\neq 0.$ Hence we obtain the desired bound \eqref{si-optimality-3}.
\end{proof}  

\noindent\underline{\textbf{Case 2: $Q_4^0\ne0$ and $Q_{42}^1=0$.}} In this case,
$
Q_4
=
Q_4^0\oplus Q_4^1,\ 
Q_4^1=Q_{41}^1.
$
By tracing the contributions of $\Lambda^\pm(\lambda)$ in \eqref{D40D04D44_1}, and using \eqref{Upsilon44-2} together with  $
\mathcal M_{\alpha,\beta}^\pm(\lambda)=\mathcal D_{\alpha,\beta}^\pm(\lambda)+\Upsilon_{\alpha,\beta}^\pm(\lambda)
$ for $\alpha,\beta \in \{0,4\},$  we obtain
\begin{align*}
	\mathcal{M}_{0,0}^\pm(\lambda)=&(\log \lambda)^{-2}  Q_0 \Lambda^\pm(\lambda) Q_0,\\
	\mathcal{M}^\pm_{0,4}(\lambda)=&(\log\lambda)^{-1}Q_0\mathfrak{D}Q_4^0+(\log\lambda)^{-1}Q_0\Lambda^\pm(\lambda)Q_4^1+(\log\lambda)^{-2}Q_0\Lambda^\pm(\lambda)Q_4,\\
	\mathcal{M}^\pm_{4,0}(\lambda)=&(\log\lambda)^{-1}Q_4^0\mathfrak{D}^*Q_0+(\log\lambda)^{-1}Q_4^1\Lambda^\pm(\lambda)Q_0+(\log\lambda)^{-2}Q_4\Lambda^\pm(\lambda)Q_0,\\
	\mathcal{M}^\pm_{4,4}(\lambda)=&(\log\lambda)^{-2}Q_4^0\mathfrak{C}_{00}^\pm Q_4^0+(\log\lambda)^{-1}\bigl(Q_4^0\mathfrak{C}_{01}^\pm Q_4^1+Q_4^1\mathfrak{C}_{10}^\pm Q_4^0\bigr)+Q_4^1\mathfrak{C}_{11}^\pm Q_4^1\\&+(\log\lambda)^{-3}Q_4^0\Lambda^\pm(\lambda)Q_4^0+(\log\lambda)^{-2}\bigl(Q_4^1\Lambda^\pm(\lambda)Q_4^0+Q_4^0\Lambda^\pm(\lambda)Q_4^1\bigr)+(\log\lambda)^{-1}Q_4^1\Lambda^\pm(\lambda)Q_4^1,
\end{align*}
where $\mathfrak{D}=(b_0Q_4^0vG_2vQ_0)^{-1},$ $\mathfrak{C}^\pm_{11}=(a_2^{\pm})^{-1}D_0$ with $D_0=\bigl(Q_{41}^1vG_4vQ_{41}^1\bigr)^{-1}$ and 
\begin{align*}
	\mathfrak{C}^\pm_{00}&=\mathfrak{D}^*	\Bigl((a_2^{\pm})^{-1}
	Q_0T_0Q_4^1D_0Q_4^1T_0Q_0
	-a_0^{\pm}Q_0vG_0vQ_0
	\Bigr)\mathfrak{D},\\
	\mathfrak{C}^\pm_{01}&=-(a_2^{\pm})^{-1}\mathfrak{D}^*Q_0T_0Q_4^1D_0,\quad 	\mathfrak{C}^\pm_{10}=-(a_2^{\pm})^{-1}D_0Q_4^1T_0Q_0\mathfrak{D}.
\end{align*}
Here,  if $Q_{41}^1L^2=\{0\}$, the inverse $(Q_{41}^1 v G_4 v Q_{41}^1)^{-1}$ is interpreted as the zero operator.

 Note that
$(Q_0vR_0^\pm(\lambda^4))(x,y)=a_0^{\pm}\lambda^{-2}(Q_0vG_0)(x,y)+(Q_0vE_1^\pm)(\lambda,x,y)$ (see \eqref{eq:expansion})  and
\begin{align*}
	\bigl(Q_{4}^0 v R_0^\pm(\lambda^4)\bigr)(x,y) &=b_0(\log\lambda) (Q_4^0vG_2)(x,y)+a_1^\pm (Q_4^0vG_2)(x,y)+(b_0 Q_4^0vG_2^0)(x, y)+\bigl(Q_{4}^0 v \widetilde{E}_1^\pm)(\lambda,x,y),\\
		\bigl(Q_{4}^1 v R_0^\pm(\lambda^4)\bigr)(x,y) &= (b_0 Q_4^1vG_2^0)(x, y) +\bigl(Q_{4}^1 v \widetilde{E}_1^\pm)(\lambda,x,y),
\end{align*}
where $\bigl\|(Q_0 v G_0)(\cdot,y)\bigr\|_{L^2}+\bigl\|(Q_4^0v G_2)(\cdot,y)\bigr\|_{L^2}+\bigl\|(Q_4^1v G_2^0)(\cdot,y)\bigr\|_{L^2}\lesssim1,$ $\bigl\|(b_0Q_4^0 v G_2^0)(\cdot,y)\bigr\|_{L^2}\lesssim\omega(y),$ and for $\ell=0,1,2,$
\begin{align*}	\bigl\|\partial_\lambda^\ell(Q_0v E_1^\pm)(\lambda,\cdot,y)\bigr\|_{L^2}\lesssim\lambda^{-\ell-}\langle y\rangle^4,\quad 	\bigl\|\partial_\lambda^\ell(Qv \widetilde{E}_1^\pm)(\lambda,\cdot,y)\bigr\|_{L^2}\lesssim\lambda^{2-\ell}\langle y\rangle^4\ \text{for}\ Q\in\{Q_4^0, Q_4^1\}.
\end{align*}
After substitution into \eqref{eq:def_mathcal_E}, the terms
independent of $\lambda$, apart from the cutoff, give the kernels
$\mathfrak T_{\alpha,\beta}^\pm$ below. Every remaining amplitude
satisfies one of the bounds
\[
|\partial_\lambda^\ell\mathcal E^\pm|
\lesssim\widetilde\chi_1(\lambda/2)
\begin{cases}
\lambda^{-\ell}|\log\lambda|^{-1}\omega(x)\omega(y),\\
\lambda^{1/2-\ell}\langle x\rangle^4\langle y\rangle^4,
\end{cases}
\qquad 0\le\ell\le2.
\]
Lemma~\ref{oscillatory}(ii), with $r=0$ and
$(\sigma,\nu)=(0,1)$ or $(-1/2,0)$, bounds the corresponding
cosine and sine integrals by
$C\omega(x)\omega(y)(|t|\log|t|)^{-1}$ and
$C|t|^{-5/4}\langle x\rangle^4\langle y\rangle^4$, respectively.
The constant amplitudes have a rapidly decaying cosine integral
and the sine integral \eqref{eq:integral_sin}.
Propositions~\ref{prop_K} and~\ref{prop_second_2_1}, together
with the boundedness of $\mathfrak T_{\alpha,\beta}^\pm$, also
give a uniform $C|t|^{-1}$ bound for each kernel remainder.
Combining these bounds and applying \eqref{eq:weight-interpolation}
with $N=4$ and $\delta=1/4$ yields
$$
|\mathcal{K}_{\mathcal{M}_{0,0}}^\pm(t,x,y)|+|\mathcal{N}_{\mathcal{M}_{0,0}}^\pm(t,x,y)|\lesssim\omega(x)\omega(y)(|t|\log|t|)^{-1},
$$
and for $B^\pm(\lambda)\in\{\mathcal{M}^\pm_{0,4}(\lambda), \mathcal{M}^\pm_{4,0}(\lambda), \mathcal{M}^\pm_{4,4}(\lambda)\}$,
\[
\bigl|\mathcal{K}_{B}^\pm(t,x,y)\bigr| \lesssim \frac{\omega(x)\omega(y)}{|t|\log|t|}, \quad
\mathcal{N}_{B}^\pm(t,x,y) = \frac{1}{2 i |t|} \mathfrak{T}^\pm_{\alpha,\beta}(x,y) + O\!\left(\frac{\omega(x)\omega(y)}{|t|\log|t|}\right), \quad |t| \ge 2,
\]
where the integral kernels $\mathfrak{T}^\pm_{\alpha,\beta}(x,y)$ are uniformly bounded and explicitly given by   
\begin{equation}  \label{sharp-operator}
	\begin{split}
	\mathfrak{T}^\pm_{0,4}(x,y) =& b_0a_0^\pm \bigl\langle  \mathfrak{D} (Q_{4}^0 v G_2)(\cdot,y), (Q_0v G_0)(\cdot) \bigr\rangle, \ \  
	\mathfrak{T}^\pm_{4,0}(x,y)= b_0a_0^\pm \bigl\langle  \mathfrak{D}^*(Q_0v G_0)(\cdot), (Q_{4}^0v G_2)(\cdot,x) \bigr\rangle,\\
		\mathfrak{T}^\pm_{4,4}(x,y)= &b_0^2 \Bigl(\bigl\langle \mathfrak{C}_{00}^\pm (Q_{4}^0 v G_2)(\cdot,y), (Q_{4}^0 v G_2)(\cdot,x) \bigr\rangle+\bigl\langle \mathfrak{C}_{01}^\pm (Q_{4}^1 v G_2^0)(\cdot,y), (Q_{4}^0 v G_2)(\cdot,x) \bigr\rangle\\
		&\ \ \ +\bigl\langle \mathfrak{C}_{10}^\pm (Q_{4}^0 v G_2)(\cdot,y), (Q_{4}^1 v G^0_2)(\cdot,x) \bigr\rangle+\bigl\langle \mathfrak{C}_{11}^\pm (Q_{4}^1 v G_2^0)(\cdot,y), (Q_{4}^1 v G^0_2)(\cdot,x) \bigr\rangle\Bigr).  
	\end{split}
\end{equation}  
Hence, 
it remains to prove the following sharp bound for the sine evolution.

\begin{proposition}\label{2-sine-sharp-2}
Assume $Q_4^0\ne0$ and $Q_{42}^1=0$. If $Q_{41}^1\neq0$, then, for
$|t|\gg1$,
\[
\bigg\|\frac{\sin(t\sqrt H)}{t\sqrt H}P_{\mathrm{ac}}(H)\chi_1(H)
\bigg\|_{L^1_\omega\to L^\infty_{-\omega}}
\sim |t|^{-1},
\]
 while if $Q_{41}^1=0$, this sine norm is bounded by  $(|t|\log|t|)^{-1}$.
\end{proposition}
\begin{proof}
Let $X=G_2v\mathfrak D^*T_0Q_{41}^1-G_2^0vQ_{41}^1$. The inverse definition of $\mathfrak D$
and the moment cancellations give
\[
b_0G_0v\mathfrak DvG_2=b_0G_2v\mathfrak D^*vG_0=G_0,
\qquad b_0^2G_2v\mathfrak D^*vG_0v\mathfrak DvG_2=G_0.
\]
Thus
\[
\mathfrak T_{0,4}^\pm+\mathfrak T_{4,0}^\pm=2a_0^\pm G_0,
\qquad
\mathfrak T_{4,4}^\pm
=b_0^2(a_2^\pm)^{-1}XD_0X^*-a_0^\pm G_0.
\]
Since $a_0^+-a_0^-=i/4$, the $(0,4)$ and $(4,0)$ blocks
together contribute
\[
-\frac{1}{2i|t|}
\bigl[
(\mathfrak T_{0,4}^++\mathfrak T_{4,0}^+)
-(\mathfrak T_{0,4}^-+\mathfrak T_{4,0}^-)
\bigr]
=-\frac{G_0}{4|t|},
\]
while the term $-a_0^\pm G_0$ in $\mathfrak T_{4,4}^\pm$
contributes $G_0/(8|t|)$. Additionally,  recall  the free sine contributes $G_0/(8|t|)$. Hence,
these three $G_0$ terms cancel.
Using $(a_2^+)^{-1}-(a_2^-)^{-1}=-1024i$ and $b_0=1/(8\pi)$,
the remaining leading term is
\[
-\frac{b_0^2}{2i|t|}
\bigl[(a_2^+)^{-1}-(a_2^-)^{-1}\bigr]XD_0X^*
=\frac{8}{\pi^2|t|}XD_0X^*.
\]
Hence, in
$\mathbb B(L^1_\omega,L^\infty_{-\omega})$,
\begin{equation}\label{eq:2d-sine-A-expansion}
\frac{\sin(t\sqrt H)}{t\sqrt H}P_{\mathrm{ac}}(H)\chi_1(H)
=\frac{8}{\pi^2|t|}XD_0X^*
+\mathcal O\bigl((|t|\log|t|)^{-1}\bigr).
\end{equation}
If $Q_{41}^1=0$, then $XD_0X^*=0$. If $Q_{41}^1\ne0$, then, for
$\psi\in Q_{41}^1L^2$,
\[
\langle Q_{41}^1vG_4vQ_{41}^1\psi,\psi\rangle
=4\sum_{i,j=1}^2|\langle x_ix_jv,\psi\rangle|^2>0
\quad (\psi\ne0).
\] Since $G_2v\mathfrak D^*T_0\psi$ is constant,
\eqref{eq:threshold-injectivity} proves that $X$ is injective on $Q_{41}^1L^2$.
Consequently $XD_0X^*\ne0$, and \eqref{eq:2d-sine-A-expansion}
gives the lower bound $|t|^{-1}$.
\end{proof}
\begin{proof}[Proof of Theorem~\ref{main_theorem_low_2}]
For part (i), combine Propositions~\ref{prop_K},
\ref{prop_second_2_3}, \ref{prop_second_1}, and~\ref{sine-sharp}
with the free estimates. For part (ii),
Proposition~\ref{prop_second_2_1} gives the unweighted bound.
The weighted block estimates above and
Propositions~\ref{sine-sharp-2} and~\ref{2-sine-sharp-2}
give the cosine bound and the two sine alternatives.
\end{proof}

\section{The third-kind resonance and eigenvalue cases}\label{sec:third_fourth}
In this section, we prove the low-energy estimates in
Theorems~\ref{main_theorem_low_31} and~\ref{main_theorem_low_32}
for third-kind resonances and zero eigenvalues. Together with Theorem~\ref{main_theorem_high}, these results give
Theorems~\ref{main_theorem-31} and~\ref{main_theorem-32}.

\begin{theorem}\label{main_theorem_low_31}
Let $H = \Delta^2 + V$ with $|V(x)| \lesssim \langle x \rangle^{-18-}$.
Assume that $H$ has no positive embedded eigenvalues and that zero is either a third-kind resonance or an eigenvalue accompanied by a d-wave resonance $(\text{i.e.,}\  Q_5 \neq 0 )$. Then
    \begin{equation}\label{optimality_31}
        \left\|\cos (t \sqrt{H}) P_{\mathrm{ac}}(H)\chi_1(H)\right\|_{L^1\to L^{\infty}} +
        \bigg\|\frac{\sin (t \sqrt{H})}{t\sqrt{H}} P_{\mathrm{ac}}(H)\chi_1(H)\bigg\|_{L^1 \to L^{\infty}} \sim \frac{1}{\log |t|}, \quad |t| \gg 1.
    \end{equation}
\end{theorem}

\begin{theorem}\label{main_theorem_low_32}
	Let $H = \Delta^2 + V$ with $|V(x)| \lesssim \langle x \rangle^{-22-}$.  
	Assume that $H$ has no positive embedded eigenvalues and that zero is an eigenvalue of $H$ but exhibits no d-wave resonance $(\text{i.e.,}\  Q_5 = 0 )$.
	\begin{enumerate}  
		\item If $H$ admits a p-wave resonance $(\text{i.e.,}\  Q_4 \neq 0 )$, then 
		$$  
		\begin{aligned}  
		\left\|\cos (t \sqrt{H}) P_{\mathrm{ac}}(H)\chi_1(H)\right\|_{L^1 \to L^{\infty}} + 
		\bigg\|\frac{\sin (t \sqrt{H})}{t\sqrt{H}} P_{\mathrm{ac}}(H)\chi_1(H)\bigg\|_{L^1 \to L^{\infty}} &\lesssim\langle t \rangle^{-1} {(\log (2+|t|))^2}.  
		\end{aligned}  
		$$  
		\item If $H$ lacks a p-wave resonance $(\text{i.e.,}\  Q_4 = 0 )$, then  
		\begin{equation*}  
		\left\|\cos(t\sqrt{H}) P_{\mathrm{ac}}(H)\chi_1(H)\right\|_{L^1 \to L^\infty} +  \bigg\|\frac{\sin(t\sqrt{H})}{t\sqrt{H}} P_{\mathrm{ac}}(H)\chi_1(H)\bigg\|_{L^1 \to L^\infty} \lesssim \langle t\rangle^{-1}.  
		\end{equation*}  
		Furthermore, we distinguish two subcases for the weighted estimate when $|t| \gg 1$:  
		\begin{itemize}  
			\item If ${S}_5L^2\neq\widetilde{S}_5L^2 $, then   
			\[   \left \|\cos (t \sqrt{H}) P_{\mathrm{ac}}(H)\chi_1(H)\right\|_{L^1_\omega\to L^\infty_{-\omega}} \lesssim\frac{1}{|t|\log |t|},\  \ \bigg\|\frac{\sin (t \sqrt{H})}{t\sqrt{H}} P_{\mathrm{ac}}(H)\chi_1(H)\bigg\|_{L^1_\omega\to L^\infty_{-\omega}}\sim|t|^{-1}.  
			\]  
			\item If ${S}_5L^2=\widetilde{S}_5L^2 $, then
			\[ 
			\left\|\cos (t \sqrt{H}) P_{\mathrm{ac}}(H)\chi_1(H)\right\|_{L^1_\omega \to L^{\infty}_{-\omega}} +  
			\bigg\|\frac{\sin (t \sqrt{H})}{t\sqrt{H}} P_{\mathrm{ac}}(H)\chi_1(H)\bigg\|_{L^1_\omega \to L^{\infty}_{-\omega}} \lesssim \frac{1}{|t| \log|t|}.  
			\] 
            \end{itemize}
		Here, the subspace characterizing the isotropic property is defined as:  
		\begin{equation} \label{S5wide}   
		\widetilde{S}_5L^2 = \left\{ \psi \in S_5 L^2 \;\middle|\;     
		\begin{aligned}    
		&\langle x_1^3 x_2 v, \psi \rangle = 0, \quad \langle x_1 x_2^3 v, \psi \rangle = 0, \\  
		&\langle x_1^4 v, \psi \rangle = \langle x_2^4 v, \psi \rangle = 3 \langle x_1^2 x_2^2 v, \psi \rangle    
		\end{aligned}    
		\right\}.    
		\end{equation}  
	\end{enumerate}  
\end{theorem}

\subsection{\texorpdfstring{The estimates for the case $Q_{5} \neq 0$}{Presence of a d-wave resonance}}     
The blocks with both indices at most $4$ were estimated in the
preceding sections. It remains to consider blocks containing the index $5$ or $6$. Theorem~\ref{thm:M_inverse}(III) gives bounded
symbols for the mixed blocks and, for $\alpha,\beta\in\{5,6\}$,
\[
\mathcal M_{\alpha,\beta}^\pm
=\mathcal A_{\alpha,\beta}(\lambda)+(\log\lambda)^{-2}\Gamma_{\alpha,\beta}^\pm(\lambda),
\]
where $\mathcal A_{\alpha,\beta}(\lambda)$ is independent of the sign $\pm$ and
$$
\left\|\partial_\lambda^{\ell} \mathcal{A}_{\alpha,\beta}(\lambda)\right\|_{\mathbb{B}\left(L^2\right)}
+\big\|\partial_\lambda^{\ell} \Gamma_{\alpha, \beta}^{ \pm}(\lambda) \big\|_{\mathbb{B}\left(L^2\right)} \lesssim \lambda^{-\ell},
\ \ell=0,1,2.
$$
Moreover, for $ \alpha=5,6, \ 0 \le \beta \le 4$ and $\ell=0,1,2,$
\begin{align}\label{3--M}
\big\|\partial_\lambda^{\ell} \mathcal{M}_{\alpha,\beta}^\pm(\lambda)\big\|_{\mathbb{B}\left(L^2\right)}
+\big\|\partial_\lambda^{\ell} \mathcal{M}_{\beta,\alpha}^\pm(\lambda)\big\|_{\mathbb{B}\left(L^2\right)}
\lesssim \lambda^{-\ell}.
\end{align}

\begin{proposition}\label{third_merged_1} 
	Let $B^\pm(\lambda)=\mathcal{M}_{\alpha, \beta}^{\pm}(\lambda)$ with $\alpha,\beta \in \{5,6\}$. Then      
	\[      
	\sup _{x, y \in \mathbb{R}^2}\left|\left(\mathcal{K}_{B}^+-\mathcal{K}_{B}^-\right)(t,x,y)\right| +      
	\sup _{x, y \in \mathbb{R}^2}\left|\left(\mathcal{N}_{B}^+-\mathcal{N}_{B}^-\right)(t,x,y)\right|  \lesssim \frac{1}{\log(2+|t|)}.      
	\]      
\end{proposition}  
\begin{proof}
We use the decomposition
$B^\pm=\mathcal A_{\alpha,\beta}
+(\log\lambda)^{-2}\Gamma_{\alpha,\beta}^\pm$ from
Theorem~\ref{thm:M_inverse}(III).
The term $\mathcal A_{\alpha,\beta}$ is independent of the sign $\pm$,
so its most singular contribution cancels in the kernel difference.
Lemma~\ref{lemma_projection}(iii) gives
\begin{equation}\label{ex-Q56}
(Q_jvR_0^\pm(\lambda^4))(\cdot,z)
=b_0(Q_jvG_2^0)(\cdot,z)
+\lambda^{m_j}
(\mathcal T_j+\mathcal T_{j,\pm})(\lambda,\cdot,z),
\qquad m_5=1,\quad m_6=2.
\end{equation}
For $j=5,6$, the bounds
\eqref{lemma_projection_G}--\eqref{lemma_projection_1} imply
\begin{equation}\label{L2-bound-56}
\begin{aligned}
\|(Q_jvG_2^0)(\cdot,z)\|_2&\lesssim1,\quad
\|\partial_\lambda^\ell\mathcal T_j(\lambda,\cdot,z)\|_2
\lesssim\lambda^{-\ell},\\
\left\|\partial_\lambda^\ell
\bigl(e^{\mp i\lambda|z|}
\mathcal T_{j,\pm}(\lambda,\cdot,z)\bigr)\right\|_2
&\lesssim\lambda^{-\ell}\langle\lambda z\rangle^{-1/2},
\qquad 0\le\ell\le2.
\end{aligned}
\end{equation}
Since $k_\alpha=k_\beta=3$, the prefactor
$\lambda^{4-k_\alpha-k_\beta}$ in \eqref{K_1} and
\eqref{Lambda_1} is $\lambda^{-2}$.
After extracting the spatial phases, the resulting integrals
have the form
\begin{equation}\label{3-4res-integral}
\begin{aligned}
&\frac{2}{\pi i}\int_0^\infty
\cos(t\lambda^2)\lambda
e^{\pm i\lambda r}\mathcal E^\pm(\lambda,x,y)\,d\lambda,\quad \text{and}\quad
\frac{2}{\pi it}\int_0^\infty
\sin(t\lambda^2)\lambda^{-1}
e^{\pm i\lambda r}\mathcal E^\pm(\lambda,x,y)\,d\lambda.
\end{aligned}
\end{equation}

\emph{The contribution of $\mathcal A_{\alpha,\beta}$.}
Set $B=\mathcal A_{\alpha,\beta}(\lambda)$.
Since $B$ is independent of the sign,
\[
\begin{aligned}
&R_0^+(\lambda^4)vQ_\alpha BQ_\beta vR_0^+(\lambda^4)
-R_0^-(\lambda^4)vQ_\alpha BQ_\beta vR_0^-(\lambda^4)\\
&\quad=
\bigl(R_0^+(\lambda^4)-R_0^-(\lambda^4)\bigr)
vQ_\alpha BQ_\beta vR_0^+(\lambda^4)+
R_0^-(\lambda^4)vQ_\alpha BQ_\beta v
\bigl(R_0^+(\lambda^4)-R_0^-(\lambda^4)\bigr).
\end{aligned}
\]
The sign-independent terms in \eqref{ex-Q56} cancel, giving
\begin{equation}\label{difference5,6}
\bigl[Q_jv(R_0^+(\lambda^4)-R_0^-(\lambda^4))\bigr](\cdot,z)
=
\lambda^{m_j}
\bigl(\mathcal T_{j,+}-\mathcal T_{j,-}\bigr)
(\lambda,\cdot,z),
\qquad j=5,6.
\end{equation}
Thus every remaining term contains an oscillatory factor
$\lambda^{m_j}\mathcal T_{j,\pm}$ and gains at least one
power of $\lambda$.
Expanding the other projected resolvent in \eqref{ex-Q56},
the possible phases in \eqref{3-4res-integral} are
\[
r\in\{|x|,\ |y|,\ |x|+|y|,\ |x|-|y|\}.
\] 
By \eqref{L2-bound-56}, Leibniz's rule
and H\"older's inequality give
\begin{equation}\label{eq:amplitude_def}
|\partial_\lambda^\ell\mathcal E^\pm(\lambda,x,y)|
\lesssim
\lambda^{-1-\ell}\langle\lambda r\rangle^{-1/2}
\widetilde\chi_1(\lambda/2),
\qquad 0\le\ell\le2.
\end{equation}
Lemma~\ref{oscillatory}(ii), with $(\sigma,\nu)=(1,0)$,
bounds this contribution by $C\langle t\rangle^{-1/2}$
for both kernels.

When $\alpha=\beta=6$, the factor in
\eqref{difference5,6} is $\lambda^2$.
Hence \eqref{eq:amplitude_def} improves to
\[
|\partial_\lambda^\ell\mathcal E^\pm|
\lesssim
\lambda^{-\ell}\langle\lambda r\rangle^{-1/2}
\widetilde\chi_1(\lambda/2).
\]
Applying Lemma~\ref{oscillatory}(ii) with
$(\sigma,\nu)=(0,0)$ yields
\begin{equation}\label{mathcalA66}
\begin{aligned}
&\sup_{x,y}
\bigl|(\mathcal K_{\mathcal A_{6,6}}^+
-\mathcal K_{\mathcal A_{6,6}}^-)(t,x,y)\bigr|+
\sup_{x,y}
\bigl|(\mathcal N_{\mathcal A_{6,6}}^+
-\mathcal N_{\mathcal A_{6,6}}^-)(t,x,y)\bigr|
\lesssim\langle t\rangle^{-1}.
\end{aligned}
\end{equation}

\emph{The contribution of
$(\log\lambda)^{-2}\Gamma_{\alpha,\beta}^\pm$.}
Pairing $b_0Q_\beta vG_2^0$ with $b_0Q_\alpha vG_2^0$
gives the only term without an additional power of $\lambda$.
Its contributions to the two kernel differences are
$\frac{2}{\pi i}K_{\alpha,\beta}(t,x,y)$ and
$\frac{2}{\pi i}N_{\alpha,\beta}(t,x,y)$, respectively, where
\begin{equation}
\label{kn-sharp}
\begin{aligned}
K_{\alpha,\beta}
={}&b_0^2\int_0^\infty
\cos(t\lambda^2)\lambda^{-1}(\log\lambda)^{-2}
\widetilde\chi_1(\lambda)
\left\langle
(\Gamma_{\alpha,\beta}^+-\Gamma_{\alpha,\beta}^-)(\lambda)
(Q_\beta vG_2^0)(\cdot,y),
(Q_\alpha vG_2^0)(\cdot,x)
\right\rangle\,d\lambda,\\
N_{\alpha,\beta}
={}&\frac{b_0^2}{t}\int_0^\infty
\sin(t\lambda^2)\lambda^{-3}(\log\lambda)^{-2}
\widetilde\chi_1(\lambda)
\left\langle
(\Gamma_{\alpha,\beta}^+-\Gamma_{\alpha,\beta}^-)(\lambda)
(Q_\beta vG_2^0)(\cdot,y),
(Q_\alpha vG_2^0)(\cdot,x)
\right\rangle\,d\lambda.
\end{aligned}
\end{equation}
Both integrals converge at zero, since their integrands are
bounded by $C\lambda^{-1}|\log\lambda|^{-2}$.
Their common amplitude in the form
\eqref{3-4res-integral} has $r=0$ and satisfies
\[
|\partial_\lambda^\ell\mathcal E|
\lesssim
\lambda^{-2-\ell}|\log\lambda|^{-2}
\widetilde\chi_1(\lambda/2),
\qquad 0\le\ell\le2.
\]
Lemma~\ref{oscillatory}(ii), with $(\sigma,\nu)=(2,2)$,
therefore gives
\begin{equation}\label{k-n}
\sup_{x,y}|K_{\alpha,\beta}(t,x,y)|
+
\sup_{x,y}|N_{\alpha,\beta}(t,x,y)|
\lesssim\frac{1}{\log(2+|t|)}.
\end{equation}

Every other product contains at least one factor
$\lambda^{m_j}\mathcal T_j$ or
$\lambda^{m_j}\mathcal T_{j,\pm}$ from \eqref{ex-Q56}.
Its phase has $r\in\{0,|x|,|y|, |x|+|y|\}$, and
\[
|\partial_\lambda^\ell\mathcal E^\pm|
\lesssim
\lambda^{-1-\ell}|\log\lambda|^{-2}
\langle\lambda r\rangle^{-1/2}
\widetilde\chi_1(\lambda/2),
\qquad 0\le\ell\le2.
\]
Lemma~\ref{oscillatory}(ii), with $(\sigma,\nu)=(1,2)$,
bounds these terms by $C\langle t\rangle^{-1/2}$.

Combining the two parts gives
\[
\begin{aligned}
(\mathcal K_B^+-\mathcal K_B^-)(t,x,y)
&=\frac{2}{\pi i}K_{\alpha,\beta}(t,x,y)
+O(\langle t\rangle^{-1/2}),\quad
(\mathcal N_B^+-\mathcal N_B^-)(t,x,y)
=\frac{2}{\pi i}N_{\alpha,\beta}(t,x,y)
+O(\langle t\rangle^{-1/2}),
\end{aligned}
\]
uniformly in $x,y$.
Together with \eqref{k-n}, this proves the proposition.
\end{proof}

\begin{proposition}\label{third_merged_2} 
	Let $B^\pm(\lambda)=\mathcal{M}_{\alpha,\beta}^{\pm}(\lambda)$ with $\alpha=5,6$, $0\le\beta\le4$, or $0\le\alpha\le4$, $\beta=5,6$. Then
	\[
\sup_{x,y\in\mathbb{R}^2}\left|\mathcal{K}_{B}^\pm(t,x,y)\right|
	+ \sup_{x,y\in\mathbb{R}^2}\left|\mathcal{N}_{B}^\pm(t,x,y)\right|
	\lesssim{\langle t\rangle^{-1/2}} \log(2+|t|).
	\]
\end{proposition}
\begin{proof}
It suffices to take $\alpha\in\{5,6\}$ and $0\le\beta\le4$.
In \eqref{K_1} and \eqref{Lambda_1}, the factor satisfies
$\lambda^{4-k_\alpha-k_\beta}=\lambda^{1-k_\beta}.$ Then 
$\lambda^{1-k_\beta}Q_\beta vR_0^\pm$ is expressed as $\lambda^{-1}\Omega_{\beta,\pm}$
for $0\leq\beta\le3$ and $\lambda^{-1}\mathcal J_{4,\pm}$ for $\beta=4$.
For the other projected resolvent, use \eqref{ex-Q56}.
Its sign-independent terms give the spatial phase
$e^{\pm i\lambda|y|}$, while $\mathcal T_{\alpha,\mp}(\cdot,x)$
gives $e^{\pm i\lambda(|x|+|y|)}$ after taking the inner product.
The bound \eqref{3--M} for $\mathcal M_{\alpha,\beta}^\pm$
and Lemma~\ref{lemma_projection} therefore give, with
$r=|y|$ or $|x|+|y|$, respectively,
\[
|\partial_\lambda^\ell\mathcal E^\pm|
\lesssim\lambda^{-1-\ell}|\log\lambda|^{\delta_{\beta4}}
\langle\lambda r\rangle^{-1/2}\widetilde\chi_1(\lambda/2),
\qquad 0\le\ell\le2.
\]
Lemma~\ref{oscillatory}(ii) with
$(\sigma,\nu)=(1,-\delta_{\beta4})$ proves the bound.
\end{proof}

	It remains to establish the sharpness of the $(\log|t|)^{-1}$ bound.
	The preceding estimates show that, for $|t|\gg1$, 
the decay rate is determined by the kernels ${K}_{\alpha,\beta}(t,x,y)$ and ${N}_{\alpha,\beta}(t,x,y)$ defined in \eqref{kn-sharp}.
	To identify their leading-order behavior, we need
	a more precise decomposition of
	\(\Gamma_{\alpha,\beta}^\pm(\lambda)\).    Returning to the proof of Theorem~\ref{thm:M_inverse} in Section~\ref{subsec:proof_inverse}, we combine \eqref{Gamma55}--\eqref{Gamma66}
	to derive the refined expansion:  
	\[    
	\Gamma_{\alpha,\beta}^{\pm}(\lambda) = \mathcal{L}_{\alpha,\beta}^\pm + (-\log \lambda)^{-1/2} \Lambda^\pm(\lambda),
	\]    
where the leading terms are given by  
	\begin{equation}\label{mathcal_L} 
		\begin{aligned}  
			\mathcal{L}_{5,5}^{\pm} &= -a_3^\pm b_1^{-2}(Q_5 v G_6 v Q_5)^{-1}, \ \  
			&&\mathcal{L}_{5,6}^{\pm} = -a_3^\pm \mathcal{D}_{5,5} (Q_5 v G_6^0 v Q_6) \mathcal{D}_{6,6}, \\  
			\mathcal{L}_{6,5}^{\pm} &= -a_3^\pm \mathcal{D}_{6,6} (Q_6 v G_6^0 v Q_5) \mathcal{D}_{5,5}, \ \  
			&&\mathcal{L}_{6,6}^{\pm} = a_3^\pm b_1 \mathcal{D}_{6,6} (Q_6 v G_6^0 v Q_5) \mathcal{D}_{5,5} (Q_5 v G_6^0 v Q_6) \mathcal{D}_{6,6},  
		\end{aligned}  
	\end{equation}  
	with $\mathcal{D}_{5,5} = -(b_1 Q_5 v G_6 v Q_5)^{-1}$ and $\mathcal{D}_{6,6} = (b_1 Q_6 v G_6^0 v Q_6)^{-1}$.
	
Lemma~\ref{oscillatory}(ii), with $(\sigma,\nu)=(2,5/2)$,
bounds the remainder contribution by $O((\log(2+|t|))^{-3/2})$.
For the constant coefficients $\mathcal L_{\alpha,\beta}^\pm$,
Remark~\ref{remark:osci} bounds the region $|t|\lambda^2\gtrsim1$
by $O((\log(2+|t|))^{-2})$. We therefore retain only the region
$|t|\lambda^2\lesssim1$. In $\mathbb B(L^1,L^\infty)$, this gives
\[      
	\begin{aligned}      
		\cos (t \sqrt{H}) P_{\mathrm{ac}}(H)\chi_1(H) &= -\frac{2}{\pi i}\sum_{\alpha,\beta = 5}^{6} T_{\mathcal{K}_{\alpha,\beta}}(t) + {O}\Big((\log|t|)^{-3/2}\Big), \\      
		\frac{\sin (t \sqrt{H})}{t\sqrt{H}} P_{\mathrm{ac}}(H)\chi_1(H) &= -\frac{2}{\pi i}\sum_{\alpha,\beta = 5}^{6} T_{\mathcal{N}_{\alpha,\beta}}(t) + {O}\Big((\log|t|)^{-3/2}\Big),   
	\end{aligned}      
\]    
where $T_{\mathcal{K}_{\alpha,\beta}}(t)$ and $T_{\mathcal{N}_{\alpha,\beta}}(t)$ are integral operators defined by the respective kernels:    
\[
\begin{aligned}      
	\mathcal{K}_{\alpha,\beta}(t,x,y) &= c_{\alpha,\beta}(x,y)\int_0^{\infty} \cos(t\lambda^2) \lambda^{-1} \widetilde{\chi}_1(\lambda)\chi_1(t \lambda^2) (\log \lambda)^{-2} \, d\lambda, \\      
	\mathcal{N}_{\alpha,\beta}(t,x,y) &= {c_{\alpha,\beta}(x,y)}{t}^{-1} \int_0^{\infty} \sin(t\lambda^2) \lambda^{-3} \widetilde{\chi}_1(\lambda) \chi_1(t \lambda^2) (\log \lambda)^{-2} \, d\lambda,   
\end{aligned}
\]
with  $c_{\alpha,\beta}(x,y) = b_0^2 \big\langle (\mathcal{L}_{\alpha,\beta}^+ -\mathcal{L}_{\alpha,\beta}^-) (Q_\beta v G_2^0)(\cdot, y), (Q_\alpha v G_2^0)(\cdot, x) \big\rangle$.

The following lower bounds complete the proof of Theorem~\ref{main_theorem_low_31}.

\begin{proposition}\label{sharp-3-4}
Assume $Q_5\ne0$. For $|t|\gg1$, the leading operators satisfy  
\begin{equation}\label{low_bound_weighted_1}  
		\Big\|\sum_{\alpha,\beta = 5}^{6}T_{ \mathcal{K}_{\alpha,\beta}}(t)\Big\|_{L^1 \to L^\infty} \gtrsim \frac{1}{\log |t|}, \quad   
		\Big\|\sum_{\alpha,\beta = 5}^{6} T_{\mathcal{N}_{\alpha,\beta}}(t)\Big\|_{L^1 \to L^\infty} \gtrsim \frac{1}{\log |t|}.  
	\end{equation}  
\end{proposition}  
\begin{proof}
Choose $0\ne\psi\in Q_5L^2$ and set $\varphi=vU\psi$.
Since $v\in L^2$, Cauchy--Schwarz gives $\varphi\in L^1$.
The identities $Q_jT_0=0$ for $j=5,6$ imply
\[
b_0Q_5vG_2^0\varphi=-Q_5\psi=-\psi,
\qquad
b_0Q_6vG_2^0\varphi=-Q_6\psi=0.
\]
Hence,
\[
\big\langle T_{\mathcal K_{\alpha,\beta}}(t)\varphi,\varphi\big\rangle
=
\big\langle T_{\mathcal N_{\alpha,\beta}}(t)\varphi,\varphi\big\rangle
=0,
\qquad
(\alpha,\beta)\in\{(5,6),(6,5),(6,6)\},
\]
and only the $(5,5)$ term remains. By \eqref{mathcal_L},
its coefficient is
\[
\begin{aligned}
c_0
&:=\iint_{\mathbb R^2\times\mathbb R^2}
c_{5,5}(x,y)\varphi(y)\overline{\varphi(x)}\,dy\,dx=-(a_3^+-a_3^-)b_1^{-2}
\big\langle(Q_5vG_6vQ_5)^{-1}\psi,\psi\big\rangle
\ne0,
\end{aligned}
\]
where $(Q_5vG_6vQ_5)^{-1}$ is negative definite on $Q_5L^2$,
and $a_3^+-a_3^-=-i/9216$.

On the support of $\chi_1(t\lambda^2)$,
$|t|\lambda^2\le2\lambda_0<1$. Hence
$\cos(t\lambda^2)\ge\cos1$ and
$\sin(t\lambda^2)/(t\lambda^2)\ge\sin1$.
For $|t|$ sufficiently large, both cutoffs equal one on
$0\le\lambda\le\lambda_0^{1/2}|t|^{-1/2}$. Therefore both
scalar integrals defining $\mathcal K_{5,5}$ and
$\mathcal N_{5,5}$ are bounded below by a positive constant times
\[
\int_0^{\lambda_0^{1/2}|t|^{-1/2}}
\frac{d\lambda}{\lambda(\log\lambda)^2}
=\frac{2}{\log(|t|/\lambda_0)}.
\]
It follows that
\begin{equation}\label{eq:lower_bounds}
\bigg|\Big\langle
\sum_{\alpha,\beta=5}^6T_{\mathcal K_{\alpha,\beta}}(t)
\varphi,\varphi\Big\rangle\bigg|
\gtrsim(\log|t|)^{-1},
\qquad
\bigg|\Big\langle
\sum_{\alpha,\beta=5}^6T_{\mathcal N_{\alpha,\beta}}(t)
\varphi,\varphi\Big\rangle\bigg|
\gtrsim(\log|t|)^{-1}.
\end{equation}
Since $|\langle T\varphi,\varphi\rangle|
\le\|T\|_{L^1\to L^\infty}\|\varphi\|_1^2$ and $\varphi$
is fixed, these inequalities prove \eqref{low_bound_weighted_1}.
\end{proof}

\begin{proof}[Proof of Theorem~\ref{main_theorem_low_31}]
Propositions~\ref{third_merged_1} and~\ref{third_merged_2}, together
with the preceding block estimates, give the upper bound.
Proposition~\ref{sharp-3-4} gives a nonzero leading coefficient,
while the remainder is $O((\log|t|)^{-3/2})$. This proves the
lower bound.
\end{proof}

\subsection{\texorpdfstring{The estimates for the case $Q_{5} = 0$}{Absence of a d-wave resonance}}
Assume $Q_5=0$. The absence of this block improves the inverse
expansion. We record the coefficients needed below and prove them at
the end of Section~\ref{subsec:proof_inverse}.

\begin{lemma}\label{M_inverse_optimal}  
	Assume that zero is an eigenvalue of $H$ but has no $d$-wave resonance $(\text{i.e.,}\  Q_5 = 0)$. Suppose that $|V(x)| \lesssim \langle x \rangle^{-22-}$ and $0<\lambda\ll1$. Then 
	\begin{itemize}  
		\item[(i)] For $0 \le \alpha, \beta \le 4$, the estimates for $\mathcal{M}_{\alpha,\beta}^\pm(\lambda)$ established in Theorem~\ref{thm:M_inverse} remain valid.
		
		\item[(ii)] For $1 \le \alpha \le 3$ and $\ell = 0,1,2$,  
		\begin{align}  \label{M6123}
		\bigl\|\partial_\lambda^{\ell} \mathcal{M}_{6,\alpha}^\pm(\lambda)\bigr\|_{\mathbb{B}(L^2)}  
		+ \bigl\|\partial_\lambda^{\ell} \mathcal{M}_{\alpha,6}^\pm(\lambda)\bigr\|_{\mathbb{B}(L^2)}  
		\lesssim \lambda^{1-\ell}.  
		\end{align}
		
		\item[(iii)] For $(\alpha,\beta) =(6,0),\,(0,6),\,(6,4),\,(4,6),\,(6,6)$, the operators admit the following expansions:  
		\[  
		\begin{aligned}  
			\mathcal{M}_{6,0}^{\pm}(\lambda) &= \lambda (\log \lambda)Q_6\Lambda^\pm(\lambda)Q_0, \quad 
			\mathcal{M}_{0,6}^{\pm}(\lambda) = \lambda (\log \lambda)Q_0\Lambda^\pm(\lambda)Q_6,  \\  
			\mathcal{M}_{6,4}^{\pm}(\lambda) &=  \lambda (\log \lambda)^2 Q_6 \Lambda^\pm(\lambda)Q_{42}^1+\lambda (\log \lambda) Q_6 \Lambda^\pm(\lambda)Q_4,\\
		\mathcal{M}_{4,6}^{\pm}(\lambda) &= \lambda (\log \lambda)^2 Q_{42}^1 \Lambda^\pm(\lambda)Q_{6} +\lambda (\log \lambda) Q_4 \Lambda^\pm(\lambda)Q_6, \\ 
		\mathcal{M}_{6,6}^{\pm}(\lambda) &= \mathcal{D}_{6,6}
			+  \lambda^2 (\log \lambda)^2Q_6 \Lambda^\pm(\lambda)Q_6.
		\end{aligned}  
		\]  
	\end{itemize}
	
	Moreover, if we additionally assume $Q_4 = 0$, then the operators $\mathcal{M}_{\alpha,4}^{\pm}(\lambda)$ and $\mathcal{M}_{4,\alpha}^{\pm}(\lambda)$ vanish identically for all $0\leq \alpha\leq 6$, and the following refined expansions hold:
	\[  
	\begin{aligned}  
		\mathcal{M}_{6,0}^{\pm}(\lambda) &= -a_2^\pm (a_0^\pm)^{-1} \|V\|_{L^1}^{-1} \lambda \mathcal{D}_{6,6} Q_6 vG_4 vQ_0   
		+ \lambda^2 (-\log \lambda)^{3/2}Q_6 \Lambda^\pm(\lambda)Q_0, \\  
		\mathcal{M}_{0,6}^{\pm}(\lambda) &= -a_2^\pm (a_0^\pm)^{-1} \|V\|_{L^1}^{-1} \lambda Q_0 vG_4 vQ_6 \mathcal{D}_{6,6}  
		+ \lambda^2 (-\log \lambda)^{3/2} Q_0\Lambda^\pm(\lambda)Q_6,  \\  
		\mathcal{M}_{6,6}^{\pm}(\lambda) &= \mathcal{D}_{6,6} -a_4^\pm \lambda^2 \mathcal{D}_{6,6} Q_6 vG_8 vQ_6 \mathcal{D}_{6,6} \\  
		&\quad +(a_2^\pm)^2 (a_0^\pm)^{-1} \|V\|_{L^1}^{-1} \lambda^2 \mathcal{D}_{6,6} Q_6 vG_4 v  Q_0 vG_4 vQ_6 \mathcal{D}_{6,6} +\lambda^3 (-\log \lambda)^{3/2} Q_6\Lambda^\pm(\lambda)Q_6.  
	\end{aligned}  
	\]  
Here $\mathcal D_{6,6}=(b_1Q_6vG_6^0vQ_6)^{-1}$ 
and  $\Lambda^\pm(\lambda)$ denotes a generic $\lambda$-dependent operator in $\mathbb{B}(L^2)$ (possibly differing at each occurrence) satisfying
$\|\partial_\lambda^{\ell} \Lambda^\pm(\lambda)\|_{\mathbb{B}(L^2)} \lesssim \lambda^{-\ell}$ for $ \ell=0,1,2.
$
\end{lemma}

It remains to estimate the kernels associated with
$\mathcal M_{6,\alpha}^\pm$, $\mathcal M_{\alpha,6}^\pm$
($0\le\alpha\le4$), and $\mathcal M_{6,6}^\pm$.
We distinguish whether a $p$-wave resonance is present.

\textbf{Part I: $Q_4\ne0$ (presence of a $p$-wave resonance).}

\begin{proposition}\label{prop_part1}    
	Assume $Q_5 = 0$ and $Q_4 \neq 0$. The blocks with an index $6$ satisfy the following kernel estimates:  
	\begin{itemize}
		\item[(i)] If $B^\pm(\lambda) \in \{ \mathcal{M}_{\alpha,6}^{\pm}(\lambda), \mathcal{M}_{6,\alpha}^{\pm}(\lambda) \}$ for $1 \le \alpha \le 3$, then  
		\[
\sup_{x,y\in\mathbb{R}^2}\left|\mathcal{K}_{B}^\pm(t,x,y)\right|
+\sup_{x,y\in\mathbb{R}^2}\left|\mathcal{N}_{B}^\pm(t,x,y)\right|
	\lesssim \langle t\rangle^{-1}.
	\] 
	Moreover,  both $\mathcal{K}_{B}^\pm$ and $\mathcal{N}_{B}^\pm$  satisfy the  improved weighted bound $\omega(x)\omega(y)(|t|\log|t|)^{-1}$ for $|t| \ge 2$.
	In particular, the same unweighted and weighted bounds remain valid under the alternative assumption $Q_5 = Q_4 = 0$, since in this case $\mathcal{M}_{\alpha,6}^{\pm}(\lambda)$ and $\mathcal{M}_{6,\alpha}^{\pm}(\lambda)$ also satisfy the estimate \eqref{M6123}.
		\item[(ii)] If $B^\pm(\lambda) \in \{ \mathcal{M}_{0,6}^{\pm}(\lambda), \mathcal{M}_{6,0}^{\pm}(\lambda) \}$, then
		\[
\sup_{x,y\in\mathbb{R}^2}\left|\mathcal{K}_{B}^\pm(t,x,y)\right|
	+ \sup_{x,y\in\mathbb{R}^2}\left|\mathcal{N}_{B}^\pm(t,x,y)\right|
	\lesssim \langle t\rangle^{-1} \log(2+|t|).
	\]
		\item[(iii)] If $B^\pm(\lambda) \in \{ \mathcal{M}_{4,6}^{\pm}(\lambda), \mathcal{M}_{6,4}^{\pm}(\lambda), \mathcal{M}_{6,6}^{\pm}(\lambda) \}$, then  
		\[
\sup _{x, y \in \mathbb{R}^2}\left|\left(\mathcal{K}_{B}^+-\mathcal{K}_{B}^-\right)(t,x,y)\right| +
	\sup _{x, y \in \mathbb{R}^2}\left|\left(\mathcal{N}_{B}^+-\mathcal{N}_{B}^-\right)(t,x,y)\right|
	\lesssim \langle t\rangle^{-1} (\log(2+|t|))^{2}.
	\]
	\end{itemize}  
\end{proposition}   
\begin{proof}
	The proof follows the same canonical oscillatory-integral reduction
	used in the preceding propositions.  We use
	$
	Q_6vR_0^\pm(\lambda^4)
	=
	b_0Q_6vG_2^0
	+
	\lambda^2(\mathcal T_6+\mathcal T_{6,\pm}),
	$
	together with the projected resolvent expansions adjacent to
	$Q_0$, $Q_\alpha$ $(1\leq\alpha\leq3)$, and $Q_4$. Inserting
	these expansions and the bounds from Lemma
	\ref{M_inverse_optimal} into \eqref{K_1} and \eqref{Lambda_1}, we see that the
	resulting amplitudes satisfy the hypotheses of Lemma
	\ref{oscillatory}(ii) with $(\sigma,\nu)=(0,0)$ for the blocks
	$\mathcal M_{6,\alpha}^\pm,\mathcal M_{\alpha,6}^\pm$
	$(1\leq\alpha\leq3)$, with $(\sigma,\nu)=(0,-1)$ for the blocks
	$\mathcal M_{6,0}^\pm,\mathcal M_{0,6}^\pm$, and with
	$(\sigma,\nu)=(0,-2)$ for the blocks
	$\mathcal M_{6,4}^\pm,\mathcal M_{4,6}^\pm,\lambda^2 (\log \lambda)^2Q_6 \Lambda^\pm(\lambda)Q_6$.
	These three cases give respectively
	$\langle t\rangle^{-1}$,
	$\langle t\rangle^{-1}\log(2+|t|)$, and
	$\langle t\rangle^{-1}(\log(2+|t|))^2$. The sign-independent
	principal part $\mathcal D_{6,6}$ is handled as in
	\eqref{mathcalA66} and satisfies the stronger
	$\langle t\rangle^{-1}$ bound. 
    The weighted estimate in part (i)
	is obtained exactly as in Proposition \ref{prop_K_weight_1}.
\end{proof}

\begin{proof}[Proof of Theorem~\ref{main_theorem_low_32}(1)]
Proposition~\ref{prop_part1} treats the blocks containing $6$.
Propositions~\ref{prop_K}, \ref{prop_second_2_3}, and
\ref{prop_second_2_1} treat the remaining blocks. Adding the free
low-energy bounds proves the theorem.
\end{proof}  
\textbf{Part II: $Q_4=0$ (Absence of a $p$-wave resonance).}
When $Q_4=0$, all blocks with an index $4$ vanish.
Lemma~\ref{M_inverse_optimal} also improves the bounds for
$\mathcal M_{6,0}^\pm$ and $\mathcal M_{0,6}^\pm$
from $O_2(\lambda|\log\lambda|)$ to $O_2(\lambda)$,
and the bound for $\mathcal M_{6,6}^\pm-\mathcal D_{6,6}$
from $O_2(\lambda^2|\log\lambda|^2)$ to $O_2(\lambda^2)$, i.e.,
\begin{align}\label{Q4=Q5}
\mathcal M_{6,0}^\pm
&=\lambda Q_6\Lambda^\pm Q_0,\qquad
\mathcal M_{0,6}^\pm
=\lambda Q_0\Lambda^\pm Q_6,\qquad
\mathcal M_{6,6}^\pm
=\mathcal D_{6,6}+\lambda^2Q_6\Lambda^\pm Q_6.
\end{align}

We apply the argument of Proposition~\ref{prop_part1}
to these improved bounds. Insert the projected resolvent
decompositions from Lemma~\ref{lemma_projection} into
\eqref{K_1} and \eqref{Lambda_1}, and extract the spatial
phases $e^{\pm i\lambda r}$.
The amplitudes associated with
$\mathcal M_{6,0}^\pm$, $\mathcal M_{0,6}^\pm$, and
$\mathcal M_{6,6}^\pm-\mathcal D_{6,6}$
now satisfy \eqref{condit 2} with $(\sigma,\nu)=(0,0)$.
Lemma~\ref{oscillatory}(ii) therefore gives
$\langle t\rangle^{-1}$ for their kernel contributions,
without logarithmic loss.
The spectral difference associated with $\mathcal D_{6,6}$
has the same bound $\langle t\rangle^{-1}$ by \eqref{mathcalA66}.

Propositions~\ref{prop_K} and~\ref{prop_part1}(i)
treat the remaining blocks. Adding the free low-energy estimates
proves the unweighted bound in
Theorem~\ref{main_theorem_low_32}(2).
We next compute the leading terms to determine the weighted decay.

\begin{proposition}\label{prop_part2}  
	Under the assumptions of Lemma~\ref{M_inverse_optimal}, suppose
$Q_4=0$. For $|t|\ge2$, the low-energy propagators have the expansions in $\mathbb{B}(L^1_\omega, L^\infty_{-\omega})$:
	\[
		\begin{aligned}
			\cos(t\sqrt{H}) P_{\mathrm{ac}}(H) \chi_1(H) &= -\frac{2}{\pi i}\widetilde{{\mathcal{B}}} I_c(t) + \mathcal{O}\!\bigl( |t|^{-1} (\log |t|)^{-1} \bigr), \\[4pt]  
			\frac{\sin(t\sqrt{H})}{t\sqrt{H}} P_{\mathrm{ac}}(H) \chi_1(H) &= -\frac{2}{\pi i}\widetilde{\mathcal{B}} I_s(t) + \mathcal{O}\!\bigl( |t|^{-1} (\log |t|)^{-1} \bigr),  
		\end{aligned}
	\]
	where $I_c(t)=O_N(|t|^{-N})$ for every $N>0$,
$I_s(t)=\pi/(4|t|)+O(|t|^{-2})$, and 
	\begin{equation}  \label{mathcal B}
		\begin{split}
			\widetilde{\mathcal{B}} &= - b_0^2(a_4^+ - a_4^-) G_2^0 v \mathcal{D}_{6,6} Q_6 v G_8 v Q_6 \mathcal{D}_{6,6} v G_2^0 \\  
			&\quad + b_0^2\bigl((a_2^+)^2 (a_0^+)^{-1} - (a_2^-)^2 (a_0^-)^{-1}\bigr)\|V\|_{L^1}^{-1} G_2^0 v \mathcal{D}_{6,6} Q_6 v G_4 v Q_0 v G_4 v Q_6 \mathcal{D}_{6,6} v G_2^0.  
		\end{split}
	\end{equation}
\end{proposition}
\begin{proof}
By Propositions~\ref{prop_K_weight_1} and~\ref{prop_part1}(i),
the blocks involving an index $1,2,3$ satisfy the required weighted
bound. The free contribution minus the $(0,0)$ block also satisfies
this bound: their leading sine coefficients $G_0/(8|t|)$ cancel,
by Propositions~\ref{prop_free_sin} and~\ref{prop_K_weight_2},
while the free cosine satisfies the bound in
Remark~\ref{rmk:free_case}.
It remains to compute the $(6,0)$, $(0,6)$, and $(6,6)$ blocks.

The cancellations $Q_6(x^\gamma v)=0$ for $|\gamma|\le3$ give
\[
(Q_6vG_4)(\cdot,y)=Q_6(v|\cdot|^4),
\qquad
\|V\|_1^{-1}Q_6vG_4vQ_0vG_0=Q_6vG_4.
\]
We also use the adjoint of the second identity.
The free resolvent expansion gives
\[
\begin{aligned}
Q_0vR_0^\pm(\lambda^4)
&=a_0^\pm\lambda^{-2}Q_0vG_0+Q_0vE_1^\pm,\\
Q_6vR_0^\pm(\lambda^4)
&=b_0Q_6vG_2^0
+a_2^\pm\lambda^2Q_6vG_4+Q_6vE_2^\pm,
\end{aligned}
\]
where, for $0\le\ell\le2$,
\[
\begin{aligned}
\|\partial_\lambda^\ell(Q_0vE_1^\pm)(\cdot,y)\|_2
&\lesssim\lambda^{-\ell-}\langle y\rangle^4,\quad
\|\partial_\lambda^\ell(Q_6vE_2^\pm)(\cdot,y)\|_2
\lesssim\lambda^{4-\ell-}\langle y\rangle^6.
\end{aligned}
\]

Write $\mathcal E_{\alpha,\beta}^\pm$ for the amplitude in
\eqref{eq:def_mathcal_E} with
$B^\pm=\mathcal M_{\alpha,\beta}^\pm$.
Insert the preceding expansions and
Lemma~\ref{M_inverse_optimal} into these amplitudes.
In the $(6,6)$ block, the singular term
\[
b_0^2\lambda^{-2}\widetilde\chi_1(\lambda)
[G_2^0v\mathcal D_{6,6}vG_2^0](x,y)
\]
is independent of the sign and cancels before integration.
Apart from the cutoff, the constant operator coefficient from
the $(6,0)$ and $(0,6)$ spectral differences is
\[
-b_0(a_2^+-a_2^-)
\bigl(
G_2^0v\mathcal D_{6,6}vG_4
+G_4v\mathcal D_{6,6}vG_2^0
\bigr).
\]
The cross terms between $b_0Q_6vG_2^0$ and
$a_2^\pm\lambda^2Q_6vG_4$ in the $(6,6)$ block give its
negative. These terms therefore cancel.
The two $\lambda^2$ coefficients of
$\mathcal M_{6,6}^\pm$ leave precisely
$\widetilde{\mathcal B}$ in \eqref{mathcal B}. Thus
\[
\sum_{(\alpha,\beta)\in\{(6,0),(0,6),(6,6)\}}
\bigl(\mathcal E_{\alpha,\beta}^+
-\mathcal E_{\alpha,\beta}^-\bigr)
=
\widetilde\chi_1(\lambda)\widetilde{\mathcal B}(x,y)
+\mathcal E_{\mathrm{rem}}(\lambda,x,y).
\]

We now estimate this remainder.
After multiplication by the prefactors and the principal
resolvent terms, the inverse remainders in
Lemma~\ref{M_inverse_optimal} contribute
$O_2(\lambda|\log\lambda|^{3/2})$.
Terms containing a free-resolvent remainder contribute
$O_2(\lambda^{2-})$ with polynomial spatial weights.
The remaining products of the displayed principal terms also
have at least two additional powers of $\lambda$.
Consequently,
\[
|\partial_\lambda^\ell\mathcal E_{\mathrm{rem}}(\lambda,x,y)|
\lesssim
\lambda^{1/2-\ell}\widetilde\chi_1(\lambda/2)
\langle x\rangle^6\langle y\rangle^6,
\qquad 0\le\ell\le2.
\]
Lemma~\ref{oscillatory}(ii), with $r=0$ and
$(\sigma,\nu)=(-1/2,0)$, gives
\[
\begin{aligned}
&\left|
\int_0^\infty
\cos(t\lambda^2)\lambda
\mathcal E_{\mathrm{rem}}(\lambda,x,y)\,d\lambda
\right|+
\left|
\frac1t\int_0^\infty
\sin(t\lambda^2)\lambda^{-1}
\mathcal E_{\mathrm{rem}}(\lambda,x,y)\,d\lambda
\right|
\lesssim
|t|^{-5/4}\langle x\rangle^6\langle y\rangle^6.
\end{aligned}
\]
Hence the three blocks contribute
$\frac{2}{\pi i}\widetilde{\mathcal B}I_c(t)$ and
$\frac{2}{\pi i}\widetilde{\mathcal B}I_s(t)$,
respectively, up to these kernel remainders, where
\[
\begin{aligned}
I_c(t)
&=\int_0^\infty
\cos(t\lambda^2)\lambda\widetilde\chi_1(\lambda)\,d\lambda,\quad
I_s(t)
=\frac1t\int_0^\infty
\sin(t\lambda^2)\lambda^{-1}
\widetilde\chi_1(\lambda)\,d\lambda.
\end{aligned}
\]
As in the proof of Proposition~\ref{prop_K_weight_2},
the substitution $s=\lambda^2$ and integration by parts give
\[
I_c(t)=O_N(|t|^{-N})\quad\text{for every }N>0,
\qquad
I_s(t)=\frac{\pi}{4|t|}+O(|t|^{-2}).
\]

The unweighted estimates following \eqref{Q4=Q5} bound the
three block differences uniformly by $C|t|^{-1}$.
Also, \eqref{lemma_projection_G} makes
$\widetilde{\mathcal B}(x,y)$ uniformly bounded.
Since $I_c(t),I_s(t)=O(|t|^{-1})$, the two kernel remainders
have the same uniform bound.
Combining this with their
$C|t|^{-5/4}\langle x\rangle^6\langle y\rangle^6$ bound and
applying \eqref{eq:weight-interpolation} with $N=6$ and
$\delta=1/4$ gives
$C\omega(x)\omega(y)(|t|\log|t|)^{-1}$ for both remainders.
Finally, the minus sign in \eqref{id-RV} and the previously
estimated blocks give the stated expansions in
$\mathbb B(L^1_\omega,L^\infty_{-\omega})$.
\end{proof}

\begin{proof}[Proof of Theorem~\ref{main_theorem_low_32}(2)]
The estimates following \eqref{Q4=Q5} prove the unweighted bound.
Proposition~\ref{prop_part2} gives the weighted cosine bound and,
since $I_s(t)=\pi/(4|t|)+O(|t|^{-2})$, the two sine estimates
according as $\widetilde{\mathcal B}=0$ or
$\widetilde{\mathcal B}\ne0$.
Lemma~\ref{lemma:non-zero_1} below gives the required criterion
for $\widetilde{\mathcal B}=0$.
\end{proof}

\begin{lemma}\label{lemma:non-zero_1}  
	The operator $\widetilde{\mathcal{B}} \colon L^1_\omega \to L^\infty_{-\omega}$ is non-zero if and only if $S_5L^2 \neq \widetilde{S}_5L^2$.   
\end{lemma}  

\begin{proof}
Substituting the values of $a_0^\pm$, $a_2^\pm$, and $a_4^\pm$
from Lemma~\ref{lem-reso} into \eqref{mathcal B} gives
\[
\widetilde{\mathcal B}
=-\frac{ib_0^2}{16384}
G_2^0vQ_6\mathcal D_{6,6}\mathbf K
\mathcal D_{6,6}Q_6vG_2^0,
\qquad
\mathbf K=\frac1{36}Q_6vG_8vQ_6
-\|V\|_1^{-1}Q_6vG_4vQ_0vG_4vQ_6.
\]
The operator $G_2^0vQ_6:Q_6L^2\to L^\infty_{-\omega}$ is
injective by \eqref{eq:threshold-injectivity}.
The operator $Q_6vG_2^0:L^1_\omega\to Q_6L^2$ is onto:
if $\psi\in Q_6L^2$ is orthogonal to its range, then
\[
0=\langle Q_6vG_2^0g,\psi\rangle
=\int_{\mathbb R^2}g(x)\overline{(G_2^0v\psi)(x)}\,dx
\qquad(g\in C_c^\infty).
\]
Injectivity gives $\psi=0$, and the range is closed because
$Q_6L^2$ is finite dimensional.
Since $\mathcal D_{6,6}$ is invertible on $Q_6L^2$,
\[
\widetilde{\mathcal B}=0
\quad\Longleftrightarrow\quad
\mathbf K=0\ \text{on }Q_6L^2.
\]
The operator $\mathbf K$ is self-adjoint. For $f\in Q_6L^2$,
the cancellations of all moments of order at most three give
\[
\langle\mathbf Kf,f\rangle
=\frac49\sum_{i,j,k,l=1}^2
|\langle x_ix_jx_kx_lv,f\rangle|^2
+\frac43\sum_{i,j=1}^2
|\langle x_ix_j|x|^2v,f\rangle|^2
-\frac56|\langle |x|^4v,f\rangle|^2.
\]
Set $A=\langle x_1^4v,f\rangle$, $B=\langle x_2^4v,f\rangle$,
$C=\langle x_1^2x_2^2v,f\rangle$,
$D=\langle x_1^3x_2v,f\rangle$, and
$E=\langle x_1x_2^3v,f\rangle$. Then
\begin{equation}\label{eq:negative_squares}
\langle\mathbf Kf,f\rangle
=\frac89|A-B|^2+\frac1{18}|A+B-6C|^2
+\frac{32}9|D+E|^2+\frac89|D-E|^2\ge0.
\end{equation}
Thus $f\in\ker\mathbf K$ exactly when $A=B=3C$ and $D=E=0$.
By \eqref{S5wide}, these conditions define
$\widetilde S_5L^2$. Since $Q_6L^2=S_5L^2$, we conclude that
$\widetilde{\mathcal B}=0$ exactly when
$S_5L^2=\widetilde S_5L^2$.
\end{proof}  
\section{A characterization of zero resonances}\label{sec:proof_characterization} 
In this section, we characterize the threshold projection spaces by
solutions of $H\phi=0$. The proof uses moment cancellations and the
fundamental solution $b_0G_2^0$ to determine their spatial decay.
We use the spaces $W_{-\sigma}$ from Subsection~\ref{main result} and
the projections in Definition~\ref{defS}. The following proposition
and Definition~\ref{definition1} imply
Proposition~\ref{characterizations}.

\begin{proposition}\label{characterizations-1}
	Assume that $H=\Delta^2+V$, $V\not\equiv0$, and
$|V(x)|\lesssim\langle x\rangle^{-\mu}$ with $\mu>11$.
Then the following characterizations hold:
	\begin{itemize}
		\item[(i)] 
		$\psi \in S_{2} L^2$ if and only if $\psi = Uv\phi$ with $\phi \in W_{-2}(\mathbb{R}^2)$ satisfying $H\phi = 0$ and
		\begin{equation}\label{phi-phi0}
			\phi(x) = \phi_0(x) + \sum_{|\alpha|=1} C_\alpha x^\alpha \quad \text{for some } \phi_0 \in W_{-1}(\mathbb{R}^2).
		\end{equation}
		\item[(ii)] 
		$\psi \in S_{3} L^2$ if and only if $\psi = Uv\phi$ with $\phi \in W_{-1}(\mathbb{R}^2)$ satisfying $H\phi = 0$.
		\item[(iii)] 
		$\psi \in S_{4} L^2$ if and only if $\psi = Uv\phi$ with $\phi \in W_{0}(\mathbb{R}^2)$ satisfying $H\phi = 0$.
		\item[(iv)] 
		$\psi \in S_{5} L^2$ if and only if $\psi = Uv\phi$ with $\phi \in L^2(\mathbb{R}^2)$ satisfying $H\phi = 0$.
	\end{itemize}
\end{proposition}
In particular, $Q_4\ne0$ is equivalent to the existence of a $p$-wave resonance,
that is, a nonzero solution in $W_{-1}\setminus W_0$.
Similarly, $Q_5\ne0$ is equivalent to the existence of a $d$-wave resonance,
that is, a nonzero solution in $W_0\setminus L^2$.

\subsection{\texorpdfstring{Preliminary estimates for $G_2^0$}{Estimates for the fundamental solution}}
We first record an estimate for $G_2^0(x)=|x|^2\log|x|$.

\begin{lemma} \label{estimate_G_2^0}
	Let $j \in \mathbb{N}^+$ and $f \in L^{2,\sigma}(\mathbb{R}^2)$ with $\sigma > \max\{j, 2\} + 1$. Suppose $f$ satisfies 
\begin{equation}\label{eq:cancellation}
		\int_{\mathbb{R}^2} y^\alpha f(y) \, dy = 0 \quad \text{for all multi-indices } \alpha \text{ with } |\alpha| \le j-1.
	\end{equation}
	Then 
	\[
	\left| (G_2^0f)(x) \right| \lesssim \|f\|_{L^{2,\sigma}} \, \langle x \rangle^{2-j} (1 + \log \langle x \rangle), \quad \forall x \in \mathbb{R}^2.
	\]
\end{lemma}

\begin{proof}
Cauchy--Schwarz gives the bound for $|x|\le2$, since $\sigma>3$.
For $r=|x|>2$, split the integral at $|y|=r/2$. The outer part is
bounded by
\[
\|f\|_{L^{2,\sigma}}
\left(\int_{|y|>r/2}|y|^4(1+\log|y|)^2
\langle y\rangle^{-2\sigma}\,dy\right)^{1/2}
\lesssim\|f\|_{L^{2,\sigma}}r^{3-\sigma}(1+\log r).
\]
On $|y|\le r/2$, Taylor expansion to order $j-1$ gives
\[
\begin{aligned}
\int_{|y|\le r/2}G_2^0(x-y)f(y)\,dy
={}&-\sum_{|\alpha|\le j-1}\frac{(-1)^{|\alpha|}}{\alpha!}
\partial^\alpha G_2^0(x)
\int_{|y|>r/2}y^\alpha f(y)\,dy\\
&+j\sum_{|\alpha|=j}\frac{(-1)^j}{\alpha!}
\int_{|y|\le r/2}y^\alpha f(y)
\int_0^1(1-s)^{j-1}
\partial^\alpha G_2^0(x-sy)\,ds\,dy.
\end{aligned}
\]
Here \eqref{eq:cancellation} replaces each inner moment by the
negative of its outer moment. For $|\alpha|\le j-1$,
\[
\left|\int_{|y|>r/2}y^\alpha f(y)\,dy\right|
\lesssim\|f\|_{L^{2,\sigma}}r^{|\alpha|+1-\sigma},\qquad
|\partial^\alpha G_2^0(x)|\lesssim r^{2-|\alpha|}(1+\log r).
\]
Multiplying these bounds and summing over $|\alpha|\le j-1$ yields the same estimate as for the outer part.
The integral Taylor
remainder is bounded by
\[
Cr^{2-j}(1+\log r)\int_{|y|\le r/2}|y|^j|f(y)|\,dy
\lesssim\|f\|_{L^{2,\sigma}}r^{2-j}(1+\log r),
\]
because $|x-sy|\sim r$ for $0\le s\le1$ and $\sigma>j+1$.
Finally, $3-\sigma<2-j$ gives the claim.
\end{proof}

\subsection{Proof of Proposition \ref{characterizations-1}}
\begin{proof}
For $k=1,2,3$, we construct the correspondence between $S_{k+1}L^2$
and solutions $\phi\in W_{k-3}$ through $\psi=Uv\phi$; for $k=4$,
the solution space is $L^2$. In the case $k=1$, we also retain the
condition \eqref{phi-phi0}.

\textbf{Proof of ``$\bm{\Longrightarrow}$'' (From Subspace to Resonance).}  
Assume $\psi \in S_{{k}+1}L^2$. 
For ${k}=1,2$, Definition \ref{defS} gives $S_{2-{k}}T_0\psi = 0$. Thus $T_0\psi = (I - S_{2-{k}})T_0\psi = v \sum_{|\alpha| \le 2-{k}} C_{\alpha} x^\alpha$. Then
\begin{align}\label{k=1,2}
	U\psi = -b_0 v G_2^0 v\psi + T_0\psi = -b_0 v G_2^0 v\psi + v\sum_{|\alpha| \le 2-{k}} C_{\alpha }x^\alpha.
\end{align}
For ${k}=3,4$, the condition $PT_0\psi = 0$ yields $T_0\psi =(P+S_0) T_0\psi=0$, implying
\begin{align}\label{k=3,4}
	U\psi = -b_0 v G_2^0 v\psi.
\end{align}
By the cancellation properties of $S_{{k}+1}$ (Definition \ref{defS}), $f = v\psi$ satisfies \eqref{eq:cancellation} with $j=2$ (for ${k}=1$) and $j={k}$ (for ${k}=2,3,4$). Applying Lemma \ref{estimate_G_2^0}, we obtain:
\begin{equation}\label{estimate_phi_G}
	\left| G_2^0 v\psi(x) \right| \lesssim
	\begin{cases}
		1 + \log \langle x \rangle, & {k} = 1,2, \\
		\langle x \rangle^{2 - {k}} \big( 1 + \log \langle x \rangle \big), & {k} = 3,4.
	\end{cases}
\end{equation}
Thus $G_2^0v\psi\in W_{-1}$ for $k=1,2$, $G_2^0v\psi\in W_0$ for $k=3$, and $G_2^0v\psi\in L^2$ for $k=4$.

We now define the function  
\begin{equation}\label{expression_psi}
	\phi(x) :=
	\begin{cases}
		-b_0 G_2^0 v\psi(x) + \sum_{|\alpha| \le 2 - {k}} C_\alpha x^\alpha, & {k} = 1,2, \\
		-b_0 G_2^0 v\psi(x), & {k} = 3,4.
	\end{cases}
\end{equation}
Equations~\eqref{k=1,2} and~\eqref{k=3,4} give $\psi=Uv\phi$.
The function $\phi$ belongs to the required space and satisfies
\eqref{phi-phi0} when $k=1$. Since
$\Delta^2G_2^0=b_0^{-1}\delta_0$,
\[
(\Delta^2 + V)\phi = -b_0 \Delta^2 G_2^0 v\psi + V\phi = -v( Uv\phi) + V\phi = -V\phi + V\phi= 0,
\]
which completes the forward implication.

\textbf{Proof of ``$\bm{\Longleftarrow}$'' (From Resonance to Subspace).} 
Suppose $H\phi=0$ in the stated solution space and set $\psi=Uv\phi$.
The decay of $V$ gives $\psi\in L^2$ and $v\psi=V\phi\in L^{2,\sigma}$
for some $\sigma>5$. In particular, the moments used below converge
absolutely.

\textbf{Step 1. Moment cancellations.}
Choose $\eta\in C_c^\infty(\mathbb R^2)$ equal to one on the unit
ball and zero outside the ball of radius two. Set $\eta_R(x)=\eta(x/R)$
and $A_R=\{R\le|x|\le2R\}$. Since $\Delta^2x^\alpha=0$ for
$|\alpha|\le3$, testing the equation gives
\begin{equation}\label{psi-phi}
\int x^\alpha v\psi\,\eta_R\,dx
=-\int\phi\,\Delta^2(x^\alpha\eta_R)\,dx.
\end{equation}
For $s\ge0$, scaling on $A_R$ yields
\[
\left|\int\phi\,\Delta^2(x^\alpha\eta_R)\,dx\right|
\lesssim R^{|\alpha|+s-3}
\|\phi\|_{L^{2,-s}(A_R)}.
\]
If $k=2,3$, take $s=3-k+\varepsilon$ with $0<\varepsilon<1$.
For $|\alpha|\le k-1$, the exponent is at most $-1+\varepsilon$,
so the right side tends to zero. For $k=1$, write
$\phi=\phi_0+\sum_{|\beta|=1}C_\beta x^\beta$ as in
\eqref{phi-phi0}. The polynomial contributes zero to
\eqref{psi-phi}; applying the estimate to $\phi_0\in W_{-1}$
with $s=1+\varepsilon$ proves the moments of orders zero and one.
For $k=4$, take $s=0$. The estimate becomes
\[
\left|\int x^\alpha v\psi\,\eta_R\,dx\right|
\lesssim R^{|\alpha|-3}\|\phi\|_{L^2(A_R)}\longrightarrow0,
\qquad |\alpha|\le3.
\]
Dominated convergence gives the full moment
cancellations in all four cases.

\textbf{Step 2. Representation of the solution.}
Lemma~\ref{estimate_G_2^0} and these moments imply
\begin{equation}\label{estimate}
\phi+b_0G_2^0v\psi\in
\begin{cases}
W_{k-3},&k=1,2,3,\\
L^2,&k=4.
\end{cases}
\end{equation}
Moreover, $\Delta^2(\phi+b_0G_2^0v\psi)=0$. Its Fourier transform
is therefore supported at the origin and is a finite sum of
derivatives of the Dirac distribution. Thus the function is a
polynomial. By \eqref{estimate}, its degree is at most $2-k$ for
$k=1,2$, and it vanishes for $k=3,4$. This proves
\eqref{expression_psi} in every case.

\textbf{Step 3. Projection conditions.}
Multiplying \eqref{expression_psi} by $v$ gives
\[
T_0\psi=v(\phi+b_0G_2^0v\psi)
=\begin{cases}
v\displaystyle\sum_{|\alpha|\le2-k}C_\alpha x^\alpha,&k=1,2,\\
0,&k=3,4.
\end{cases}
\]
Hence $S_{2-k}T_0\psi=0$ for $k=1,2$ and $T_0\psi=0$ for
$k=3,4$. Together with Step~1, these are exactly the conditions in
Definition~\ref{defS} for $\psi\in S_{k+1}L^2$.

The correspondence is one-to-one. If $Uv\phi=0$, the representation
shows that $\phi$ is a polynomial of degree at most one and
$v\phi=0$. Since $V\not\equiv0$, this polynomial vanishes on a set
of positive measure, so $\phi=0$.
\end{proof}

The same representation also identifies the trace moment used in
Theorem~\ref{main_theorem-2}. If $\phi\in\mathcal E$, then
$\phi=c-b_0G_2^0(V\phi)$ and the moments of $V\phi$ of orders
$0$ and $1$ vanish.
Since
\[
\frac12\sum_{i,j=1}^2 y_i y_j\partial_{ij}G_2^0(r\theta)
=(\log r+\tfrac12)|y|^2+(\theta\cdot y)^2,
\]
the Taylor argument in Lemma~\ref{estimate_G_2^0}, retaining this
quadratic term and estimating the third-order remainder, gives
uniformly for $\theta\in\mathbb S^1$,
\begin{equation}\label{eq:second-far-field}
\begin{aligned}
G_2^0(V\phi)(r\theta)
&=(\log r+\tfrac12)\overline{\tau(\phi)}+\theta^{\!T}\overline{\mathfrak M_2(\phi)}\theta
+O(r^{-1}\log r),\\
\phi(r\theta)&=-b_0\overline{\tau(\phi)}\log r+O(1).
\end{aligned}
\end{equation}
The complex conjugates follow from our inner-product convention. The representation is locally bounded, so
$\ker\tau=\mathcal E\cap L^\infty$.

\section{\texorpdfstring{Asymptotic expansions of $(M^{\pm}(\lambda))^{-1}$ at zero energy}{Asymptotic expansions of the inverse at zero energy}}
\label{subsec:proof_inverse}

In this section, we prove Theorem~\ref{thm:M_inverse} and
Lemma~\ref{M_inverse_optimal} by block inversion and Neumann
expansion. We first record the two algebraic formulas used below.

\begin{lemma}[{\cite[Lemma 3.12]{JK}}]\label{lemma_1}
	Let $\mathfrak{X}, \mathfrak{Y}, X, Y$ be vector spaces. Let $\mathbf{A} : \mathfrak{X} \rightarrow \mathfrak{Y}, \mathbf{B} : X \rightarrow \mathfrak{X}, \mathbf{C} : \mathfrak{Y} \rightarrow Y$ be linear operators. Define $\mathbf{D}=\mathbf{C}\mathbf{A}\mathbf{B}$. If $\mathbf{D}^{-1}$ exists, then
	$$
\mathbf{A}^{-1}=\mathbf{B}\mathbf{D}^{-1}\mathbf{C}
	$$
	provided $\mathbf{B}$ is surjective and $\mathbf{C}$ is injective.
\end{lemma}

\begin{lemma}[{\cite[Lemma 2.3]{JN}}]\label{lemma_2}
	Let $A$ be an operator matrix on $\mathcal{H}=\mathcal{H}_1 \medoplus \mathcal{H}_2$:
	$$
	A=\left(\begin{array}{ll}
		a_{11} & a_{12} \\
		a_{21} & a_{22}
	\end{array}\right), \quad a_{i j}: \mathcal{H}_j \rightarrow \mathcal{H}_i, 1 \leq i, j \leq 2,
	$$
	where all four entries are bounded. Suppose $a_{11}$ has a bounded inverse. Define 
	$$
	\boldsymbol{d} := a_{22}-a_{21} a_{11}^{-1} a_{12}.
	$$
	Then the following statements hold:
	\begin{enumerate}
		\item[(i)] $A$ has a bounded inverse if and only if $\boldsymbol{d}$
		has a bounded inverse. Furthermore, if $\boldsymbol{d}$ has a bounded inverse, we have
		\begin{equation}\label{A-m}
		A^{-1}=\left(\begin{array}{cc}
			a_{11}^{-1} a_{12} \boldsymbol{d}^{-1} a_{21} a_{11}^{-1}+a_{11}^{-1} & -a_{11}^{-1} a_{12} \boldsymbol{d}^{-1}  \\
			-\boldsymbol{d}^{-1} a_{21} a_{11}^{-1}& \boldsymbol{d}^{-1}
		\end{array}\right) .
		\end{equation}
		\item[(ii)] If $A$ is self-adjoint, then $A > 0$ (strictly positive) if and only if $a_{11} > 0$ and $\boldsymbol{d} > 0$.
	\end{enumerate}
\end{lemma}

\begin{proof}[Proof of Theorem \ref{thm:M_inverse}]
For zero-energy type $\mathbf k\in\{0,1,2,3,4\}$,
Proposition~\ref{characterizations Q} gives the orthogonal decomposition
$L^2=\bigoplus_{\alpha=0}^{\mathbf{k}+2}Q_\alpha L^2$.
	 For  $0<\lambda\ll1$, we define an isomorphism $\mathbf{B}_\lambda: \bigoplus_{\alpha=0}^{\mathbf{k}+2} Q_\alpha L^2 \to L^2$ by
\[
\mathbf{B}_\lambda(f_\alpha)_{\alpha=0}^{\mathbf{k}+2} := \sum_{\alpha=0}^{\mathbf{k}+2} \lambda^{-k_\alpha} \sigma_\alpha(\lambda) Q_\alpha f_\alpha,
	\]
	where the scaling factors $\sigma_\alpha(\lambda)$ are given by	\begin{equation}\label{tab:sigma_alpha}
		\begin{array}{c|ccccccc}
			\alpha & 0 & 1 & 2 & 3 & 4 & 5 & 6 \\
			\hline
			\sigma_\alpha(\lambda) & 1 & (-\log \lambda)^{-1/2} & 1 & (-\log \lambda)^{1/2} & 1 & (-\log \lambda)^{-1/2} & 1
		\end{array}
	\end{equation}
and $k_\alpha$ is given in \eqref{tab:k_alpha}. Since $\mathbf{B}_\lambda$ is bijective, Lemma~\ref{lemma_1} dictates that if $\mathbf{B}_\lambda^*M^{\pm}(\lambda) \mathbf{B}_\lambda$ is invertible on $\bigoplus_{\alpha=0}^{\mathbf{k}+2} Q_\alpha L^2$ for sufficiently small $\lambda > 0$, then
\begin{equation}\label{M-nin}
	(M^{\pm}(\lambda))^{-1} = \mathbf{B}_\lambda (\mathbf{B}_\lambda^* M^{\pm}(\lambda) \mathbf{B}_\lambda)^{-1} \mathbf{B}_\lambda^*.
\end{equation}
The resolvent expansions \eqref{Mpm-1}--\eqref{Mpm-3} and the
cancellations in Remark~\ref{cancellationQ} give
\[
\mathbf B_\lambda^*M^\pm(\lambda)\mathbf B_\lambda
=\lambda^{-2}(\mathbf D_{\mathbf k}^\pm+\mathbf S_{\mathbf k}^\pm(\lambda)),
\qquad
\mathbf D_{\mathbf k}^\pm=\bigl( E_{\alpha, \beta}^{\pm} \bigr)_{\alpha, \beta=0}^{\mathbf{k}+2},\quad
\mathbf S_{\mathbf k}^\pm=\bigl( r_{\alpha, \beta}^{\pm}(\lambda) \bigr)_{\alpha, \beta=0}^{\mathbf{k}+2}.
\]
The entries, indexed by $0\le\alpha,\beta\le\mathbf k+2$, are as follows.
	\begin{itemize}
		\item \textbf{Remainder entries.}
		\begin{equation}\label{partial}
			\begin{split}
				r_{\alpha,\beta}^{\pm}(\lambda) &= (-\log \lambda)^{-1/2} Q_\alpha T_0 Q_\beta + \lambda^2 (-\log \lambda)^{-1/2} \Lambda^\pm(\lambda),
				\quad \text{for } (\alpha,\beta) \in \{(1,2), (2,1)\}; \\
                r_{\alpha,\beta}^{\pm}(\lambda) &= b_1 (-\log \lambda)^{-1/2} Q_\alpha v G_6^0 v Q_\beta + \lambda^2 (-\log \lambda)^{-1/2} \Lambda^\pm(\lambda),
				\quad \text{for } (\alpha,\beta) \in \{(5,6), (6,5)\}; \\
                r_{1,1}^{\pm}(\lambda) &= (-\log \lambda)^{-1} \left( a_1^{\pm} Q_1 v G_2 v Q_1 + Q_1 T_0 Q_1 \right) + \lambda^2 (-\log \lambda)^{-1} \Lambda^\pm(\lambda); \\
				r_{\alpha,\beta}^{\pm}(\lambda) &= (-\log \lambda)^{-1} a_2^\pm Q_\alpha v G_4 v Q_\beta + \lambda^2 \Lambda^\pm(\lambda),
				\quad \text{for } (\alpha,\beta) \in \{(1,5), (5,1)\};\\
                r_{5,5}^{\pm}(\lambda) &= (-\log \lambda)^{-1} \left( a_3^\pm Q_5 v G_6 v Q_5 + b_1 Q_5 v G_6^0 v Q_5 \right) + \lambda^2 (-\log \lambda)^{-1} \Lambda^\pm(\lambda); \\
        r_{\alpha,\beta}^\pm(\lambda) &= b_1 \lambda \log \lambda Q_\alpha v G_6 v Q_\beta + \lambda \Lambda^\pm(\lambda),
				\quad \text{for } (\alpha,\beta) \in \{(4,6), (6,4)\};\\
				r_{\alpha,\beta}^\pm(\lambda) &=  \lambda a_2^\pm Q_\alpha v G_4 v Q_\beta + \lambda^3 (-\log \lambda) \Lambda^\pm(\lambda),
				\quad \text{for } (\alpha,\beta) \in \{(0,6), (6,0)\};  
			\end{split}
		\end{equation}
		and for all other index pairs $(\alpha,\beta)$,
		\begin{align}\label{partial2}
			r_{\alpha,\beta}^{\pm}(\lambda) = \lambda (-\log \lambda)^{3/2} \Lambda^\pm(\lambda).
		\end{align}
Here $\Lambda^\pm(\lambda)=O_2(1)$ has the meaning fixed in
Theorem~\ref{thm:M_inverse}.

		\item \textbf{Leading matrix for $\mathbf k=4$.}
		\begin{equation}\label{D_k}
			\mathbf{D}_{4}^{\pm} =
			\begin{pmatrix}
				a_0^\pm Q_0vG_0vQ_0 & 0 & E_{0,4}^{\pm} & 0 \\
				0 & \mathbf{D}_{1,1} & 0 & 0 \\
				E_{4,0}^{\pm} & 0 & a_2^\pm Q_4vG_4vQ_4 & 0 \\
				0 & 0 & 0 & \mathbf{D}_{2,2}
			\end{pmatrix},
		\end{equation}
		where $E_{\alpha, \beta}^{\pm} = g_0^{\pm}(\lambda) Q_\alpha vG_2 vQ_\beta + Q_\alpha T_0 Q_\beta$ for $(\alpha, \beta) \in \{(0,4), (4,0)\}$, and 
		\begin{equation}\label{D}
			\mathbf{D}_{1,1} = \begin{pmatrix}
				-b_0 Q_1vG_2vQ_1 & 0 & Q_1T_0Q_3 \\
				0 & Q_2T_0Q_2 & 0 \\
				Q_3T_0Q_1 & 0 & 0
			\end{pmatrix}, \
			\mathbf{D}_{2,2} = \begin{pmatrix}
				-b_1 Q_5vG_6vQ_5 & 0 \\
				0 & b_1 Q_6vG_6^0vQ_6
			\end{pmatrix}.
		\end{equation}
		We abbreviate the entries of $\mathbf{D}_{1,1}$ and $\mathbf{D}_{2,2}$ as $E_{\alpha,\beta}$ (dropping the $\pm$), since they are independent of the sign $\pm$.
\item \textbf{General case for $0 \leq \mathbf{k} \leq 3$.}
		For general $\mathbf{k}$, the operator matrix $\mathbf{D}_{\mathbf{k}}^{\pm}$
		can be formally understood as  the restriction of $\mathbf{D}_{4}^{\pm}$ onto $\bigoplus_{\alpha \le \mathbf{k}+2} Q_\alpha L^2$.
		\end{itemize}
We detail the subsequent inversion for $\mathbf{k}=4$. By the algebraic nature of our construction, the explicit formula for $(M^{\pm}(\lambda))^{-1}$ when $\mathbf{k}=4$  recovers the general $\mathbf{k}$ case  upon setting $Q_\alpha = 0$ for all $\alpha > \mathbf{k}+2$.

{\textbf{Step 1. Exact Invertibility of the Leading Matrix $\mathbf{D}_{4}^{\pm}$.}}

	By \eqref{D_k}, it suffices to invert $\mathbf D_{1,1}$,
$\mathbf D_{2,2}$, and
$\mathbf A^\pm=(E_{\alpha,\beta}^\pm)_{\alpha,\beta\in\{0,4\}}$
on their respective projection spaces.
The spaces $Q_jL^2$ are finite-dimensional for $j\ne2$, so injectivity suffices for their diagonal blocks.
	
		\textbf{(i) Invertibility of $\mathbf D_{1,1}$.}
First, $E_{2,2} = Q_2 T_0 Q_2$ is invertible on $Q_2 L^2$ by the definition of $Q_2$ and Riesz-Fredholm theory.
For any nonzero $f \in Q_1 L^2$, we have
	$$
	\langle Q_1 v G_2 v Q_1 f, f \rangle = -2 \sum_{j=1}^2 |\langle x_j v, f \rangle|^2 < 0,
	$$
Hence the Schur complement is invertible if
$Q_1T_0Q_3$ is injective on $Q_3L^2$. If $f\in Q_3L^2$ and
$Q_1T_0f=0$, then $S_1T_0f=0$ because $Q_3L^2\subset S_2L^2$.
Thus $S_0T_0f=(Q_1+S_1)T_0f=0$, so
$f\in S_3L^2\cap Q_3L^2=\{0\}$.
Lemma~\ref{lemma_2} now gives the invertibility of $\mathbf D_{1,1}$.

\textbf{(ii) Invertibility of $\mathbf D_{2,2}$.}
 Since $\mathbf{D}_{2,2}$ is a diagonal operator matrix, it suffices to show that $E_{5,5}$ and $E_{6,6}$ are invertible on $Q_5 L^2$ and $Q_6 L^2,$ respectively. For any nonzero $f \in Q_5 L^2$, we have
	\[
	\langle E_{5,5}f, f \rangle = -b_1 \langle Q_5 v G_6 v Q_5 f, f \rangle
	= b_1 \Big( 12 \sum_{j=1}^{2} |\langle |x|^2 x_j v, f \rangle|^2 + 8 \sum_{i,j,\ell=1}^{2} |\langle x_i x_j x_\ell v, f \rangle|^2 \Big)>0.
	\]
	For $f\in Q_6L^2=S_5L^2$, the associated quadratic form evaluates to
	\[
	\langle E_{6,6} f, f \rangle = \int_{\mathbb{R}^2} |(\Delta^2)^{-1} v f(x)|^2 \, dx = \|b_0 G_2^0 v f\|_{L^2}^2.
	\]
	Proposition~\ref{characterizations-1}(iv) and \eqref{expression_psi} guarantee that $b_0 G_2^0 v f \in L^2$ and that $f =-b_0 U v G_2^0 v f$. Thus, if $\langle E_{6,6} f, f \rangle = 0$, then $G_2^0 v f = 0$, which yields $f = 0$. This establishes the strict positivity of $E_{6,6}$ on $Q_6 L^2$, ensuring its invertibility (see also \cite[Lemma 4.9]{LSY21} for proof).
		
	\textbf{(iii) Invertibility of $\mathbf A^\pm$.}
We compute its inverse separately for $Q_4^0=0$ and $Q_4^0\ne0$.
	
	\textbf{Case 1: $Q_4^0=0$.} Then $Q_4=Q_4^1$, so
	\begin{equation}\label{eq:cancel_Q_4}
		Q_4 v = Q_4(x_j v) = Q_4(|x|^2 v) = 0.
	\end{equation}
	Consequently, $Q_4 v G_2 v Q_0 = Q_0 v G_2 v Q_4 = 0$, reducing $\mathbf{A}^{\pm}$ to the  form  
	\[
	\mathbf{A}^{\pm} = \begin{pmatrix}
		a_0^\pm Q_0 v G_0 v Q_0 & Q_0 T_0 Q_4^1 \\
		Q_4^1 T_0 Q_0 & a_2^\pm Q_4^1 v G_4 v Q_4^1
	\end{pmatrix}.
	\]
Note that $a_0^\pm Q_0 v G_0 v Q_0= a_0^\pm\|V\|_{L^1}Q_0$ with $a_0^\pm=\pm i/8$; hence, $(a_0^\pm Q_0 v G_0 v Q_0)^{-1}=(a_0^\pm)^{-1}\|V\|_{L^1}^{-1}Q_0.$
	By Lemma~\ref{lemma_2}, $\mathbf{A}^{\pm}$ is invertible if and only if the Schur complement 
	\[
	\boldsymbol{d}^\pm := a_2^\pm Q_4^1 v G_4 v Q_4^1 - Q_4^1 T_0 Q_0 (a_0^\pm Q_0 v G_0 v Q_0)^{-1} Q_0 T_0 Q_4^1
	\]
	is invertible on $Q_4L^2$ (here $Q_4=Q_4^1$). For any $f \in Q_4 L^2$, using \eqref{eq:cancel_Q_4}, we 
	derive that
	\begin{align}\label{d-M}
		\langle \boldsymbol{d}^\pm f, f \rangle = \pm i \Bigl[ \sum_{j,k=1}^2 \tfrac{1}{128} |\langle x_j x_k v, f \rangle|^2 + \langle 8 (Q_0 v G_0 v Q_0)^{-1} Q_0 T_0 Q_4^1 f, Q_0 T_0 Q_4^1 f \rangle \Bigr].
	\end{align}
Since $(Q_0vG_0vQ_0)^{-1}>0$,  the condition $\langle \boldsymbol{d}^\pm f, f \rangle = 0$ forces $\langle x_j x_k v, f \rangle = 0$ for all $j,k$, as well as $P T_0 f = 0$.  Thus
$f\in S_4L^2\cap Q_4L^2=\{0\}$, proving invertibility.
Formula~\eqref{A-m} gives
$(\mathbf A^\pm)^{-1}=(\mathcal D_{i,j}^\pm)_{i,j\in\{0,4\}}$,
where  
\begin{equation}\label{D00}  
	\begin{aligned}  
		\mathcal{D}_{0,0}^{\pm} &= (a_0^\pm)^{-1}\|V\|_{L^1}^{-1}Q_0 + (a_0^\pm)^{-2} \mathcal{W} (\boldsymbol{d}^\pm)^{-1} \mathcal{W}^*,  \quad  
		\mathcal{D}_{0,4}^{\pm} = - (a_0^\pm)^{-1} \mathcal{W} (\boldsymbol{d}^\pm)^{-1},  \\  
		\mathcal{D}_{4,0}^{\pm} &= - (a_0^\pm)^{-1} (\boldsymbol{d}^\pm)^{-1} \mathcal{W}^*,\quad  
		\mathcal{D}_{4,4}^{\pm} = (\boldsymbol{d}^\pm)^{-1},  
	\end{aligned}  
\end{equation}  
where we have denoted $\mathcal{W} := \|V\|_{L^1}^{-1} Q_0 T_0 Q_4^1$.  
	
\textbf{Case 2: $Q_4^0\ne0$.} On
$(Q_0L^2\oplus Q_4^0L^2)\oplus Q_4^1L^2$, write
$\mathbf A^\pm=(\mathbf A_{ij}^\pm)_{i,j=1}^2$, with
	\begin{align}\label{Aij}
		\mathbf{A}_{11}^{\pm} &= \begin{pmatrix}
			a_0^\pm Q_0 v G_0 v Q_0 & g_0^{\pm}(\lambda) Q_0 v G_2 v Q_4^0 + Q_0 T_0 Q_4^0 \\[6pt]
			g_0^{\pm}(\lambda) Q_4^0 v G_2 v Q_0 + Q_4^0 T_0 Q_0 & a_2^\pm Q_4^0 v G_4 v Q_4^0
		\end{pmatrix}, \nonumber\\[8pt]
		\mathbf{A}_{12}^{\pm} &= \begin{pmatrix} Q_0 T_0 Q_{4}^1  \\ a_2^\pm Q_4^0 v G_4 v Q_{4}^1 \end{pmatrix}, \quad
		\mathbf{A}_{21}^{\pm} = \begin{pmatrix} Q_{4}^1 T_0 Q_0 & a_2^\pm Q_4^1 v G_4 v Q_4^0 \end{pmatrix}, \quad
		\mathbf{A}_{22}^{\pm} = a_2^\pm Q_4^1 v G_4 v Q_4^1.
	\end{align}
	
	First, we show that $\mathbf{A}_{11}^{\pm}$ is invertible on $Q_0 L^2 \oplus Q_4^0 L^2$. Write $\mathbf{A}_{11}^{\pm} = (\log \lambda) \mathbf{A}_0 + \mathbf{A}_1^\pm$ with
	\[
	\mathbf{A}_0 = \begin{pmatrix}
		0 & b_0 Q_0 v G_2 v Q_4^0 \\
		b_0 Q_4^0 v G_2 v Q_0 & 0
	\end{pmatrix},\quad
	\mathbf{A}_1^\pm = \begin{pmatrix}
		a_0^\pm Q_0 v G_0 v Q_0 & Q_0 T_0 Q_4^0 + a_1^{\pm} Q_0 v G_2 v Q_4^0 \\
		Q_4^0 T_0 Q_0 + a_1^{\pm} Q_4^0 v G_2 v Q_0 & a_2^\pm Q_4^0 v G_4 v Q_4^0
	\end{pmatrix}.
	\]
The off-diagonal entries of $\mathbf A_0$ are invertible between
the one-dimensional spaces $Q_0L^2$ and $Q_4^0L^2$. Indeed, for
$f\in Q_0L^2$,
	$$
	Q_4^0 v G_2 v Q_0 f = \langle f, v \rangle Q_4^0 (|x|^2 v)=0 \Longleftrightarrow
	\langle f, v \rangle = 0 \Longleftrightarrow f = 0.
	$$
The other off-diagonal entry is its adjoint, so
\[
\mathbf{A}_0^{-1} = \begin{pmatrix}
		0 & (b_0 Q_4^0 v G_2 v Q_0)^{-1} \\
		(b_0 Q_0 v G_2 v Q_4^0)^{-1} & 0
	\end{pmatrix}.
	\]
	For sufficiently small $\lambda>0$, the Neumann series yields
	\begin{equation}\label{eq:inverse_A_11}
		(\mathbf{A}_{11}^{\pm})^{-1} = (\log \lambda)^{-1} \mathbf{A}_0^{-1} - (\log \lambda)^{-2} \mathbf{A}_0^{-1} \mathbf{A}_1^\pm \mathbf{A}_0^{-1} + (\log \lambda)^{-3}\Lambda^\pm(\lambda).
	\end{equation}
	By Lemma~\ref{lemma_2}, $\mathbf{A}^{\pm} = \bigl( \mathbf{A}_{ij}^{\pm}\bigr)_{i,j=1}^2$ is invertible if and only if
	$$
	\begin{aligned}
		\boldsymbol{d}^\pm
		&= a_2^\pm Q_4^1 v G_4 v Q_4^1
		- \begin{pmatrix} Q_4^1 T_0 Q_0 & a_2^\pm Q_4^1 v G_4 v Q_4^0 \end{pmatrix}
		(\mathbf{A}_{11}^{\pm})^{-1}
		\begin{pmatrix} Q_0 T_0 Q_4^1 \\ a_2^\pm Q_4^0 v G_4 v Q_4^1 \end{pmatrix}
	\end{aligned}
	$$
is invertible on $Q_4^1 L^2$.  We analyze this  invertibility  by splitting $Q_4^1 L^2 = Q_{41}^1 L^2 \oplus Q_{42}^1 L^2$.

	$\bullet$ \textbf{Subcase 2a ($Q_{42}^1 = 0$):}  
		Here $Q_4^1 = Q_{41}^1$. We first consider the nontrivial case: $
		Q_{41}^1L^2\neq\{0\}.
		$  Since
\begin{equation}\label{eq:positive}
			a_2^\pm \langle Q_{4}^1 v G_4 v Q_{4}^1 f, f \rangle	=a_2^\pm \langle Q_{41}^1 v G_4 v Q_{41}^1 f, f \rangle = a_2^\pm \sum_{j,\ell=1}^2 4 |\langle x_j x_\ell v, f \rangle|^2,\quad \forall f\in Q_4^1 L^2=Q_{41}^1 L^2,
		\end{equation}
with $a_2^\pm = \pm i/512$, by definition of $Q_{41}^1,$  this quadratic form is nondegenerate. Thus, $a_2^\pm Q_{4}^1 v G_4 v Q_{4}^1$ is invertible on $Q_4^1 L^2$. Combining  this with $(\mathbf A_{11}^\pm)^{-1}=O_2(|\log\lambda|^{-1})$,
		yields
\[
(\boldsymbol d^\pm)^{-1}=(a_2^\pm)^{-1}D_0
+(\log\lambda)^{-1}\Lambda^\pm,
\qquad D_0=(Q_{41}^1vG_4vQ_{41}^1)^{-1}.
\]
	Substituting the expressions \eqref{eq:inverse_A_11} for $(\mathbf{A}_{11}^{\pm})^{-1}$ and \eqref{Aij} for $\mathbf{A}_{ij}^{\pm}$ into \eqref{A-m}, we deduce the block representation
		$
		(\mathbf{A}^\pm)^{-1} = \big( \mathcal{D}_{\alpha,\beta}^{\pm}(\lambda) \big)_{\alpha,\beta \in \{0,4\}},
		$
		where
\begin{equation}\label{D40D04D44_1}
			\begin{split}
				\mathcal{D}_{0,0}^{\pm}(\lambda) & = (\log \lambda)^{-2} Q_0 \Lambda^\pm(\lambda) Q_0, \\
				\mathcal{D}_{0,4}^{\pm}(\lambda) &= (\log \lambda)^{-1} \left[ Q_0 \Lambda^\pm(\lambda) Q_4^0 + Q_0 \Lambda^\pm(\lambda) Q_{4}^1 \right] , \\
				\mathcal{D}_{4,0}^{\pm}(\lambda) &= (\log \lambda)^{-1} \left[ Q_4^0 \Lambda^\pm(\lambda) Q_0 + Q_{4}^1 \Lambda^\pm(\lambda) Q_0 \right], \\
				\mathcal{D}_{4,4}^{\pm}(\lambda) &= (\log \lambda)^{-2} Q_4^0 \Lambda^\pm(\lambda) Q_4^0
				+ (\log \lambda)^{-1} \left[ Q_4^0 \Lambda^\pm(\lambda) Q_{4}^1 + Q_{4}^1 \Lambda^\pm(\lambda) Q_4^0 \right] \\
				& \quad + (\log \lambda)^{-1}Q_{4}^1 \Lambda^\pm(\lambda) Q_{4}^1+(a_2^\pm)^{-1} D_0.
			\end{split}
		\end{equation}

If $Q_{41}^1=0$, then the $Q_4^1$-block is absent.
Formula~\eqref{eq:inverse_A_11} gives \eqref{D40D04D44_1} directly,
with $D_0=0$.

		$\bullet$ \textbf{Subcase 2b ($Q_{42}^1 \neq 0$):} Observe that $Q_4 v G_4 v Q_{42}^1 = 0$. With respect to the decomposition $Q_4^1 L^2 = Q_{41}^1 L^2 \oplus Q_{42}^1 L^2$, we can express $\boldsymbol{d}^\pm=\bigl(d^\pm_{ij}\bigr)_{i,j=1}^2$, where
		\begin{equation}\label{eq:d_ij}
			\begin{split}
				d_{11}^\pm &= a_2^\pm Q_{41}^1 v G_4 v Q_{41}^1 + (\log \lambda)^{-1}\Lambda^\pm(\lambda), \ \
				d_{12}^\pm = -(\log \lambda)^{-1} a_2^\pm B + (\log \lambda)^{-2}\Lambda^\pm(\lambda), \\
				d_{21}^\pm &= -(\log \lambda)^{-1} a_2^\pm B^* + (\log \lambda)^{-2}\Lambda^\pm(\lambda), \ \
				d_{22}^\pm = (\log \lambda)^{-2} a_2^\pm \mathcal{D} + (\log \lambda)^{-3}\Lambda^\pm(\lambda),
			\end{split}
		\end{equation}
		with
		\[
		\begin{aligned}
			B &= Q_{41}^1 v G_4 v Q_4^0 (b_0 Q_0 v G_2 v Q_4^0)^{-1} Q_0 T_0 Q_{42}^1, \\
			\mathcal{D} &= Q_{42}^1 T_0 Q_0 (b_0 Q_4^0 v G_2 v Q_0)^{-1} Q_4^0 v G_4 v Q_4^0 (b_0 Q_0 v G_2 v Q_4^0)^{-1} Q_0 T_0 Q_{42}^1.
		\end{aligned}
		\]

	Assume $
		Q_{41}^1L^2\neq\{0\}
		$ and set
$D_0=(Q_{41}^1vG_4vQ_{41}^1)^{-1}$.
By \eqref{eq:positive}, for sufficiently small $\lambda>0$,
\[
(d_{11}^\pm)^{-1}=(a_2^\pm)^{-1}D_0
+(\log\lambda)^{-1}\Lambda^\pm(\lambda).
\]
 By Lemma~\ref{lemma_2}, $\boldsymbol{d}^\pm=\bigl(d^\pm_{ij}\bigr)_{i,j=1}^2$ is invertible if and only if
\begin{equation}\label{d_1}
			\boldsymbol{d}^\pm_1 := d_{22}^\pm - d_{21}^\pm (d_{11}^\pm)^{-1} d_{12}^\pm = (\log \lambda)^{-2} a_2^\pm \mathcal{S} + (\log \lambda)^{-3}\Lambda^\pm(\lambda)
		\end{equation}
		is invertible on $Q_{42}^1 L^2$, where
\begin{equation*}
			\mathcal{S} = \mathcal{D} - B^* D_0 B=\mathcal{W}_1^* \boldsymbol{d}_2 \mathcal{W}_1,
		\end{equation*}
		with $\mathcal{W}_1 = (b_0 Q_0 v G_2 v Q_4^0)^{-1} Q_0 T_0 Q_{42}^1$ and
		$$
		\boldsymbol{d}_2 = Q_4^0 v G_4 v Q_4^0 - Q_4^0 v G_4 v Q_{41}^1 (Q_{41}^1 v G_4 v Q_{41}^1)^{-1} Q_{41}^1 v G_4 v Q_4^0.
		$$
		Observe that $Q_{41}^1 v G_4 v Q_{41}^1$ is strictly positive on $Q_{41}^1 L^2$ and the block matrix
		\begin{equation}\label{QvGvQ}
			\begin{pmatrix}
				Q_{4}^0 v G_4 v Q_{4}^0 & Q_{4}^0 v G_4 v Q_{41}^1 \\
				Q_{41}^1 v G_4 v Q_{4}^0 & Q_{41}^1 v G_4 v Q_{41}^1
			\end{pmatrix}
		\end{equation}
		is strictly positive on $Q_4^0L^2\oplus Q_{41}^1L^2$. Indeed,
for $f$ in this space,
\[
\langle vG_4vf,f\rangle
=2|\langle |x|^2v,f\rangle|^2
+4\sum_{i,j=1}^2|\langle x_ix_jv,f\rangle|^2.
\]
If this vanishes, then $f\in S_4^1L^2$, which is orthogonal to
$Q_4^0L^2\oplus Q_{41}^1L^2$. Thus $f=0$, and
Lemma~\ref{lemma_2}(ii) gives $\boldsymbol d_2>0$ on $Q_4^0L^2$.
		
		If $f\in Q_{42}^1L^2$ and $\langle\mathcal Sf,f\rangle=0$,
then $\mathcal W_1f=0$, hence $PT_0f=0$.
Thus $f\in S_4L^2\cap Q_{42}^1L^2=\{0\}$, so $\mathcal S>0$.
Equation~\eqref{d_1} gives
		\[
		(\boldsymbol{d}_1^\pm)^{-1} = (\log \lambda)^{2} (a_2^\pm)^{-1} \mathcal{S}^{-1} + (\log \lambda)\Lambda^\pm(\lambda).
		\]
		The Schur formula \eqref{A-m} therefore gives the inverses of $\boldsymbol d^\pm$
and $\mathbf A^\pm$.
        
		Next, we turn to deriving the formula for $	(\mathbf{A}^{\pm})^{-1}.$
		On $Q_{41}^1L^2\oplus Q_{42}^1L^2$, \eqref{A-m} gives
		\begin{align*}
			(\boldsymbol{d}^\pm)^{-1}
			= \begin{pmatrix}
				(a_2^\pm)^{-1}\bigl(D_0 + D_0 B \mathcal{S}^{-1} B^* D_0\bigr) + (\log \lambda)^{-1}\Lambda^\pm(\lambda) & (\log \lambda) (a_2^\pm)^{-1}D_0 B \mathcal{S}^{-1} + \Lambda^\pm(\lambda) \\[4pt]
				(\log \lambda) (a_2^\pm)^{-1}\mathcal{S}^{-1} B^* D_0 + \Lambda^\pm(\lambda) & (\log \lambda)^2 (a_2^\pm)^{-1}\mathcal{S}^{-1} + (\log \lambda)\Lambda^\pm(\lambda)
			\end{pmatrix}.
		\end{align*}
		Insert this inverse and \eqref{eq:inverse_A_11} into \eqref{A-m}.
The vanishing relation $Q_4^0vG_4vQ_{42}^1=0$ and the diagonal
$O_2(|\log\lambda|^{-2})$ entries of $(\mathbf A_{11}^\pm)^{-1}$
give the following entries on $Q_0L^2\oplus Q_4L^2$.
The coefficients $\mathcal C^\pm,\mathcal C_1^\pm,\mathcal C_2^\pm$
are those in \eqref{eq:mathcal_C}:
\begin{equation}\label{D40D04D44_2}
        \begin{split}
				\mathcal{D}_{0,0}^{\pm}(\lambda) &= (\log \lambda)^{-2} Q_0 \Lambda^\pm(\lambda) Q_0,   \\
				\mathcal{D}_{0,4}^{\pm}(\lambda) &= (\log \lambda)^{-1} Q_0 \Lambda^\pm(\lambda) Q_{4}^0+  Q_0 \Lambda^\pm(\lambda) Q_{4}^1,   \\
				\mathcal{D}_{4,0}^{\pm}(\lambda) &= (\log \lambda)^{-1} Q_{4}^0 \Lambda^\pm(\lambda) Q_0 +Q_{4}^1 \Lambda^\pm(\lambda) Q_0,  \\
				\mathcal{D}_{4,4}^{\pm}(\lambda)  &=Q_4^0\mathcal{C}^\pm Q_4^0+(\log \lambda)^{-1} Q_4^0\Lambda^\pm(\lambda) Q_4^0
				+\bigl(Q_4^0 \Lambda^\pm(\lambda) Q_{41}^1 + Q_{41}^1 \Lambda^\pm(\lambda) Q_4^0\bigr) \\
				& + (\log \lambda)\Bigl[(Q_4^0 \mathcal{C}_1^\pm Q_{42}^1 +Q_{42}^1 \mathcal{C}_2^\pm Q_4^0) +(Q_{41}^1 \Lambda^\pm(\lambda) Q_{42}^1 +Q_{42}^1 \Lambda^\pm(\lambda) Q_{41}^1) \Bigr]     \\
				&+\bigl(Q_4^0 \Lambda^\pm(\lambda) Q_{42}^1 +Q_{42}^1 \Lambda^\pm(\lambda)  Q_4^0\bigr)+ Q_{41}^1 \Lambda^\pm(\lambda) Q_{41}^1 +(\log \lambda)^2 (a_2^\pm)^{-1} Q_{42}^1 \mathcal{S}^{-1} Q_{42}^1 \\
				& + (\log \lambda) Q_{42}^1 \Lambda^\pm(\lambda)Q_{42}^1.
        \end{split}
        \end{equation}

	If $Q_{41}^1=0$, the $Q_{41}^1$-block is absent and the same formulas hold
with $D_0=0$.

Thus, we have established the invertibility of $\mathbf{A}^{\pm}$ across all three possible cases: $Q^0_4=0$; $Q^0_4\neq 0$ with $Q_{42}^1=0$; and $Q^0_4\neq 0$ with $Q_{42}^1\neq 0$. Then we conclude that $\mathbf{D}_{4}^{\pm}$ is invertible on $\medoplus_{\alpha=0}^{6} Q_\alpha L^2$, with
\begin{equation}\label{inverse_D4}
\left(\mathbf{D}_{4}^{\pm}\right)^{-1} =
	\begin{pmatrix}
		\mathcal{D}_{0,0}^{\pm}(\lambda) & 0 & \mathcal{D}_{0,4}^{\pm}(\lambda) & 0 \\
		0 & \left(\mathbf{D}_{1,1}\right)^{-1} & 0 & 0 \\
		\mathcal{D}_{4,0}^{\pm}(\lambda) & 0 & \mathcal{D}_{4,4}^{\pm}(\lambda) & 0 \\
		0 & 0 & 0 & \left(\mathbf{D}_{2,2}\right)^{-1}
	\end{pmatrix}
	:= \left( \mathcal{D}_{\alpha,\beta}^\pm(\lambda) \right)_{\alpha,\beta=0}^6,
\end{equation}
The entries indexed by $\{0,4\}$ are given by \eqref{D00}, \eqref{D40D04D44_1},
or \eqref{D40D04D44_2}, according to the case. The other two blocks are
		\begin{align}\label{D-invert}
	\left(\mathbf{D}_{1,1}\right)^{-1}=
			\begin{pmatrix}
				\mathcal{D}_{1,1} & 0 & \mathcal{D}_{1,3} \\
				0 & \mathcal{D}_{2,2} & 0 \\
				\mathcal{D}_{3,1} & 0 & \mathcal{D}_{3,3}
			\end{pmatrix}, \quad
			\left(\mathbf{D}_{2,2} \right)^{-1}=
			\begin{pmatrix}
				\mathcal{D}_{5,5} & 0 \\
				0 & \mathcal{D}_{6,6}
			\end{pmatrix}.
		\end{align}
These entries are independent of $\lambda$ and the sign $\pm$. In particular,
		\begin{align*}
			\mathcal{D}_{2,2} = (Q_2 T_0 Q_2)^{-1}, \quad
			\mathcal{D}_{5,5} = (-b_1 Q_5 v G_6 v Q_5)^{-1}, \quad
			\mathcal{D}_{6,6} = (b_1 Q_6 v G_6^0 v Q_6)^{-1}.
		\end{align*}

        {\textbf{Step 2. Asymptotic expansion of $(M^{\pm}(\lambda))^{-1}$ via Neumann series.}}
    
We begin by estimating the matrix operators $\mathbf S_4^\pm(\mathbf D_4^\pm)^{-1}$.
From \eqref{D00}, \eqref{D40D04D44_1}, and
\eqref{D40D04D44_2}, 
\[
\Big\| \partial_\lambda^\ell \mathcal{D}_{\alpha,\beta}^{\pm}(\lambda) \Big\|_{\mathbb{B}(L^2)} \lesssim
\begin{cases}
\lambda^{-\ell}, & \text{if } (\alpha, \beta) \neq (4,4), \\[4pt]
|\log \lambda|^2 \lambda^{-\ell}, & \text{if } (\alpha, \beta) = (4,4),
\end{cases}
\]
for $\ell = 0,1,2$. Combining this with the remainder estimates \eqref{partial}--\eqref{partial2} for $r_{\alpha,j}^\pm(\lambda)$, and noting that each factor $\mathcal{D}_{4,4}^{\pm}(\lambda)$ is paired with a remainder entry $r_{\alpha,4}^\pm(\lambda)$ satisfying $r_{\alpha,4}^\pm=O_2(\lambda|\log\lambda|^{3/2})$, we obtain
\[
\mathbf{S}_{4}^{\pm}(\lambda)\bigl(\mathbf{D}_{4}^{\pm}\bigr)^{-1} = O_2\bigl((-\log \lambda)^{-1/2}\bigr).
\]
We then employ the Neumann series expansion in the small parameter $(-\log \lambda)^{-1/2}$:
\begin{align}
\left(\mathbf{B}_\lambda^* M^{\pm}(\lambda) \mathbf{B}_\lambda\right)^{-1}
&=\lambda^2 \left(\mathbf{D}_{4}^{\pm}\right)^{-1}\left(\Id +\mathbf{S}_{4}^{\pm}(\lambda)\left(\mathbf{D}_{4}^{\pm}\right)^{-1}\right)^{-1} \nonumber\\
&= \lambda^2 \Bigl[ \left(\mathbf{D}_{4}^{\pm}\right)^{-1} + \sum_{n=1}^{\infty}(-1)^n \left(\mathbf{D}_{4}^{\pm}\right)^{-1} \left( \mathbf{S}_{4}^{\pm}(\lambda) \left(\mathbf{D}_{4}^{\pm}\right)^{-1} \right)^n \Bigr] \nonumber\\
&:= \lambda^2 \left( \mathcal{D}_{\alpha,\beta}^\pm(\lambda) + \Upsilon_{\alpha,\beta}^\pm(\lambda) \right)_{\alpha,\beta=0}^{6}. \label{MM-inver}
\end{align}

The remainder matrix can be written as
\[
\left(\Upsilon_{\alpha,\beta}^\pm\right)_{\alpha,\beta=0}^{6}
=-(\mathbf D_4^\pm)^{-1}\mathbf S_4^\pm(\mathbf D_4^\pm)^{-1}
\bigl(\Id+\mathbf S_4^\pm(\mathbf D_4^\pm)^{-1}\bigr)^{-1}.
\]
Every occurrence of $\mathcal D_{4,4}^\pm=O_2(|\log\lambda|^2)$
in the first three factors is adjacent to
$r_{4,j}^\pm$ or $r_{j,4}^\pm=O_2(\lambda|\log\lambda|^{3/2})$.
Their worst product is $O_2(\lambda|\log\lambda|^{11/2})$;
all other entries are $O_2(|\log\lambda|^{-1/2})$.
The last inverse is $O_2(1)$, so
\begin{equation}\label{Upsilon}
\left\|\partial_\lambda^\ell \Upsilon_{\alpha,\beta}^\pm(\lambda)\right\|_{\mathbb{B}(L^2)} \lesssim (-\log \lambda)^{-1/2} \lambda^{-\ell}, \qquad \ell = 0,1,2.
\end{equation}

Finally, reverting the preconditioning step via \eqref{M-nin}, we obtain the full inverse operator on $L^2$:
\[
\left(M^{\pm}(\lambda)\right)^{-1} = \sum_{0 \le \alpha, \beta \le 6} \lambda^{2 - k_\alpha - k_\beta} Q_\alpha \mathcal{M}_{\alpha,\beta}^{\pm}(\lambda) Q_\beta,
\]
where the unified block operators are  defined by
\[
\mathcal{M}_{\alpha,\beta}^{\pm}(\lambda) := \sigma_\alpha(\lambda) \sigma_\beta(\lambda) \left(\mathcal{D}_{\alpha,\beta}^\pm(\lambda) + \Upsilon_{\alpha,\beta}^\pm(\lambda) \right).
\]

	{\textbf{Step 3. Extraction of Decay Rates for $\mathcal{M}_{\alpha,\beta}^{\pm}(\lambda)$.}}

Our final task is to track the powers of $\lambda$ and $\log \lambda$ appearing in $\mathcal{D}_{\alpha,\beta}^\pm$, $\Upsilon_{\alpha,\beta}^\pm.$
	
	\textbf{(1) Analysis of $\mathcal{M}_{\alpha,\beta}^{\pm}(\lambda)$ for $0 \leq \alpha, \beta \leq 4.$}
The entries
$\mathcal D_{0,3}^\pm,\mathcal D_{3,0}^\pm,
\mathcal D_{2,3}^\pm,\mathcal D_{3,2}^\pm$ vanish by
\eqref{inverse_D4}--\eqref{D-invert}. Multiplying the corresponding
remainders \eqref{Upsilon} by the factors in \eqref{tab:sigma_alpha}
gives \eqref{M0-1}; only the $(3,3)$ entry retains a factor
$|\log\lambda|$.

For $\mathcal{M}_{0,0}^{\pm}(\lambda)=\mathcal{D}_{0,0}^\pm(\lambda) + \Upsilon_{0,0}^\pm(\lambda)$, we now examine the refined structure.
Since $\mathcal{D}_{0,\beta}^\pm(\lambda)$ and $\mathcal{D}_{\beta,0}^\pm(\lambda)$ are nonzero only for $\beta \in \{0,4\}$, every term in $\Upsilon_{0,0}^\pm(\lambda)$ contains at least a factor $r_{i,j}^\pm$ with $i$ or $j$ in $\{0,4\}$. The estimates \eqref{partial}--\eqref{partial2} therefore yield
\begin{equation}\label{Upsilon-00-rough}
	\left\|
	\partial_\lambda^\ell
	\Upsilon_{0,0}^\pm(\lambda)
	\right\|_{\mathbb B(L^2)}
	\lesssim
	\lambda^{-\ell+1-},
	\qquad
	\ell=0,1,2.
\end{equation}
If $Q_4^0\ne0$, by \eqref{D40D04D44_1},
\eqref{D40D04D44_2}, and~\eqref{Upsilon-00-rough}, one has
 $$\mathcal{M}_{0,0}^{\pm}(\lambda) = (\log \lambda)^{-2} Q_0 \Lambda^\pm(\lambda) Q_0.$$
If $Q_4^0=0$, all leading entries are $O_2(1)$ by \eqref{D00}.
The same argument then gives
$\Upsilon_{0,0}^\pm=O_2(\lambda|\log\lambda|^{3/2})$, and hence
\begin{align*}
\mathcal M_{0,0}^\pm(\lambda)
=
\mathcal D_{0,0}^\pm
+
\lambda(-\log\lambda)^{3/2}
Q_0\Lambda^\pm(\lambda)Q_0.\end{align*}

For $0\le\alpha\le4$,
$\mathcal M_{\alpha,4}^\pm
=\sigma_\alpha(\mathcal D_{\alpha,4}^\pm+\Upsilon_{\alpha,4}^\pm)$.
The entries with reversed indices follow by adjointness.
Since $\mathcal D_{j,4}^\pm=0$ for $j\notin\{0,4\}$, every
Neumann product contributing to $\Upsilon_{\alpha,4}^\pm$ contains
$r_{i,j}^\pm$ with $j\in\{0,4\}$.
By Step~2, all leading entries are $O_2(1)$, except that $\mathcal D_{4,4}^\pm=O_2(|\log\lambda|^2)$ 
when $Q_4^0\ne0$ and $Q_{42}^1\ne0$.

If $Q_4^0 = 0$ or $Q_{42}^1 = 0$, then
\begin{align}\label{Upsilon44-2}
\Upsilon_{\alpha,4}^\pm(\lambda) = \lambda (-\log \lambda)^{3/2} \Lambda^\pm(\lambda) \qquad \text{for all } \alpha \in \{0,1,2,3,4\}.
\end{align}

If $Q_4^0 \neq 0$ and $Q_{42}^1 \neq 0$, the slowest-decaying term in $\Upsilon_{\alpha,4}^\pm(\lambda)$ is
\begin{align*}
\mathcal{D}_{\alpha,j}^{\pm}(\lambda)r_{j,4}^\pm(\lambda) \mathcal{D}_{4,4}^{\pm}(\lambda) &= \lambda (-\log \lambda)^{7/2} \Lambda^\pm(\lambda) \quad \text{for } \alpha \in \{0,1,2,3\},\\
\mathcal{D}_{4,4}^{\pm}(\lambda) r_{4,4}^\pm(\lambda) \mathcal{D}_{4,4}^{\pm}(\lambda) &= \lambda (-\log \lambda)^{11/2} \Lambda^\pm(\lambda)\quad \text{for } \alpha = 4,
\end{align*}
which implies that
\begin{align}\label{Upsilon44}
\Upsilon_{\alpha,4}^\pm(\lambda)= \lambda (-\log \lambda)^{7/2} \Lambda^\pm(\lambda) \quad \text{for } \alpha \in \{0,1,2,3\},\ \
	\Upsilon_{4,4}^\pm(\lambda)= \lambda (-\log \lambda)^{11/2} \Lambda^\pm(\lambda)\quad \text{for } \alpha = 4,
\end{align}
Thus, combining these estimates with the definition of $\sigma_\alpha(\lambda)$ and the expansions of $\mathcal{D}_{i,j}^{\pm}(\lambda)$ ($i,j \in \{0,4\}$) from \eqref{D00}, \eqref{D40D04D44_1}, and \eqref{D40D04D44_2}, we obtain the desired results for all $\mathcal{M}_{\alpha,4}^{\pm}(\lambda) = \sigma_\alpha(\lambda) \mathcal{D}_{\alpha,4}^\pm(\lambda) + \sigma_\alpha(\lambda) \Upsilon_{\alpha,4}^\pm(\lambda)$ with $0 \le \alpha \le 4$.

\textbf{(2) Analysis of $\mathcal{M}_{\alpha,\beta}^{\pm}(\lambda)$ and $\mathcal{M}_{\beta,\alpha}^{\pm}(\lambda)$ for $\beta \in \{5,6\}$ and $0 \leq \alpha \leq 6.$}
	
	For $0\le\alpha\le4$, it suffices to estimate
$\mathcal M_{\beta,\alpha}^\pm$ since
$\mathcal M_{\alpha,\beta}^\pm=(\mathcal M_{\beta,\alpha}^\mp)^*$.
	For $0\le\alpha\le4$ and $\beta\in\{5,6\}$,
the leading entry vanishes, so
$\mathcal M_{\beta,\alpha}^\pm
=\sigma_\beta\sigma_\alpha\Upsilon_{\beta,\alpha}^\pm$.
Then \eqref{tab:sigma_alpha} and \eqref{Upsilon} give the
claimed $O_2(1)$ bound.

For the case $(5,5)$, note that $\mathcal{D}_{\alpha,5}^\pm(\lambda)$ and $\mathcal{D}_{5,\alpha}^\pm(\lambda)$ are nonzero only for $\alpha = 5$. Combining this with \eqref{MM-inver} and the structure of matrix multiplication, we derive
\begin{align*}
		\mathcal{M}_{5,5}^{\pm}(\lambda) &= (-\log \lambda)^{-1} \Big[  \mathcal{D}_{5,5}-\mathcal{D}_{5,5}r_{5,5}^{\pm}(\lambda)\mathcal{D}_{5,5}+\sum_{j,l=0}^{6}\mathcal{D}_{5,5}r_{5,j}^{\pm}(\lambda)\mathcal{D}_{j,l}^{\pm}(\lambda)r_{l,5}^{\pm}(\lambda)\mathcal{D}_{5,5}+ (-\log \lambda)^
		{-\frac{3}{2}} \Lambda^\pm(\lambda)
		\Big].
	\end{align*}
	By the expression  \eqref{partial}--\eqref{partial2} of $r_{i,j}^{\pm}(\lambda)$ and most $\mathcal{D}_{j,l}^{\pm}(\lambda)=0$ from \eqref{inverse_D4}--\eqref{D-invert}, it follows that
	\begin{align*}
		&\mathcal{D}_{5,5}r_{5,5}^{\pm}(\lambda)\mathcal{D}_{5,5}= (-\log \lambda)^{-1}\Big[a_3^\pm b_1^{-2}(Q_5 v G_6 v Q_5)^{-1}+b_1  \mathcal{D}_{5,5}(Q_5 v G_6^0 v Q_5) \mathcal{D}_{5,5}\Big]+ \lambda^2 (-\log \lambda)^{-1} \Lambda^\pm(\lambda),\\
		&\sum_{j,l=0}^{6} \mathcal{D}_{5,5} r_{5,j}^{\pm}(\lambda) \mathcal{D}_{j,l}^{\pm}(\lambda) r_{l,5}^{\pm}(\lambda) \mathcal{D}_{5,5} = b_1^2 (-\log \lambda)^{-1} \mathcal{D}_{5,5} (Q_5 v G_6^0 v Q_6) \mathcal{D}_{6,6} (Q_6 v G_6^0 v Q_5) \mathcal{D}_{5,5} \\
		& \ \ \ \ \ \ \ \ \ \ \ \ \ \ \ \ \ \ \ \ \ \ \  \ \ \ \ \ \ \ \ \ \ \ \ \ \  \ \  \ \ \ \ \ \ \  + (-\log \lambda)^{-2} \Lambda^\pm(\lambda).
	\end{align*}
	Therefore, $\mathcal{M}_{5,5}^{\pm}(\lambda)$ admits the representation
	\[
	\mathcal{M}_{5,5}^\pm(\lambda) =  \mathcal{A}_{5,5}(\lambda) + (\log \lambda)^{-2} \Gamma_{5,5}^{\pm}(\lambda),
	\]
	where
	\begin{align}
		\mathcal{A}_{5,5}(\lambda) &= -(\log \lambda)^{-1}\mathcal{D}_{5,5} + (\log \lambda)^{-2} \left[ -b_1 \mathcal{D}_{5,5} (Q_5 v G_6^0 v Q_5) \mathcal{D}_{5,5}  + b_1^2 \mathcal{D}_{5,5} (Q_5 v G_6^0 v Q_6) \mathcal{D}_{6,6} (Q_6 v G_6^0 v Q_5) \mathcal{D}_{5,5} \right], \nonumber\\
		\Gamma_{5,5}^{\pm}(\lambda) &= -a_3^\pm b_1^{-2} (Q_5 v G_6 v Q_5)^{-1} + (-\log \lambda)^{-1/2} \Lambda^\pm(\lambda).\label{Gamma55}
	\end{align}

	For the cases $(5,6)$ and $(6,5),$
	by an analysis similar to the previous case, we obtain
	\begin{align*}
		\mathcal{M}_{5,6}^{\pm}(\lambda) &= (-\log \lambda)^{-1/2} \Upsilon_{5,6}^\pm(\lambda) \\
		&= (-\log \lambda)^{-1/2} \Big[ -\mathcal{D}_{5,5} r_{5,6}^{\pm}(\lambda) \mathcal{D}_{6,6} + \mathcal{D}_{5,5} r_{5,5}^{\pm}(\lambda) \mathcal{D}_{5,5} r_{5,6}^{\pm}(\lambda) \mathcal{D}_{6,6} \\
		&\quad - \mathcal{D}_{5,5} r_{5,6}^{\pm}(\lambda) \mathcal{D}_{6,6} r_{6,5}^{\pm}(\lambda) \mathcal{D}_{5,5} r_{5,6}^{\pm}(\lambda) \mathcal{D}_{6,6} + (-\log \lambda)^{-2} \Lambda^\pm(\lambda) \Big] \\
		&=  \mathcal{A}_{5,6}(\lambda) + (\log \lambda)^{-2} \Gamma_{5,6}^{\pm}(\lambda),
	\end{align*}
and $\mathcal M_{6,5}^\pm(\lambda)=\mathcal A_{6,5}(\lambda)+(\log\lambda)^{-2}\Gamma_{6,5}^\pm$. 
	Here $\mathcal A_{5,6}(\lambda)$ and $\mathcal A_{6,5}(\lambda)$ are $O_2(|\log\lambda|^{-1})$ and independent of the sign $\pm$, and
\begin{align}
\Gamma_{5,6}^{\pm}(\lambda) &= -a_3^\pm \mathcal{D}_{5,5} (Q_5 v G_6^0 v Q_6) \mathcal{D}_{6,6} + (-\log \lambda)^{-1/2} \Lambda^\pm(\lambda),\label{Gamma56} \\
\Gamma_{6,5}^\pm(\lambda) &=-a_3^\pm\mathcal D_{6,6}(Q_6vG_6^0vQ_5)\mathcal D_{5,5}
+(-\log\lambda)^{-1/2}\Lambda^\pm(\lambda). \label{Gamma65}
	\end{align}

For the case $(6,6)$, we have
	\begin{align*}
		\mathcal{M}_{6,6}^{\pm}(\lambda) &=\mathcal{D}_{6,6}+ \Upsilon_{6,6}^\pm(\lambda) \\
	&=\mathcal{D}_{6,6}+\mathcal{D}_{6,6} r_{6,5}^{\pm}(\lambda) \mathcal{D}_{5,5}r_{5,6}^{\pm}(\lambda) \mathcal{D}_{6,6} - \mathcal{D}_{6,6} r_{6,5}^{\pm}(\lambda) \mathcal{D}_{5,5} r_{5,5}^{\pm}(\lambda) \mathcal{D}_{5,5}r_{5,6}^{\pm}(\lambda) \mathcal{D}_{6,6} \\
		&\quad +\mathcal{D}_{6,6} r_{6,5}^{\pm}(\lambda) \mathcal{D}_{5,5} r_{5,6}^{\pm}(\lambda) \mathcal{D}_{6,6} r_{6,5}^{\pm}(\lambda) \mathcal{D}_{5,5} r_{5,6}^{\pm}(\lambda) \mathcal{D}_{6,6}+ (-\log \lambda)^{-5/2} \Lambda^\pm(\lambda)  \\
		&=  \mathcal{A}_{6,6}(\lambda) + (\log \lambda)^{-2} \Gamma_{6,6}^{\pm}(\lambda),
	\end{align*}
where
$\mathcal A_{6,6}=\mathcal D_{6,6}+O_2(|\log\lambda|^{-1})$ is independent of the sign $\pm$, and
\begin{align}
\Gamma_{6,6}^{\pm}(\lambda) &= a_3^{\pm }b_1 \mathcal{D}_{6,6} (Q_6 v G_6^0 v Q_5) \mathcal{D}_{5,5}(Q_5v G_6^0 v Q_6) \mathcal{D}_{6,6}  + (-\log \lambda)^{-1/2} \Lambda^\pm(\lambda).\label{Gamma66}
	\end{align}
\end{proof}

\begin{proof}[Proof of Lemma~\ref{M_inverse_optimal}]
Part~(i) follows from Theorem~\ref{thm:M_inverse}(III).
For the remaining claims, use the scaling \eqref{tab:sigma_alpha}
and remove the $Q_5$ component. With $I=\{0,1,2,3,4\}$,
write $\lambda^2\mathbf B_\lambda^*M^\pm\mathbf B_\lambda$ as
\[
\begin{pmatrix}A^\pm&r_{I,6}^\pm\\
r_{6,I}^\pm&\mathcal D_{6,6}^{-1}+r_{6,6}^\pm\end{pmatrix}.
\]
The free expansion through $a_4^\pm\lambda^6G_8$ gives
\begin{equation}\label{eq:eigen-refined-remainders}
\begin{aligned}
r_{6,0}^\pm&=a_2^\pm\lambda Q_6vG_4vQ_0+O_2(\lambda^3|\log\lambda|),\\
r_{6,1}^\pm&=O_2(\lambda^2|\log\lambda|^{1/2}),\quad
r_{6,2}^\pm=O_2(\lambda^2|\log\lambda|),\quad
r_{6,3}^\pm=O_2(\lambda^2|\log\lambda|^{3/2}),\\
r_{6,4}^\pm&=b_1\lambda\log\lambda\,Q_6vG_6vQ_4+O_2(\lambda),\\
r_{6,6}^\pm&=a_4^\pm\lambda^2Q_6vG_8vQ_6+O_2(\lambda^{4-\varepsilon}),
\end{aligned}
\end{equation}
where $0<\varepsilon<1$ is fixed.
The corresponding entries with reversed indices follow from
\( r_{j,6}^\pm= (r_{6,j}^\mp)^*.\)

The additional cancellations are
$Q_6vG_4vQ_j=0$ for $j=1,2,3$ and $Q_6vG_6vQ_{42}^1=0$.
In particular,
\[
r_{6,4}^\pm Q_{42}^1=O_2(\lambda),\qquad
r_{6,4}^\pm(Q_4-Q_{42}^1)=O_2(\lambda|\log\lambda|).
\]
Let $\mathcal D_I^\pm=(\mathcal D_{i,j}^\pm)_{i,j\in I}$.
Equations~\eqref{D00}, \eqref{D40D04D44_1}, and
\eqref{D40D04D44_2} give
\begin{equation}\label{eq:eigen-leading-block-orders}
\begin{aligned}
(Q_4-Q_{42}^1)\mathcal D_{4,4}^\pm(Q_4-Q_{42}^1)&=O_2(1),\\
(Q_4-Q_{42}^1)\mathcal D_{4,4}^\pm Q_{42}^1,\quad
Q_{42}^1\mathcal D_{4,4}^\pm(Q_4-Q_{42}^1)&=O_2(|\log\lambda|),\\
Q_{42}^1\mathcal D_{4,4}^\pm Q_{42}^1&=O_2(|\log\lambda|^2),\\
\mathcal D_{0,0}^\pm,\quad\mathcal D_{0,4}^\pm,\quad
\mathcal D_{4,0}^\pm&=O_2(1).
\end{aligned}
\end{equation}
The other nonzero entries have both indices in $J=\{1,2,3\}$
and are independent of $\lambda$.
For $T^\pm=r_{6,I}^\pm(A^\pm)^{-1}$, the Neumann formula
\eqref{MM-inver}, restricted to $I$, gives
\[
T^\pm=r_{6,I}^\pm\mathcal D_I^\pm
\bigl(\Id+(r_{i,j}^\pm)_{i,j\in I}\mathcal D_I^\pm\bigr)^{-1}.
\]
By \eqref{partial}--\eqref{partial2}, the matrix
$(r_{i,j}^\pm)_{i,j\in I}\mathcal D_I^\pm$ has the $J\times J$ entries of order $O_2(|\log\lambda|^{-1/2})$;
all other entries are $O_2(\lambda|\log\lambda|^{7/2})$.
A direct calculation shows that the Neumann terms beyond $r_{6,I}^\pm\mathcal D_I^\pm$
contribute $O_2(\lambda^2|\log\lambda|^N)$ to $T_0^\pm$ and $T_4^\pm$, while $T_j^\pm=O_2(\lambda^2|\log\lambda|^N)$
for $j\in J$, with some fixed $N>0$.
Combining these bounds with the leading
term $r_{6,I}^\pm\mathcal D_I^\pm$ gives
\begin{equation}\label{eq:eigen-inverse-row-orders}
\begin{aligned}
T_0^\pm&=O_2(\lambda|\log\lambda|),&
\sigma_jT_j^\pm&=O_2(\lambda^2|\log\lambda|^N)\quad(j\in J),\\
T_4^\pm Q_{42}^1&=O_2(\lambda|\log\lambda|^2),&
T_4^\pm(Q_4-Q_{42}^1)&=O_2(\lambda|\log\lambda|).
\end{aligned}
\end{equation}
The improved order of $r_{6,4}^\pm Q_{42}^1$ is needed here:
it removes the possible third logarithmic power in $T_4^\pm Q_{42}^1$.
These bounds and their reversed versions yield
$T^\pm r_{I,6}^\pm=O_2(\lambda^2|\log\lambda|^2)$, so
\[
F^\pm:=\bigl[\mathcal D_{6,6}^{-1}+r_{6,6}^\pm
-T^\pm r_{I,6}^\pm\bigr]^{-1}
=\mathcal D_{6,6}+O_2(\lambda^2|\log\lambda|^2).
\]
The Schur formula \eqref{A-m} gives the last inverse row as
$(-F^\pm T^\pm,F^\pm)$. Restoring the factors $\sigma_j$ proves
(ii)--(iii), since $\lambda^2|\log\lambda|^N=O(\lambda)$.

If $Q_4=0$, the preceding proof gives
$(A^\pm)^{-1}_{00}=(a_0^\pm)^{-1}\|V\|_1^{-1}Q_0
+O_2(\lambda|\log\lambda|^{3/2})$. Hence
\[
\begin{aligned}
T_0^\pm&=a_2^\pm(a_0^\pm)^{-1}\|V\|_1^{-1}
\lambda Q_6vG_4vQ_0+O_2(\lambda^2|\log\lambda|^{3/2}),\\
T^\pm r_{I,6}^\pm
&=(a_2^\pm)^2(a_0^\pm)^{-1}\|V\|_1^{-1}
\lambda^2Q_6vG_4vQ_0vG_4vQ_6+O_2(\lambda^3|\log\lambda|^{3/2}).
\end{aligned}
\]
Expand $F^\pm$ to first order in
$r_{6,6}^\pm-T^\pm r_{I,6}^\pm$; the quadratic error is $O_2(\lambda^4)$.
This gives the stated expansion for $\mathcal M_{6,6}^\pm$.
The expansion for $\mathcal M_{6,0}^\pm$ follows from $-F^\pm T_0^\pm$, and 
$\mathcal M_{0,6}^\pm=(\mathcal M_{6,0}^\mp)^*$.
\end{proof}

\section{The proofs of Lemmas \ref{lemma_projection} and \ref{oscillatory}}\label{section 8}
In this section, we prove Lemmas~\ref{lemma_projection} and
\ref{oscillatory} and record the two integral asymptotics used
in Proposition~\ref{prop_second_1}.

\subsection{Proof of Lemma \ref{lemma_projection}}
We use Taylor expansion and the moment cancellations of
$Q_\alpha$ to prove Lemma~\ref{lemma_projection}.  

\begin{lemma}[Lemma 3.5, \cite{LSY21}]\label{taylor}
Let $\lambda>0$, $x\ne0$, $\rho(z)=z/|z|$ for $z\ne0$, and
$\rho(0)=0$. Put
$\cos\xi_\theta=\langle x,\rho(y-\theta x)\rangle/|x|$.
For $1\le m\le4$,
\begin{equation}\label{eq:directional-taylor}
\begin{aligned}
F(\lambda|y-x|)
=\sum_{j=0}^{m-1}\frac{(-\lambda|x|)^j}{j!}
\mathcal L_j[F](\lambda|y|,\xi_0)
+\frac{(-\lambda|x|)^m}{(m-1)!}
\int_0^1(1-\theta)^{m-1}
\mathcal L_m[F](\lambda|y-\theta x|,\xi_\theta)\,d\theta.
\end{aligned}
\end{equation}
The assumptions for the four cases are:
\textup{(i)} $m=1$, $F\in C^1(\mathbb{R})$;
\textup{(ii)} $m=2$, $F\in C^2(\mathbb{R})$, $F'(0)=0$;
\textup{(iii)} $m=3$, $F\in C^3(\mathbb{R})$, $F'(0)=F''(0)=0$;
\textup{(iv)} $m=4$, $F\in C^4(\mathbb{R})$, $F^{(j)}(0)=0$ for $1\le j\le3$.
Here $\mathcal L_0[F]=F$ and, with $c=\cos\xi$, $s=\sin\xi$,
\begin{align*}
\mathcal L_1[F](r,\xi)&=cF'(r),\\
\mathcal L_2[F](r,\xi)&=c^2F''(r)+s^2r^{-1}F'(r),\\
\mathcal L_3[F](r,\xi)&=c^3F^{(3)}(r)
+3cs^2\bigl(r^{-1}F''(r)-r^{-2}F'(r)\bigr),\\
\mathcal L_4[F](r,\xi)&=c^4F^{(4)}(r)+6c^2s^2r^{-1}F^{(3)}(r)
+(15c^2s^2-3s^2)\bigl(r^{-3}F'(r)-r^{-2}F''(r)\bigr).
\end{align*}
\end{lemma}

\begin{remark}\label{remark:taylor-log}
{\rm  The profile $\widetilde G_1(p)=b_0p^2\log p\,\chi_1(p)$ belongs to \(C^1\bigl([0, \infty)\bigr)\) but fails to be \(C^2\bigl([0, \infty)\bigr)\) at the origin. However,
the conclusion \eqref{eq:directional-taylor} with $m=2$ remains valid for
	\(F(\lambda|y-x|)=\widetilde{G}_1(\lambda|y-x|)\) whenever \(x\neq0\), with the integrand
	understood almost everywhere in \(\theta\).

Indeed, for \(\varepsilon>0\), set $r_\varepsilon(\theta)=(\varepsilon^2+|y-\theta x|^2)^{1/2}$
and $h_\varepsilon(\theta)=\widetilde G_1(\lambda r_\varepsilon(\theta))$.
Since \(r_\varepsilon(\theta)>0\),
	\(h_\varepsilon\in C^2([0,1])\) and  \(\chi_1\in C_c^\infty(\mathbb{R})\), a direct calculation gives
	\[
	|h_\varepsilon''(\theta)|
	\lesssim
	\lambda^2|x|^2
	\Big(
	\bigl|\widetilde{G}_1''(\lambda r_\varepsilon(\theta))\bigr|
	+
	\frac{
		\bigl|\widetilde{G}_1'(\lambda r_\varepsilon(\theta))\bigr|
	}{
		\lambda r_\varepsilon(\theta)
	}
	\Big)\lesssim \lambda^2|x|^2\Bigl(1+\big|\log\bigl(\lambda r_\varepsilon(\theta)\bigr)\big|\chi_1\bigl(\lambda r_\varepsilon(\theta)\bigr)\Bigr).
	\]
If $y-\theta_0x=0$ for some $\theta_0\in[0,1]$, the right side is
bounded, uniformly for $0<\varepsilon\le1$, by
$C_{\lambda,x}(1+|\log|\theta-\theta_0||)$, which is integrable.
Otherwise the segment \(\{y-\theta x:0\le\theta\le1\}\) stays away from zero and the second derivatives
are uniformly bounded for small \(\varepsilon\). Moreover, for almost every \(\theta\),
$h_\varepsilon''\to\lambda^2|x|^2
\mathcal L_2[\widetilde G_1](\lambda|y-\theta x|,\xi_\theta)$ as $\varepsilon\to0.$
Dominated convergence proves \eqref{eq:directional-taylor} for this
profile with $m=2$.}
\end{remark}

\begin{proof}[Proof of Lemma~\ref{lemma_projection}]
We begin with a derivative bound used repeatedly in the following.  For any smooth profile $\mathcal{G}$ on $\mathbb{R}$ satisfying $|\partial_p^\ell \big(e^{\mp ip}\mathcal{G}(p)\big)| \lesssim \langle p \rangle^{-1/2}|p|^{-\ell}$,  applying Leibniz's rule together with the   inequality $\langle \lambda(\theta x-y) \rangle^{-1/2} \lesssim \langle \lambda y \rangle^{-1/2} \langle x \rangle^{1/2}$ ($\theta \in [0,1],  0<\lambda\ll1$) yields:
\begin{align}
	\left|\partial_\lambda^\ell \left( e^{\mp i\lambda|y|} \mathcal{G}(\lambda|\theta x-y|) \right)\right|
	&= \left|\partial_\lambda^\ell \left[ e^{\pm i\lambda(|\theta x-y|-|y|)} \left( e^{\mp i\lambda|\theta x-y|} \mathcal{G}(\lambda|\theta x-y|) \right) \right]\right| \nonumber \\
	&\lesssim \sum_{\ell_1+\ell_2=\ell} |x|^{\ell_1} \lambda^{-\ell_2} \langle \lambda(\theta x-y) \rangle^{-1/2} \lesssim \lambda^{-\ell} \langle \lambda y \rangle^{-1/2} \langle x \rangle^{\ell+1/2}. \label{eq:general_oscillatory_bound}
\end{align}
	The decay assumption \eqref{condition} makes all polynomial
	weights below integrable against $v$ in $L^2_x$. In each Taylor integral,
	$\xi$ denotes the angle $\xi_\theta$ in Lemma~\ref{taylor}, with
	the indicated spatial variable in place of $x$.
	
	\medskip
	
	\textit{Proof of part~{\rm (i)}.}
	Recall from \eqref{def:F_pm} that $R^\pm_0(\lambda^4)(x,y) = \lambda^{-2} F_\pm(\lambda|x-y|)$, where
\begin{align}\label{Fpm}
	{F}_\pm(p)=&\Bigl(a_0^\pm+b_0p^2\log p+a_1^\pm p^2
	+a_2^\pm p^4+O\bigl(p^{5}\bigr)\Bigr)\chi_1(p)+\frac{i}{8}\Bigl(\pm e^{\pm ip} w_\pm(p)-e^{-p}w_+(ip)\Bigr)\chi_2(p).
\end{align}
The expansion gives
\[
\bigl|\partial_p^\ell \bigl(e^{\mp ip} F_\pm(p)\bigr)\bigr| + \bigl|\partial_p^\ell \bigl(e^{\mp ip} F_\pm'(p)\bigr)\bigr| \lesssim \langle p \rangle^{-1/2} |p|^{-\ell}, \quad \ell = 0,1,2.
\]
For $\alpha = 0$, defining $\Omega_{0,\pm}(\lambda,x,y) := \left(Q_0 v F_\pm(\lambda|\cdot-y|)\right)(x)$ and directly applying \eqref{eq:general_oscillatory_bound} gives:
\begin{align*}
	\left\| \partial_\lambda^\ell \left(e^{\mp i \lambda|y|} \Omega_{0,\pm}(\lambda, x, y)\right) \right\|_{L^2_x}
	\lesssim \lambda^{-\ell} \langle \lambda y \rangle^{-1/2} \left\| v(x) \langle x \rangle^{\ell+1/2} \right\|_{L^2_x} \lesssim \lambda^{-\ell} \langle \lambda y \rangle^{-1/2}.
\end{align*}
For $\alpha \ge 1$, we exploit Lemma~\ref{taylor}(i) to write $$F_\pm(\lambda|x-y|) = F_\pm(\lambda|y|) - \lambda |x| \int_0^1 F'_\pm(\lambda|\theta x-y|) \cos\xi \, d\theta.$$
Since $Q_\alpha v = 0$ for $\alpha \ge 1$ (see Remark \ref{cancellationQ}), the leading term $F_\pm(\lambda|y|)$ is annihilated, yielding
\begin{align*}
	\left(Q_\alpha v R_0^{\pm}(\lambda^4)(\cdot, y)\right)(x)
	&= -\lambda^{-1} Q_\alpha \Big(|\cdot| v \int_0^1 F'_\pm(\lambda| \theta \cdot-y|)\cos\xi \, d\theta \Big)(x) := \lambda^{-1} \Omega_{\alpha,\pm}(\lambda, x, y).
\end{align*}
Applying \eqref{eq:general_oscillatory_bound} to $F'_\pm$ gives $\|\partial_\lambda^\ell (e^{\mp i \lambda|y|} \Omega_{\alpha,\pm})\|_{L^2_x} \lesssim \lambda^{-\ell} \langle \lambda y \rangle^{-1/2}$, which proves part (i).

	\medskip
	
	\textit{Proof of part~{\rm (ii)}.}
	By Remark \ref{cancellationQ},
	$
	Q_\alpha v
	=
	Q_\alpha(z_jv)
	=
	0
	$ for $ j=1,2$ and $ 2\leq\alpha\leq6.$
 Lemma \ref{taylor}(ii) and Remark \ref{remark:taylor-log} with $z\neq0$ give
	\begin{align*}
		&
		Q_\alpha
		\bigl(
		vF_\pm(\lambda|z-y|)
		\bigr)(x)=
		\lambda^2
		Q_\alpha
		\Bigl(
		|z|^2v
		\int_0^1
		(1-\theta)
		\mathcal{L}_2[F_\pm]
		\bigl(
		\lambda|y-\theta z|,
		\xi
		\bigr)
		\,d\theta
		\Bigr)(x),
	\end{align*}
	where $Q_\alpha$ acts on the $z$-variable.
   Thus, define
	$
	\bigl(Q_\alpha vR_0^\pm(\lambda^4)\bigr)(x,y)
	:=
	\mathcal{J}_{\alpha,\pm}(\lambda,x,y),
	$
	where
\begin{equation}\label{def-J}
		\mathcal{J}_{\alpha,\pm}(\lambda,x,y)
		=
		Q_\alpha
		\Bigl(
		|z|^2v
		\int_0^1
		(1-\theta)
		\mathcal{L}_2[F_\pm]
		\bigl(
		\lambda|y-\theta z|,
		\xi
		\bigr)
		\,d\theta
		\Bigr)(x).
	\end{equation}
	The expansion \eqref{Fpm} of $F_\pm$ implies
	\begin{equation}\label{L2-decomposition}
		\mathcal{L}_2[F_\pm](p,\xi)
		=
		2b_0(\log p)\chi_1(p)
		+
		\Theta_\pm(p,\xi),
	\end{equation}
		where the remainder $\Theta_\pm$ satisfies \eqref{eq:general_oscillatory_bound}, uniformly in $\xi$.
	It therefore remains only to estimate the logarithmic term. For
 $\ell=0,1,2$,
	\begin{align}
		&
		\left|
		\partial_\lambda^\ell
		\Bigl[
		e^{\mp i\lambda|y|}
		\log(\lambda |y-\theta x|)
		\chi_1(\lambda |y-\theta x|)
		\Bigr]
		\right|
	\lesssim
		\lambda^{-\ell}
		\langle\lambda y\rangle^{-1/2}
		\langle x\rangle^{\ell+1/2}
		\Bigl(
		|\log\lambda|
		+
		|y-\theta x|^{-1/2}
		\Bigr).
		\label{log-profile-bound}
	\end{align}
For $\ell\le2$, the square of
$v(x)|x|^2\langle x\rangle^{\ell+1/2}$ is bounded by
$C\langle x\rangle^{-2-\varepsilon}$ for some $\varepsilon>0$.
Splitting the integral into $|x|\le|z|/2$ and its complement gives
\[
\int_{\mathbb R^2}\langle x\rangle^{-2-\varepsilon}|z-x|^{-1}\,dx
\lesssim\langle z\rangle^{-1}.
\]
Taking $z=\theta^{-1}y$ and using
$|y-\theta x|^{-1}=\theta^{-1}|\theta^{-1}y-x|^{-1}$ gives
	\begin{equation}\label{weighted-singular-bound}
		\left\|
		v(x)|x|^2
		\langle x\rangle^{\ell+1/2}
		|y-\theta x|^{-1/2}
		\right\|_{L^2_x}
		\lesssim
		\theta^{-1/2}
		\left\langle\theta^{-1}y\right\rangle^{-1/2},
		\qquad \theta\in (0,1].
	\end{equation}
	Combining with  \eqref{def-J}--\eqref{weighted-singular-bound}, we obtain
	\begin{align*}
		&
		\left\|
		\partial_\lambda^\ell
		\Bigl(
		e^{\mp i\lambda|y|}
		\mathcal{J}_{\alpha,\pm}(\lambda,\cdot,y)
		\Bigr)
		\right\|_{L^2}
	\lesssim
		\lambda^{-\ell}
		|\log\lambda|
		\langle\lambda y\rangle^{-1/2}
		\int_0^1
		(1-\theta)\theta^{-1/2}
		\,d\theta
	\lesssim
		\lambda^{-\ell}
		|\log\lambda|
		\langle\lambda y\rangle^{-1/2}.
	\end{align*}
	This proves part~{\rm (ii)}.
	
	\medskip
	
	\textit{Proof of part~{\rm (iii)}.}
	Define
	$
	\widetilde{G}_j(r)
	:=
	b_0\chi_j(r)r^2\log r
$ for $	 j=1,2.$
	Note that
	$$
	\lambda^{-2}
	\Bigl(
	\widetilde{G}_1(\lambda r)
	+
	\widetilde{G}_2(\lambda r)
	\Bigr)
	=
	b_0r^2\log r
	+
	b_0(\log\lambda)r^2.
	$$
	Therefore,
	\begin{align}
		R_0^\pm(\lambda^4)(x,y)
		={}&
		b_0G_2^0(x,y)
		+
		b_0(\log\lambda)G_2(x,y)
	-
		\lambda^{-2}
		\widetilde{G}_2(\lambda|x-y|)
		+
		\mathcal{R}_\pm(\lambda,x,y),
		\label{basic-four-term-decomposition}
	\end{align}
	where
	$
		\mathcal{R}_\pm(\lambda,x,y)
		:=
		R_0^\pm(\lambda^4)(x,y)
		-
		\lambda^{-2}
		\widetilde{G}_1(\lambda|x-y|).
$
	We estimate the four terms in
\eqref{basic-four-term-decomposition} separately.
	
	\medskip
	
	\textit{Step 1: Deal with $(b_0QvG_2^0)(x,y)$.}
We decompose $b_0G_2^0(x,y)$ as
	$
	b_0G_2^0(x,y)
	=
	\widetilde{G}_1(|x-y|)
	+
	\widetilde{G}_2(|x-y|).
	$
Note that  the compactly supported term $\widetilde{G}_1$ is uniformly bounded,  trivially contributing ${O}(1)$ to the $L^2_x$ norm. Then  only the
	term
	$
	Q
	\bigl(
	v\widetilde{G}_2(|\cdot-y|)
	\bigr)
	$
	requires consideration.
	
First, for
	$
	Q
	\in
	\{Q_\alpha:2\leq\alpha\leq4\}
	\cup
	\{Q_4^0\},
	$
	the cancellations
	$
	Qv=Q(x_jv)=0
	$ ($j=1,2$)
	and Lemma \ref{taylor}(ii) give
	\begin{align}
		&
		Q
		\bigl(
		v\widetilde{G}_2(|\cdot-y|)
		\bigr)(x)
	=
		Q
		\Bigl(
		|\cdot|^2v
		\int_0^1
		(1-\theta)
		\mathcal{L}_2[\widetilde{G}_2]
		\bigl(
		|y-\theta\cdot|,
		\xi
		\bigr)
		\,d\theta
		\Bigr)(x).
		\label{QvG20-1}
	\end{align}
	 Moreover, the second-order derivative produces a logarithmic growth:
	$$
	\left|\mathcal{L}_2[\widetilde{G}_2](|\theta x-y|,\xi)\right| \lesssim 1 + |\log(|\theta x-y|)| \chi_2(|\theta x-y|) \lesssim \log(2+|x|)\log(2+|y|).
	$$
	Taking the $L^2_x$ norm against $|x|^2v(x)$ gives \begin{align}\label{static-log-bound}
		\|(Q v \widetilde{G}_2)(\cdot,y)\|_{L^2} \lesssim \log(2+|y|), \quad \text{for } 	Q
		\in
		\{Q_\alpha:2\leq\alpha\leq4\}
		\cup
		\{Q_4^0\}.\end{align}

	Next, for
	$
	Q
	\in
	\{Q_4^1,Q_{41}^1,Q_{42}^1,Q_5,Q_6\},
	$
	in addition to the zeroth- and first-order cancellations, these
	projections satisfy
	$
	Q\bigl(|x|^2v\bigr)=0.
	$ Consequently,
	applying Lemma \ref{taylor}(iii),
	we obtain
	\begin{align}
		Q
		\bigl(
		v\widetilde{G}_2(|\cdot-y|)
		\bigr)(x)
=&
		\frac{1}{2}
		\Bigl(
		\widetilde{G}_2''(|y|)
		-
		\frac{\widetilde{G}_2'(|y|)}{|y|}
		\Bigr)
		Q
		\Bigl(
		v\langle\rho(y),\cdot\rangle^2
		\Bigr)(x)
	\nonumber\\
	&	-
		\frac{1}{2}
		Q
		\Bigl(
		|\cdot|^3v
		\int_0^1
		(1-\theta)^2
		\mathcal{L}_3[\widetilde{G}_2]
		\bigl(
		|y-\theta\cdot|,
		\xi
		\bigr)
		\,d\theta
		\Bigr)(x).
		\label{QvG20-2}
	\end{align}
	Since
	$
	\big|
	\widetilde{G}_2''(r)
	-
	r^{-1}\widetilde{G}_2'(r)
	\big|
	\lesssim1,
	$
	and
	$
	\big|
	\mathcal{L}_3[\widetilde{G}_2](r,\xi)
	\big|
	\lesssim
	\langle r\rangle^{-1},
	$
	it follows that
	\begin{equation}\label{static-uniform-bound}
		\Big\|
		Q
		\bigl(
		v\widetilde{G}_2(|\cdot-y|)	\bigr)(x)
		\Big\|_{L_x^2}
		\lesssim1,\quad \text{for}	\ Q
		\in
		\{Q_4^1,Q_{41}^1,Q_{42}^1,Q_5,Q_6\}.
	\end{equation}
	Together with the bounded contribution of $\widetilde{G}_1$,
	\eqref{static-log-bound} and \eqref{static-uniform-bound} prove
	\eqref{lemma_projection_G}.
	
	\medskip
	
	\textit{Step 2: Deal with $-\lambda^{-2}Q(v\widetilde G_2(\lambda|\cdot-y|))$.}
Applying the same Taylor expansions as in
	\eqref{QvG20-1} and \eqref{QvG20-2}, now to
	$\widetilde{G}_2(\lambda|\cdot-y|)$, gives
	\begin{equation*}
	\left\|
		\partial_\lambda^\ell\mathcal{T}(\lambda,x,y)
		\right\|_{L_x^2}=	\left\|
		\partial_\lambda^\ell\left(
		-\lambda^{-2}
	Q
	\bigl(
	v\widetilde{G}_2(\lambda|\cdot-y|)
	\bigr)(x)\right)
		\right\|_{L_x^2}
		\lesssim
		\begin{cases}
			\lambda^{-\ell}\log(2+|y|),
			&
			Q\in
			\{Q_\alpha:2\leq\alpha\leq4\}
			\cup
			\{Q_4^0\},
			\\
			\lambda^{-\ell},
			&
			Q\in
			\{Q_4^1,Q_{41}^1\}.
		\end{cases}
	\end{equation*}
	Here $\mathcal{T}(\lambda, x,y ):=-\lambda^{-2}\bigl(Q \bigl(v\widetilde{G}_2(\lambda|\cdot-y|)\bigr)\bigr)(x)$ with $(\mathcal{T}, Q)\in\{ (\mathcal{T}_\alpha, Q_\alpha)|_{2\le\alpha\le4}, (\mathcal{T}_4^0, Q_4^0), (\mathcal{T}_4^1, Q_4^1), (\mathcal{T}_4^{1,1}, Q_{41}^1) \}.$

For $Q\in\{Q_{42}^1,Q_5,Q_6\}$, set $m(Q)=1,1,2$, respectively.
The projections $Q_{42}^1$ and $Q_5$ annihilate all quadratic
	moments, whereas $Q_6$ also annihilates all cubic moments.
	Accordingly, Lemma \ref{taylor}(iii)--(iv) gives
\[
\begin{aligned}
-\lambda^{-2}Q\bigl(v\widetilde G_2(\lambda|\cdot-y|)\bigr)
&=\frac{(-1)^{m(Q)+1}\lambda^{m(Q)}}{(m(Q)+1)!}
Q\left(|\cdot|^{m(Q)+2}v\int_0^1(1-\theta)^{m(Q)+1}\right.\\
&\hspace{22mm}\left.{}
\times\mathcal L_{m(Q)+2}[\widetilde G_2]
(\lambda|y-\theta\cdot|,\xi)\,d\theta\right)
=:\lambda^{m(Q)}\mathcal{T}(\lambda,x,y).
\end{aligned}
\]
Here $(\mathcal{T},Q)\in\{(\mathcal{T}_4^{1,2}, Q_{42}^1), (\mathcal{T}_5, Q_5), (\mathcal{T}_6, Q_6)\}.$
The bound $|\partial_\lambda^\ell \mathcal{L}_{m(Q)+2}[\widetilde{G}_2](\lambda| y- \theta \cdot|, \xi)| \lesssim  \lambda^{-\ell}$  implies
	$$
	\left\|
	\partial_\lambda^\ell
	\mathcal{T}(\lambda,\cdot,y)
	\right\|_{L^2}
	\lesssim
	\lambda^{-\ell},
	\qquad
	\mathcal{T}
	\in
	\{\mathcal{T}_4^{1,2},\mathcal{T}_5,\mathcal{T}_6\}.
	$$
	Hence we obtain
	\eqref{lemma_projection_T}.

\textit{Step 3: Deal with $	b_0(\log\lambda)(QvG_2)(x,y)$.}
The cancellations of orders zero and one give
$(QvG_2)(x,y)=Q(|\cdot|^2v)(x)$. This term vanishes for
$Q\in\{Q_4^1,Q_{41}^1,Q_{42}^1,Q_5,Q_6\}$ and gives the stated
$b_0(\log\lambda)Q(|\cdot|^2v)(x)$ term for the other projections.

\textit{Step 4: Deal with $Qv	\mathcal{R}_\pm(\lambda,x,y)$.}
	Define
	$
	F_1^\pm(p)
	:=
	F_\pm(p)-\widetilde{G}_1(p).
	$
	Then
	$
	\mathcal{R}_\pm(\lambda,x,y)
	=
	\lambda^{-2}
	F_1^\pm(\lambda|x-y|)
	$ and
		\begin{align*}
		F_1^\pm(p) &= \left(a_0^\pm  + a_1^\pm p^2 + a_2^\pm p^4 + O(p^5)\right)\chi_1(p) + \Big(\pm \frac{i}{8} e^{\pm ip} w_\pm(p) - \frac{i}{8} e^{-p} w_+(ip)\Big)\chi_2(p).
	\end{align*}
	
	For
	$
	Q
	\in
	\{Q_\alpha:2\leq\alpha\leq4\}
	\cup
	\{Q_4^0,Q_4^1,Q_{41}^1\},
	$
	Lemma \ref{taylor}(ii) gives
	\[
		Q
		\bigl(
		v\mathcal{R}_\pm(\lambda,\cdot,y)
		\bigr)(x)
	=
		Q
		\Bigl(
		|\cdot|^2v
		\int_0^1
		(1-\theta)
		\mathcal{L}_2[F_1^\pm]
		\bigl(
		\lambda|y-\theta\cdot|,
		\xi
		\bigr)
		\,d\theta
		\Bigr)(x).
	\]
	These kernels $Q
	\bigl(
	v\mathcal{R}_\pm(\lambda,\cdot,y)
	\bigr)(x)$  are denoted by
	$
	\mathcal{T}_{\alpha,\pm},
	\mathcal{T}_{4,\pm}^0,
	\mathcal{T}_{4,\pm}^1,
	\mathcal{T}_{4,\pm}^{1,1},
	$
	respectively. The estimate \eqref{lemma_projection_1} follows
	 from \eqref{eq:general_oscillatory_bound}.
	
	It remains to consider
	$
	Q\in\{Q_{42}^1,Q_5,Q_6\}.
	$
	Set
	$
	F_2^\pm(p)
	:=
	F_1^\pm(p)-a_1^\pm p^2,
	$ then $$
	(F_2^\pm)'(0)
	=
	(F_2^\pm)''(0)
	=
	(F_2^\pm)^{(3)}(0)
	=
	0.
	$$
	Moreover,  these projections annihilate all quadratic moments yielding
	$
	Q
	\bigl(
	vF_1^\pm(\lambda|\cdot-y|)
	\bigr)
	=
	Q
	\bigl(
	vF_2^\pm(\lambda|\cdot-y|)
	\bigr).
	$
	For $Q=Q_{42}^1,Q_5$, set $m(Q)=1$, while for $Q=Q_6$, set
	$m(Q)=2$. Lemma \ref{taylor}(iii)--(iv) then yields
	\begin{align*}
		&
		Q
		\bigl(
		v\mathcal{R}_\pm(\lambda,\cdot,y)
		\bigr)(x)
	=
		\frac{(-1)^{m(Q)}\lambda^{m(Q)}}{(m(Q)+1)!}
		Q
		\Bigl(
		|\cdot|^{m(Q)+2}v
		\int_0^1
		(1-\theta)^{m(Q)+1}
		\mathcal{L}_{m(Q)+2}[F_2^\pm]
		\bigl(
		\lambda|y-\theta\cdot|,
		\xi
		\bigr)
		\,d\theta
		\Bigr)(x).
	\end{align*}
	We denote the corresponding kernels by
	$
	\lambda\mathcal{T}_{4,\pm}^{1,2},
	\lambda\mathcal{T}_{5,\pm},
	\lambda^2\mathcal{T}_{6,\pm},
	$
	respectively. Applying \eqref{eq:general_oscillatory_bound} to
$\mathcal L_{m(Q)+2}[F_2^\pm]$ gives \eqref{lemma_projection_1}
for $\mathcal T_{4,\pm}^{1,2}$, $\mathcal T_{5,\pm}$, and
$\mathcal T_{6,\pm}$.

	Combining Steps 1--4 with the decomposition
	\eqref{basic-four-term-decomposition} proves all the expansions and
	estimates in part~{\rm (iii)}, and completes the proof.
\end{proof}

\subsection{Oscillatory integral estimates}

\begin{proof}[Proof of Lemma~\ref{oscillatory}]
Assume $t > 0$ without loss of generality and abbreviate $\mathcal{E}(\lambda) := \mathcal{E}^\pm(\lambda, x,y)$ and $h := h(x,y)$.
	By Euler's identities, the estimate for $\mathcal K^\pm$ reduces to
	estimating integrals of the form
\[
K=\int_0^\infty e^{i\Phi(\lambda)}\lambda\mathcal E(\lambda)\,d\lambda,
\qquad \Phi(\lambda)=\alpha t\lambda^2+\beta r\lambda,\quad \alpha,\beta\in\{-1,1\}.
\]
The same reduction applies to $\mathcal N^\pm$.
Specifically, the integrals are first truncated by a smooth cutoff:
$
	1
	=
	\chi_1(t\lambda^2)
	+
	\chi_2(t\lambda^2).
	$
For the sine integral,  use
$t^{-1}|\sin(t\lambda^2)|\lambda^{-1}\le\lambda$ on $t\lambda^2\lesssim1$.
On the support of $\chi_2(t\lambda^2)$, the multiplier
$t^{-1}\lambda^{-2}\chi_2(t\lambda^2)$ has derivatives bounded by
$C_\ell\lambda^{-\ell}$ for $\ell\le2$. It thus suffices to estimate
	$K$.

	We decompose
	$
	K(t,x,y)
	=
	K_{\mathrm{lo}}(t,x,y)
	+
	K_{\mathrm{hi}}(t,x,y)
	$
by
	$
	1
	=
	\chi_1(t\lambda^2)
	+
	\chi_2(t\lambda^2).
	$
Note that	$
	\Phi'(\lambda)
	=
	2\alpha t\lambda+\beta r,
$
then	a positive stationary point can occur only
 when \begin{align}\label{critial}
 \Phi'(\rho)
 =0
 \  \text{ with}
	\
	\rho
	=
	\frac{|r|}{2t}>0.
	\end{align}
Choose $\eta\in C_c^\infty((1/4,4))$ equal to one on $[1/2,2]$.
	If there exists $\rho$ satisfying the condition \eqref{critial}, define
	$
	\eta_\rho(\lambda)
	=
	\eta\left({\lambda}/{\rho}\right);
	$
	otherwise, set $\eta_\rho=0$.
	 We then decompose the high-frequency 
	$
	K_{\mathrm{hi}}
	=
	K_{\mathrm{stat}}
	+
	K_{\mathrm{non}}
	$
	where
	$$
	K_{\mathrm{stat}}
	=
	\int_0^\infty
	e^{i\Phi(\lambda)}
	\eta_\rho(\lambda)
	\chi_2(t\lambda^2)
	\lambda\mathcal E(\lambda)d\lambda
	,\quad
	K_{\mathrm{non}}
	=
	\int_0^\infty
	e^{i\Phi(\lambda)}
	\bigl(1-\eta_\rho(\lambda)\bigr)
	\chi_2(t\lambda^2)
	\lambda\mathcal E(\lambda)d\lambda.
	$$
	On the support of $\eta_\rho$, one has
	$
	\lambda\sim\rho,
	$
	whereas on the support of the non-stationary amplitude,
	$
	|\Phi'(\lambda)|
	\gtrsim
	t\lambda.
	$

Suppose that
$|\partial_\lambda^\ell\mathcal E|
\lesssim h\lambda^{-\ell}\langle\lambda r\rangle^{-1/2}w(\lambda)$,
where $w=1$ in (i) and
$w=\lambda^{-\sigma}|\log\lambda|^{-\nu}$ in (ii).
Since $|\Phi''|=2t$, van der Corput's lemma gives
\begin{equation}\label{van-1}
|K_{\rm stat}|\lesssim
h t^{-1/2}\rho\langle\rho r\rangle^{-1/2}w(\rho)
\lesssim h t^{-1}w(\rho).
\end{equation}

For the nonstationary part, use a dyadic partition:
$\Lambda_j=2^jt^{-1/2}\gtrsim t^{-1/2}$. 	
Set 
$$
a_j(\lambda)
	:=
	\psi_j(\lambda)
	\bigl(1-\eta_\rho(\lambda)\bigr)
	\chi_2(t\lambda^2)
	\lambda\mathcal E(\lambda),
\qquad 	\psi_j(\lambda)
	:=
	\psi \big(\frac{\lambda}{\Lambda_j}\big).
$$ Its derivatives satisfy
$|\partial_\lambda^\ell a_j|
\lesssim h\Lambda_j^{1-\ell}
\langle\Lambda_jr\rangle^{-1/2}w(\Lambda_j)$ for $\ell\le2$.
Two integrations by parts give
\begin{equation}\label{Ij}
K_j=\int e^{i\Phi}\mathcal L^2a_j\,d\lambda,
\qquad\mathcal Lf=\partial_\lambda\bigl(f/(i\Phi')\bigr),
\qquad
|K_j|\lesssim h t^{-2}\Lambda_j^{-2}w(\Lambda_j).
\end{equation}
Here we used $|\Phi''|=2t$ and the support length $O(\Lambda_j)$.

For (i), the low part is bounded by
$Ch\int_0^{Ct^{-1/2}}\lambda\,d\lambda\lesssim ht^{-1}$.
Equation~\eqref{van-1} gives the same bound for the stationary part,
and summing \eqref{Ij} over all $j$ satisfying $\Lambda_j\gtrsim t^{-1/2}$ gives
$ht^{-2}\sum_j\Lambda_j^{-2}\lesssim ht^{-1}$. 

For (ii), the compact
	support of $\widetilde{\chi}_1(\lambda/2)$ and assumption
	\eqref{condit 2} give $|K|\lesssim h$ for
$t\lesssim1$, because
$\int_0^{1/2}\lambda^{1-\sigma}|\log\lambda|^{-\nu}d\lambda$
is finite in the stated ranges. For $t\gg1$, substitution
$u=-\log\lambda$ gives
\begin{equation}\label{low-fre}
|K_{\rm lo}|\lesssim
\begin{cases}
h t^{-1+\sigma/2}(\log t)^{-\nu},&\sigma<2,\\
h(\log t)^{1-\nu},&\sigma=2,\ \nu>1.
\end{cases}
\end{equation}
	For the high-frequency part, note that under the given parameter conditions ($0 < \sigma < 2$, $\nu \in \mathbb{R}$; or $\sigma = 0$, $\nu \le 0$; or $\sigma = 2$, $\nu > 1$),
	the map $\lambda \mapsto \lambda^{-\sigma}|\log \lambda|^{-\nu}$ is decreasing  on the sufficiently small support of
	$\widetilde{\chi}_1(\lambda/2)$.  Hence, whenever
	$\rho\gtrsim t^{-1/2}$, it follows that
$w(\rho)\lesssim t^{\sigma/2}(\log t)^{-\nu}$.
Therefore
\begin{equation}\label{Ksat}
|K_{\rm stat}|\lesssim h t^{-1+\sigma/2}(\log t)^{-\nu}.
\end{equation}
Only scales $t^{-1/2}\lesssim\Lambda_j\le \lambda_*$ enter the
nonstationary part, where $\lambda_*:= 2(2\lambda_0)^{1/4}\in(0,1)$ with $0<\lambda_0\ll1$.  For every $\sigma>-2$ and $\nu\in\mathbb R$,
\begin{equation}\label{est-b}
\sum_j\Lambda_j^{-\sigma-2}|\log\Lambda_j|^{-\nu}
\lesssim t^{1+\sigma/2}(\log t)^{-\nu}.
\end{equation}
Indeed, choose the positive integer $J$ with $2^Jt^{-1/2}\sim \lambda_*$. On
$0<j\le J/2$, $|\log\Lambda_j|\sim \log t$ and the geometric sum
converges since $\sigma+2>0$. On $j>J/2$, 	we have $
0<	|\log \lambda_*|\leq	|\log\Lambda_j|
\lesssim\log t.$ 
After removing the factor
$t^{1+\sigma/2}$, the sum is bounded by
$Ct^{-(\sigma+2)/4}(\log t)^{\max\{-\nu,0\}}$, which is
$O((\log t)^{-\nu})$. Equations~\eqref{Ij} and \eqref{est-b} give
\begin{equation}\label{Knon}
|K_{\rm non}|\lesssim h t^{-1+\sigma/2}(\log t)^{-\nu}.
\end{equation}
Combining \eqref{low-fre}--\eqref{Knon} proves (ii). If $r=0$,
there is no stationary part, so only $-2<\sigma<2$ is needed for
the estimate, with arbitrary $\nu$.
\end{proof}
\begin{remark}\label{remark:osci}
{\rm The proof also gives the same bounds for the exponential integral
with $\Phi(\lambda)=at\lambda^2+br\lambda$. In particular, under
\eqref{condit 1},
\begin{equation}\label{rem-1}
\left|\int_0^\infty e^{i\Phi(\lambda)}\lambda
\mathcal E(\lambda,x,y)\,d\lambda\right|
\lesssim h(x,y)|t|^{-1}.
\end{equation}
Under \eqref{condit 2} with $\sigma=2$, $\nu>1$, and $|t|\ge2$,
the regions $|t|\lambda^2\lesssim1$ and $|t|\lambda^2\gtrsim1$
contribute respectively $O(h(\log|t|)^{1-\nu})$ and
$O(h(\log|t|)^{-\nu})$, by \eqref{low-fre}--\eqref{Knon}.}
\end{remark}

\begin{lemma}\label{lemma:I_J}    
	For $|t| \ge 2$, the following asymptotic expansions hold:    
	\[
    \begin{aligned}    
		I(t) &:= \int_0^\infty \cos(t\lambda^2) \lambda \widetilde{\chi}_1(\lambda) (\log \lambda)^2 \, d\lambda = \frac{\pi}{8} \frac{\log |t|}{|t|} + O\big(|t|^{-1}\big),  \\    
		J(t) &:= t^{-1} \int_0^\infty \sin(t\lambda^2) \lambda^{-1} \widetilde{\chi}_1(\lambda) (\log \lambda)^2 \, d\lambda = \frac{\pi}{16} \frac{(\log |t|)^2}{|t|} + O\Big(\frac{\log |t|}{|t|}\Big).    
	\end{aligned}
    \]
\end{lemma}    
\begin{proof}
The asserted expansions follow from the change of variables
$
u=|t|\lambda^2,
$
the identity
$
(\log\lambda)^2
=
\frac14(\log u-\log|t|)^2,
$
and repeated integrations by parts. We omit the
routine details.
\end{proof}

\subsection* {Acknowledgements} The authors are partially supported by NSFC (grants No. 12171182 and 12531005). The authors thank Professor Avy Soffer for his interest and helpful discussions.

\subsection*{Data Availability}
Data sharing is not applicable to this article as no new data were created or analyzed in this study.
\subsection*{Competing Interests} The authors have no relevant financial or non-financial interests to disclose.

 \end{document}